\documentclass[a4paper,11pt]{amsart}
\usepackage{amssymb,amsmath,amscd,tikz,url}
\usepackage[all]{xy}
\usepackage{hyperref}

\newtheorem{thm}{Theorem}[section]
\newtheorem{conjecture}[thm]{Conjecture}
\newtheorem{cor}[thm]{Corollary}
\newtheorem{lem}[thm]{Lemma}
\newtheorem{prop}[thm]{Proposition}

\newtheorem{rem}[thm]{Remark}
\newtheorem{notation}[thm]{Notation}
\newtheorem{ex}[thm]{Example}

\newtheorem{comment}[thm]{Comment}

\newcommand{\boldq}{{u}}

\newcommand{\Zset}{\mathbb{Z}}
\newcommand{\Rset}{\mathbb{R}}
\newcommand{\CC}{\mathbb{C}}
\newcommand{\Cset}{\mathbb{C}}

\def\ss{{\rm ss}}
\def\top{{\rm top}}

\newcommand{\Nset}{\mathbb{N}}

\newcommand{\ab}{{\rm ab}}

\def\fo{{\mathfrak o}}
\def\fS{{\mathfrak S}}

\def\Stab{{\rm Stab}}
\def\PPhi{{\Phi}}

\def\rG{{{\rm G}}}

\def\cD{{\mathcal D}}
\def\cK{{\mathcal K}}

\def\Cent{{\rm Z}}
\def\Z{{\rm Z}}
\def\Nor{{\rm N}}

\def\rec{{\rm rec}}
\def\ie{{\it i.e.,\,}}

\def\cI{{\mathcal I}}
\def\cJ{{\mathcal J}}

\def\im{{\rm im\,}}

\def\identity{{\rm Id}}

\def\H{{\rm H}}
\def\HP{{\rm HP}}
\def\KLR{{\rm KLR}}
\def\SL{{\rm SL}}

\def\SO{{\rm SO}}

\def\GL{{\rm GL}}

\def\der{{\rm der}}
\def\Aut{{\rm Aut}}
\def\Hom{{\rm Hom}}
\def\End{{\rm End}}
\def\Id{{\rm I}}

\def\Ind{{\rm Ind}}

\def\top{{\rm top}}
\def\triv{{\rm triv}}
\def\unr{{\rm unr}}
\def\JL{{\rm{JL}}}

\def\cEesq{{\rm \bf Irr}_{\rm{ess}L^2}} 
\def\bI{{\mathbf I}}
\def\bW{{\mathbf W}}
\def\WD{{\bW_F\times\SL_2(\Cset)}}

\def\spe{{R}}
\def\Morita{{\rm Morita}}

\def\der{{\rm der}}

\def\Mat{{\rm M}}

\def\Irr{{\mathbf {Irr}}}

\def\Prim{{\rm Prim}}
\def\cInd{{\rm c\!\!-\!\!Ind}}

\def\tr{{\rm tr}}

\def\cG{{\mathcal G}}

\def\cH{{\mathcal H}}

\def\cM{{\mathcal M}}
\def\cO{{\mathcal O}}

\def\cS{{\mathcal S}}

\def\cT{{\mathcal T}}

\def\cV{{\mathcal V}}
\def\cW{{\mathcal W}}

\def\fs{{\mathfrak s}}

\def\fB{{\mathfrak B}}

\def\fR{{\mathfrak R}}

\def\fA{{\mathfrak A}}
\def\fB{{\mathfrak B}}
\def\fR{{\mathfrak R}}

\def\Ad{{\rm Ad}}
\def\triv{{\rm triv}}
\def\Sp{{\rm Sp}}

\def\tW{{\widetilde W}}

\def\q{{/\!/}}

\def\ab{{\rm ab}}
\def\der{{\rm der}}

\def\PPhi_F{{\rm Frob}}
\def\Frob{{\rm Frob}}
\def\Flag{{\rm Flag}}
\def\pt{{\rm pt}}
\def\cpt{{\rm cpt}}
\def\square{{\natural}}
\newcommand{\nnu}{\varphi}

\newcommand{\Q}{\mathbb Q}

\newcommand{\C}{\mathbb C}
\newcommand{\matje}[4]{\left(\begin{smallmatrix} #1 & #2 \\ 
#3 & #4 \end{smallmatrix}\right)}
\def\lexp#1#2{{\kern\scriptspace\vphantom{#2}^{#1}\kern-\scriptspace#2}}

\begin{document}
\title[Geometric structure]{Geometric structure and the local Langlands conjecture}   


\author[A.-M. Aubert]{Anne-Marie Aubert}
\address{Institut de Math\'ematiques de Jussieu, U.M.R. 7586 du C.N.R.S., U.P.M.C., Paris, France}
\email{aubert@math.jussieu.fr}
\author[P. Baum]{Paul Baum}
\address{Mathematics Department, Pennsylvania State University,  University Park, PA 16802, USA}
\email{baum@math.psu.edu}
\thanks{The second author was partially supported by NSF grant DMS-0701184}
\author[R. Plymen]{Roger Plymen}
\address{School of Mathematics, Southampton University, Southampton SO17 1BJ,  England \emph{and} School of Mathematics, Manchester University,
Manchester M13 9PL, England}
\email{r.j.plymen@soton.ac.uk \quad plymen@manchester.ac.uk}
\author[M. Solleveld]{Maarten Solleveld}
\address{Radboud Universiteit Nijmegen, Heyendaalseweg 135, 6525AJ Nijmegen, the Netherlands}
\email{m.solleveld@science.ru.nl}

\date{\today}
\maketitle

\begin{abstract}   
We prove that a strengthened form of the local Langlands conjecture is valid throughout the principal
series of any connected split reductive $p$-adic group. The method of proof is to establish the presence  
of a very simple geometric structure, in both the smooth dual and the Langlands parameters.  
We  prove that this geometric structure is present, in the same way, for the general linear group, 
including all of its inner forms.
With these results as evidence, we give a detailed formulation of a general geometric structure conjecture.
\end{abstract}

\tableofcontents

\section{Introduction}
Let $\mathcal{G}$ be a connected reductive $p$-adic group. The smooth dual of $\mathcal{G}$ 
--- denoted $\Irr (\cG)$ --- is the set of equivalence classes of smooth irreducible 
representations of $\mathcal{G}$. In this paper we state a conjecture
based on \cite{ABP1,ABP2,ABP3,ABP4} which asserts that a
very simple geometric structure is present in $\Irr(\mathcal{G})$. A first feature of our 
conjecture is that it provides a guide to determining $\Irr(\mathcal{G})$.
A second feature of the conjecture is that it connects very closely to the
local Langlands conjecture.

For any connected reductive $p$-adic group $\mathcal{G}$, validity of 
the conjecture gives an explicit description of Bernstein's infinitesimal 
character \cite{BeDe} and of the intersections of L-packets
with Bernstein components in $\Irr(\mathcal{G})$.

The conjecture can be stated at four levels:
\begin{itemize}
\item $K$-theory of $C^*$-algebras
\item Periodic cyclic homology of finite type algebras
\item Geometric equivalence of finite type algebras
\item Representation theory
\end{itemize}
At the level of $K$-theory, the conjecture interacts with the Baum--Connes 
conjecture \cite{BCH}. 
BC has been proved for reductive p-adic groups by V. Lafforgue \cite{L}. 
The conjecture of this paper ABPS can be viewed as a ``lifting"
of BC from $K$-theory to representation theory. In this paper ABPS will be stated at the level 
of representation theory. 

The overall point of view of the paper is as follows. Denote the $L$-group of the 
$p$-adic group $\mathcal{G}$ by $^{L}\mathcal{G}$. 
If $\cG$ is split, then $^{L}\mathcal{G}$  is a connected reductive complex algebraic group. 
A Langlands parameter is a homomorphism of topological groups
\[
\bW_F\times \SL_2 (\mathbb{C})\longrightarrow  ^{L}\!\mathcal{G} 
\]
which is required to satisfy some conditions. Here $\bW_F$ is the Weil group of the $p$-adic 
field $F$ over which $\mathcal{G}$ is defined.  Let $G$ denote the complex dual group of $\cG$, and let 
$\{\mathrm{Langlands\:\: parameters}\}/G$ be the set of all the Langlands parameters for 
$\mathcal{G}$ modulo the action of $G$.
The local Langlands correspondence asserts that there is a surjective finite-to-one map
\[
\Irr(\cG)\longrightarrow \{\mathrm{Langlands\:\: parameters}\}/G ,
\]
which is natural in various ways. A more subtle version \cite{Lusztig1983, KL, ArNote,Vog} conjectures that 
one can naturally enhance the Langlands parameters with irreducible representations of 
certain finite groups, such that the map becomes bijective.

The aim of the paper is to introduce into this context a countable disjoint union of complex 
affine varieties, denoted \{Extended quotients\}, such that there is
a commutative triangle of maps
\[ 
\xymatrix{   
& \{\text{Extended quotients}\} \ar[dr] \ar[dl] & \\ 
\Irr(\mathcal{G}) \ar[rr] & & 
\{\text{Langlands parameters}\}/G \ }
\] 
in which the left slanted arrow is bijective and the other two arrows are surjective and 
finite to one. 

The key point is that 
in practice \{Extended quotients\} is much more easily calculated than either 
$\Irr(\mathcal{G})$ or the $G$-conjugacy classes of Langlands parameters.
In examples, bijectivity of the left slanted map 
is proved by using results on the representation theory of affine Hecke algebras. 
The right slanted map is defined and studied by using an appropriate generalization of 
the Springer correspondence.\\ 

The paper is divided into three parts:
\begin{itemize}
\item  Part 1: Statement of the conjecture
\item  Part 2: Examples
\item  Part 3: Principal series of connected split reductive $p$-adic groups
\end{itemize}

We should emphasize that the  conjecture in Part 1 of this article is a \emph{strengthening} 
of the geometric conjecture formulated in 
\cite{ABP1, ABP2, ABP3, ABP4}.   In \cite{ABP1, ABP2, ABP3, ABP4} the local Langlands 
correspondence was not part of the conjecture.
Now, by contrast, the local Langlands correspondence is locked into our conjecture. 

Let $\fs$ be a point in the Bernstein spectrum of $\cG$, let $\Irr^\fs (\cG)$ be the part 
of $\Irr (\cG)$ belonging to $\fs$ and let \{Langlands parameters$\}^\fs$ be the set of Langlands 
parameters whose L-packets contain elements of $\Irr^\fs (\cG)$.  
Let $H^{\fs}$ be the stabilizer in $G$ of this set of Langlands parameters.

Furthermore, let $T^{\fs}$ and 
$W^{\fs}$ be the complex torus and finite group assigned by Bernstein to 
$\fs$, and let $(T^{\fs}\q W^{\fs})_2$ is the extended quotient (of the second kind) 
for the action of $W^{\fs}$ on $T^{\fs}$. 
According to our conjecture, the local Langlands correspondence, restricted to the objects
attached to $\fs$, \emph{factors through} the extended quotient $(T^{\fs}\q W^{\fs})_2$. 
In this precise sense, our conjecture reveals a geometric structure latent in the local 
Langlands conjecture.

The essence of our conjecture is:

\begin{conjecture} \label{conj:I1}
Let $\cG$ be a quasi-split connected reductive $p$-adic group or
an inner form of $\GL_n(F)$, and let $\Irr(\cG)^{\fs}$ 
be any Bernstein component in $\Irr(\cG)$.   Then there is a commutative triangle 
\[ 
\xymatrix{   
& (T^{\fs}/\!/W^\fs)_2 \ar[dr]\ar[dl] & \\  
\Irr(\mathcal{G})^\fs \ar[rr] & & \{\mathrm{Langlands}\:\:\mathrm{parameters}\}^\fs/ H^{\fs}}
\]
in which the horizontal arrow is the map of the local Langlands conjecture, and the two 
slanted arrows are canonically defined maps with the left slanted arrow bijective and the
right slanted arrow surjective and finite-to-one.
\end{conjecture}

In Part 2, we give  an  account of the general linear 
group $\GL_n (F)$ and its inner forms
$\GL_m (D)$.  Here,  $D$ is an $F$-division algebra of dimension $d^2$ over its 
centre $F$ and $n = md$.  Except for $\GL_n (F)$, these groups are non-split.
The main result of Part 2 is:

\begin{thm} \label{thm:I2}
Conjecture \ref{conj:I1} is valid for  $\cG = \GL_m (D)$.   
In this case, for each Bernstein component $\Irr(\cG)^{\fs} \subset \Irr(\cG)$, 
all three maps in the commutative triangle are bijective.   
\end{thm}

Calculations involving two  other examples --- $\cG = \Sp_{2n}(F)$ and $\cG = G_2(F)$ --- 
are also given in Part 2.

\bigskip

In Part 3 we assume that the local field $F$ satisfies a mild restriction on its residual 
characteristic, depending on $\mathcal G$. For the principal series we then prove that the 
conjectured geometric structure (i.e. extended quotients) is present in the enhanced \cite{Lusztig1983, KL, ArNote,Vog}  Langlands 
parameters. In this case the enhanced Langlands parameters reduce to 
simpler data, which we call Kazhdan--Lusztig--Reeder parameters. 

More precisely, the main result in Part 3, namely Theorem \ref{thm:main}, is:

\begin{thm} \label{thm:I3}
Let $\cG$ be a connected split reductive $p$-adic group.   
Assume that the residual characteristic of the local field $F$ 
is not a torsion prime for $\cG$.  Let $\Irr(\cG)^{\fs}$ be a Bernstein component in the 
principal series of $\cG$. Then Conjecture \ref{conj:I1} is valid for $\Irr(\cG)^{\fs}$ 
i.e. there is a commutative triangle of natural bijections
\[ 
\xymatrix{   
& (T^{\fs}/\!/W^\fs)_2 \ar[dr]\ar[dl] & \\  
\Irr(\mathcal{G})^\fs    \ar[rr] & & \{\KLR\:\:\mathrm{parameters}\}^\fs/ H^\fs }
\]
In this triangle
$\{\KLR \; \mathrm{parameters }\}^\fs/H^\fs$ is the set of Kazhdan--Lusztig--Reeder 
parameters for the Bernstein component $\Irr(\mathcal{G})^\fs$, modulo conjugation 
by $H^{\fs}$. 
\end{thm} 

The construction of the bottom horizontal map in the triangle generalizes results of 
Reeder \cite{R}. Reeder requires, when the inducing character is ramified, that $\cG$ 
shall have connected centre; we have removed this restriction.  Therefore, our result 
applies to the principal series of $\SL_n(F)$.   

Additional points in Part 3 are ---  labelling by unipotent classes, 
correcting cocharacters, and proof of the L-packet conjecture stated in \cite{ABP4}.

Finally, the appendix defines \emph{geometric equivalence}. This is an equivalence relation on 
finite type algebras which is a weakening of Morita equivalence. Geometric equivalence 
underlies and is the foundation of Conjecture \ref{conj:I1}   

\bigskip

\textbf{Acknowledgements.}
Thanks to Mark Reeder for drawing our attention to the article of Kato \cite{Kat}.   
We thank Joseph Bernstein, David Kazhdan, George Lusztig, and David Vogan for enlightening comments and discussions.

\part{Statement of the conjecture}

\section{Extended quotient}
\label{sec:extquot}
Let $\Gamma$ be a finite group acting on  a complex affine variety $X$ 
as automorphisms of the affine variety  
\[ 
\Gamma \times X \to X.
\]
The quotient variety $X/\Gamma$ is obtained by collapsing each orbit to a point. 

 For $x\in X $, $\Gamma_x$ denotes the stabilizer group of $x$:
$$ 
\Gamma_x = \{\gamma\in \Gamma : \gamma x = x\}.
$$
$c(\Gamma_x)$ denotes the set of conjugacy classes of  $\Gamma_x$. The extended quotient 
is obtained by replacing the orbit of $x$ by $c(\Gamma_x)$. This is done as follows:\\

\noindent Set $\widetilde{X} = \{(\gamma, x) \in \Gamma \times X : \gamma x = x\}$.  
$\widetilde{X}$ is an affine variety and is a subvariety of $\Gamma \times X$. 
The group $\Gamma$ acts on $\widetilde{X}$:
\begin{align*}
\Gamma &\times \widetilde{X} \to \widetilde{X}\\
\alpha(\gamma, x) = & (\alpha\gamma \alpha^{-1}, \alpha x), \quad\quad \alpha \in \Gamma,
\quad (\gamma, x) \in \widetilde{X}.
\end{align*}

\noindent The extended quotient, denoted $ X/\!/\Gamma $,  is $\widetilde{X}/\Gamma$. 
Thus the extended quotient  $ X/\!/\Gamma $ is the usual quotient for the action of 
$\Gamma$ on $\widetilde{X}$. 
The projection $\widetilde{X} \to X$,  $(\gamma, x) \mapsto x$ is $\Gamma$-equivariant 
and so passes to quotient spaces to give a morphism of affine varieties
\[
\rho\colon  X/\!/\Gamma \to X/\Gamma.
\]   
This map will be referred to as the projection of the extended quotient onto the ordinary quotient.\\

\noindent The inclusion 
\begin{align*}
X&\hookrightarrow \widetilde{X}\\
x&\mapsto (e,x)\qquad e=\text{identity element of }\Gamma
\end{align*}
is $\Gamma$-equivariant and so passes to quotient spaces to give an inclusion of affine 
varieties $X/\Gamma\hookrightarrow X/\!/\Gamma$. This will be referred to as the inclusion 
of the ordinary quotient in the extended quotient. We will denote $X/\!/\Gamma$ with 
$X/\Gamma$ removed by $X/\!/\Gamma-X/\Gamma$.

\section{Bernstein spectrum}
\label{sec:Bernstein}
We recall some well-known parts of Bernstein's work on $p$-adic groups,
which can be found for example in \cite{BeDe,Renard}.

With $\mathcal G$ fixed, a cuspidal pair is a pair $(\cM, \sigma)$ where 
$\cM$ is a Levi factor of a parabolic subgroup $\mathcal P$ of $\cG$ 
and $\sigma$ is an irreducible supercuspidal representation of $\cM$. 
Here supercuspidal means that the support of any matrix coefficient of
such a representation is compact modulo the centre of the group. Pairs
$(\cM,\sigma)$ and $(\cM,\sigma')$ with $\sigma$ isomorphic to $\sigma'$ are
considered equal.
The group $\cG$ acts on the space of cuspidal pairs by conjugation:
\[
g \cdot (\cM ,\sigma) = (g \cM g^{-1}, \sigma \circ \mathrm{Ad}_g^{-1} ) .
\]
We denote the space of $\cG$-conjugacy classes by $\Omega (\cG)$. We can inflate 
$\sigma$ to an irreducible smooth $\mathcal P$-representation. Normalized smooth 
induction then produces a $\cG$-representation $I_{\mathcal P}^{\cG} (\sigma)$.

For any irreducible smooth $\cG$-representation $\pi$ there is a cuspidal pair 
$(\cM,\sigma)$, unique up to conjugation, such that $\pi$ is a subquotient of
$I_{\mathcal P}^{\cG} (\sigma)$. (The collection of irreducible subquotients of 
the latter representation does not depend on the choice of $\mathcal P$.) 
The $\cG$-conjugacy class of $(M,\sigma)$ is called the \emph{cuspidal support}
of $\pi$. We write the cuspidal support map as
\[
\mathbf{Sc} : \Irr (\cG) \to \Omega (\cG) .
\]
For any unramified character $\nu$ of $\cM ,\; (\cM,\sigma \otimes \nu)$ is again
a cuspidal pair. Two cuspidal pairs $(\cM,\sigma)$ and $(\cM',\sigma')$ are said
to be inertially equivalent, written $(\cM,\sigma) \sim (\cM',\sigma')$, 
if there exists an unramified character 
$\nu : \cM \to \C^\times$ and an element $g \in \mathcal G$ such that
\[
g \cdot (\mathcal M, \psi \otimes \sigma) = (\mathcal M', \sigma ') .
\]
The Bernstein spectrum of $\mathcal G$, denoted $\mathfrak{B}(\mathcal G)$, 
is the set of inertial equivalence classes of cuspidal pairs. It is a countable set,
infinite unless $\mathcal G = 1$. Let $\fs = [\cM ,\sigma]_\cG \in \mathfrak{B} (\cG)$
be the inertial equivalence class of $(\cM,\sigma)$ and let $\Irr (\cG)^\fs$ be the 
subset of $\Irr (\cG)$ of representations that have cuspidal support in $\fs$. Then 
$\Irr (\mathcal G)$ is the disjoint union of the Bernstein components $\Irr (\cG )^\fs$:
\[
\Irr (\mathcal G) = \bigsqcup_{\mathfrak{s}\in \mathfrak{B}(\mathcal G)} 
\Irr (\mathcal G)^\fs .
\] 
The space $X_{\unr}(\cM)$ of unramified characters of $\cM$ is a natural way a
complex algebraic torus. Put 
\begin{equation} \label{eqn:Psis}
\mathrm{Stab}(\sigma) = \{ \nu \in X_{\unr}(\cM) \mid \sigma \otimes \nu \cong \sigma \} .
\end{equation}
This is known to be a finite group, so $X_{\unr}(M) / \mathrm{Stab}(\sigma)$
is again a complex algebraic torus. The map
\begin{equation} \label{eq:TsIrr}
X_{\unr}(\cM) / \mathrm{Stab}(\sigma) \to \Irr (\cM)^{[\cM,\sigma]_\cM} ,\;
\nu \mapsto \sigma \otimes \nu 
\end{equation}
is bijective and thus provides $\Irr (\cM)^{[\cM,\sigma]_\cM}$ with the structure of 
an algebraic torus. This structure is canonical, in the sense that it does not 
depend on the choice of $\sigma$ in $\Irr (\cM)^{[\cM,\sigma]_\cM}$.

The Weyl group of $(\cG,\cM)$ is defined as
\[
W(\cG ,\cM) := N_\cG (\cM) / \cM .
\]
It is a finite group which generalizes the notion of the Weyl group associated to
a maximal torus. The Weyl group of $(\cG ,\cM)$ acts naturally on $\Irr (\cM)$, via
the conjugation action on $\cM$. The subgroup
\begin{equation}\label{eq:defWs}
W^\fs := \{ w \in W (\cG,\cM) \mid w \text{ stabilizes } [\cM,\sigma ]_\cM \} .
\end{equation}
acts on $\Irr (\cM)^{[\cM,\sigma]_\cM}$. We define
\begin{equation}\label{eq:defTs}
T^\fs := \Irr (\cM)^{[\cM,\sigma]_\cM} 
\end{equation}
with the structure \eqref{eq:TsIrr} as algebraic torus and the $W^\fs$-action \eqref{eq:defWs}.
We note that the $W^\fs$-action is literally by automorphisms of the algebraic variety $T^\fs$, 
via \eqref{eq:TsIrr} they need not become group automorphisms. Two elements of $T^\fs$ are
$\cG$-conjugate if and only if they are in the same $W^\fs$-orbit.

An inertially equivalent cuspidal pair $(\cM',\sigma')$ would yield a torus $T'^\fs$ which
is isomorphic to $T^\fs$ via conjugation in $\cG$. Such an isomorphism $T^\fs \cong T'^\fs$
is unique up to the action of $W^\fs$.

The element of $T^\fs / W^\fs$ associated to  any $\pi \in \Irr (\cG)^\fs$ is called its 
\emph{infinitesimal central character}, denoted $\pi^\fs (\pi)$.
Another result of Bernstein is the existence of a unital finite type 
$\mathcal O (T^\fs / W^\fs)$-algebra $\mathcal H^\fs$, whose irreducible modules 
are in natural bijection with $\Irr (\cG )^\fs$. The construction is such that $\mathcal H^\fs$
has centre $\mathcal O (T^\fs / W^\fs)$ and that $\pi^\fs (\pi)$ is precisely the central 
character of the corresponding $\mathcal H^\fs$-module. 

Since $\Irr (\mathcal H^\fs)$ is
in bijection with the collection of primitive ideals of $\mathcal H^\fs$, we can endow
it with the Jacobson topology. By transferring this topology to $\Irr (\cG )^\fs$, we make 
the latter into a (nonseparated) algebraic variety. (In fact this topology agrees with the
topology on $\Irr (\cG )^\fs$ considered as a subspace of $\Irr (\cG )$, endowed with the
Jacobson topology from the Hecke algebra of $\cG$.) 

\noindent \underline{Summary}: 
For each Bernstein component $\fs \in\mathfrak{B}(\mathcal G)$ there are:
\begin{enumerate}
\item A finite group $W^{\mathfrak{s}}$ acting on a complex torus $T^\fs$;
\item A subset $\Irr (\mathcal G)^\fs$ of $\Irr (\mathcal G)$;
\item A morphism of algebraic varieties 
\[
\pi^\fs \colon \Irr (\mathcal G)^\fs \longrightarrow T^\fs /W^\fs ;
\]
\item A unital finite-type $\mathcal{O}(T^\fs/W^\fs)$-algebra $\mathcal H^\fs$ with
\[
\Irr (\mathcal H^\fs) = \Irr (\mathcal G)^\fs .
\]
\end{enumerate}

\section{Approximate statement of the conjecture}
\label{sec:approx}

As above, $\mathcal G$ is a quasi-split connected reductive $p$-adic group or an inner form of $\GL_n(F)$, and 
$\mathfrak{s}$ is a point in the Bernstein spectrum of $\mathcal G$.

Consider the two maps indicated by vertical arrows:
\begin{displaymath} 
\xymatrix{
T^\fs/\!/W^\fs \ar[dd]_{\rho^\fs}&  & \Irr (\mathcal G)^\fs \ar[dd] ^{\pi^\fs}  \\
&   & \\
T^\fs / W^\fs  & & T^\fs / W^\fs }
\end{displaymath}
Here $\pi^\fs$ is the infinitesimal character and $\rho^\fs$ is the 
projection of the extended quotient on the ordinary quotient.  In practice 
$T^\fs /\!/ W^\fs$ and $\rho^\fs$ 
are much easier to calculate than $\Irr (\mathcal G)^\fs$ and $\pi^\fs$.

An approximate statement of the conjecture is:
\[
\pi^\fs \colon \Irr (\mathcal G)^\fs \longrightarrow T^\fs /W^\fs 
\quad \mathrm{and} \quad
\rho^\fs \colon T^\fs /\!/W^\fs \longrightarrow T^\fs/W^\fs \: \mathrm{are\: almost\: the \:same}.
\]
The precise statement of the conjecture --- in particular, precise meaning of ``are almost the same" 
--- is given in section \ref{sec:statement} below. 

$\pi^\fs$ and $\rho^\fs$ are both surjective finite-to-one maps and morphisms of algebraic varieties.
For $x\in T^\fs / W^\fs$, denote by $\#(x, \rho^\fs) ,\; \#(x, \pi^\fs)$
the number of points in the pre-image of $x$ using $\rho^\fs ,\; \pi^\fs$. The numbers
$\#(x, \pi^\fs)$ are of interest in describing exactly what happens when 
$\Irr (\mathcal G)^\fs$ is constructed by parabolic induction. 

Within $T^\fs / W^\fs$ there are algebraic sub-varieties $R(\rho^\fs) ,\; R(\pi^\fs)$ defined by
\[
R(\rho^\fs) :=\{x\in T^\fs / W^\fs
\thinspace \mid \thinspace \#(x, \rho^\fs) > 1\}
\]
\[
R(\pi^\fs) :=\{x\in T^\fs / W^\fs
\thinspace \mid \thinspace \#(x, \pi^\fs) > 1\}
\]
It is immediate that
\[
R(\rho^\fs) = \rho^\fs (T^\fs /\!/ W^\fs - T^\fs / W^\fs)
\]
$R(\pi^\fs)$ will be referred to as the \emph{sub-variety of reducibility}.

In many examples $R(\rho^\fs) \neq R(\pi^\fs)$. 
Hence in these examples it is impossible to have a bijection 
\[
T^\fs /\!/ W^\fs \longrightarrow \Irr (\mathcal G)^\fs   
\]
with commutativity in the diagram
\begin{displaymath} 
\xymatrix{
T^\fs /\!/ W^\fs \ar[dr]_{\rho^\fs} \ar[rr]  &  &     
\Irr (\mathcal G)^\fs \ar[dl] ^{\pi^\fs}  \\
& T^\fs / W^\fs & }
\end{displaymath}
A more precise statement of the conjecture is that after a simple algebraic correction
(``correcting cocharacters") $\rho^\fs$ becomes isomorphic to $\pi^\fs$. Thus $\rho^\fs$ 
is an easily calculable map which can be algebraically corrected to give $\pi^\fs$.
An implication of this is that within the algebraic variety $T^\fs/W^\fs$ there is a 
flat family of sub-varieties connecting $R(\rho^\fs)$ and $R(\pi^\fs)$.

\section{Extended quotient of the second kind}
\label{sec:extquot2}
With $\Gamma ,\: X ,\: \Gamma_x$ as in Section \ref{sec:extquot} above, let $\Irr (\Gamma_x)$ 
be the set of (equivalence classes of) irreducible representations of $\Gamma_x$. 
The \emph{extended quotient of the second kind}, 
denoted $(X/\!/\Gamma)_2$, is constructed by replacing the orbit of x (for the given action of 
$\Gamma$ on $X$) by $\Irr(\Gamma_x)$. This is done as follows : \\

\noindent Set $\widetilde{X}_2 = \{(x, \tau)\thinspace \big|\thinspace x \in X$ and
$\tau \in \Irr(\Gamma_x)\}$. Then $\Gamma$ acts on $\widetilde{X}_2.$
\begin{align*}
& \Gamma \times \widetilde{X}_2 \to \widetilde{X}_2 , \\
& \gamma(x, \tau) = (\gamma x, \gamma_*\tau) ,
\end{align*}
where $\gamma_*\colon \Irr(\Gamma_x)\rightarrow \Irr(\Gamma_{\gamma x})$.
Now we define
\[
(X/\!/\Gamma)_2  := \:\widetilde{X}_2/\Gamma , 
\]
i.e. $ (X/\!/\Gamma)_2 $ is the usual quotient for the action of $\Gamma$ on $\widetilde{X}_2$.
The projection $\widetilde{X}_2\rightarrow  X$\quad $(x, \tau) \mapsto x$\quad is 
$\Gamma$-equivariant  and so passes to quotient spaces to give the projection of 
$(X/\!/\Gamma)_2$ onto $X/\Gamma$.
\[
\rho_2 \colon (X/\!/\Gamma)_2 \longrightarrow X/\Gamma  
\]
Denote by $\mathrm{triv}_x$ the trivial one-dimensional representation of $\Gamma_{x}$.
\noindent The inclusion 
\begin{align*}
X&\hookrightarrow \widetilde{X}_2\\
x&\mapsto (x, \mathrm{triv}_x)
\end{align*}
is $\Gamma$-equivariant and so passes to quotient spaces to give an inclusion
\[
X/\Gamma\hookrightarrow (X/\!/\Gamma)_2
\] 
This will be referred to as the inclusion of the ordinary quotient in the extended quotient
of the second kind.

\section{Comparison of the two extended quotients}
\label{sec:extquots}
With $X, \Gamma$ as above, there is a non-canonical bijection 
$\epsilon \colon X/\!/\Gamma\rightarrow (X/\!/\Gamma)_2$ with commutativity in the diagrams
\begin{equation}\label{eq:diagramsExtquots} 
\xymatrix{
X/\!/\Gamma \ar[dr]_{\rho} \ar[rr]^{\epsilon}  &  &  (X/\!/\Gamma)_2 \ar[dl]^{\rho_2}  &
X/\!/\Gamma  \ar[rr]^{\epsilon}  &  &  (X/\!/\Gamma)_2  \\
& X/\Gamma & & & X/\Gamma \ar[ul] \ar[ur] & }
\end{equation}
To construct the bijection $\epsilon$, some choices must be made. 
We use a family $\psi$ of bijections
\[
\psi_x : c (\Gamma_x) \to \Irr (\Gamma_x)
\]
such that for all $x \in X$:
\begin{enumerate}
\item $\psi_x ([1]) = \mathrm{triv}_x$;
\item $\psi_{\gamma x} ([\gamma g \gamma^{-1}]) = \phi_x ([g]) \circ \mathrm{Ad}_\gamma^{-1}$
for all $g \in \Gamma_x, \gamma \in \Gamma$;
\item $\psi_x = \psi_y$ if $\Gamma_x = \Gamma_y$ and $x,y$ belong to the same connected 
component of the variety $X^{\Gamma_x}$.
\end{enumerate}
We shall refer to such a family of bijections as a $c$-$\Irr$ system. Clearly $\psi$
induces a map $\widetilde X \to {\widetilde X}_2$ which preserves the $X$-coordinates.
By property (2) this map is $\Gamma$-equivariant, so it descends to a map
\[
\epsilon = \epsilon_\psi : X /\!/ \Gamma \to (X /\!/ \Gamma )_2 .
\]
Observe that $\epsilon_\psi$ makes the diagrams from \eqref{eq:diagramsExtquots} commute,
the first by construction and the second by property (1).
The restriction of $\epsilon_\psi$ to the fiber over $\Gamma x \in X / \Gamma$
is $\psi_x$, and in particular is bijective. Therefore $\epsilon_\psi$ is bijective.
Property (3) is not really needed, it serves to exclude ugly examples.

One way to topologize $(X \q \Gamma )_2$ is via the $X$-coordinate, in other words, by simply
pulling back the topology from $X / \Gamma$ via the natural projection. With respect to this 
naive topology $\epsilon_\psi$ is continuous. This continuous bijection, however, is not usually
a homeomorphism. In most cases $X/\!/\Gamma$ is more separated than $(X/\!/\Gamma)_2.$ 

But there are other useful topologies. Let $\mathcal{O}(X)$ be the coordinate algebra of the 
affine variety $X$ and let $\mathcal{O}(X)\rtimes \Gamma$ be the crossed-product algebra for 
the action of $\Gamma$ on $\mathcal{O}(X)$. There are canonical bijections
\[
\Irr (\mathcal{O}(X)\rtimes \Gamma) \longleftrightarrow 
\mathrm{Prim}(\mathcal{O}(X)\rtimes\Gamma) \longleftrightarrow (X/\!/ \Gamma)_2 ,
\]
where Prim($\mathcal{O}(X)\rtimes \Gamma$) is the set of primitive ideals in this algebra.
The irreducible module associated to $(x,\tau) \in (X/\!/ \Gamma)_2$ is
\begin{equation} \label{eq:extquot2Module}
\mathrm{Ind}_{\mathcal O (X) \rtimes \Gamma_x}^{\mathcal O (X) \rtimes \Gamma} (\C_x \otimes \tau) .
\end{equation}
The space Prim($\mathcal{O}(X)\rtimes \Gamma$) is endowed with the Jacobson topology, which
makes it a nonseparated algebraic variety. This can be transferred to a topology on 
$(X/\!/ \Gamma)_2$, which we also call the Jacobson topology. It is slightly coarser than 
the naive topology described above.

The bijection $\epsilon_\psi$ is not always continuous with respect to the Jacobson topology.
In fact, it follows readily from \eqref{eq:extquot2Module} that $\epsilon_\psi$ is continuous 
in this sense if and only the following additional condition is satisfied:
\begin{itemize}
\item[(4)] Suppose that $x,y$ lie in the same connected component of the variety $X^{\Gamma_x}$,
that $\Gamma_y \supset \Gamma_x$ and that $\gamma \in \Gamma_x$. Then the 
$\Gamma_y$-representation $\psi_y ([\gamma])$ appears in $\mathrm{Ind}_{\Gamma_x}^{\Gamma_y} 
\psi_x ([\gamma])$.
\end{itemize}
While the first three conditions are easy to fulfill, the fourth can be rather difficult.

The two finite-type $\mathcal{O}(X/\Gamma)$-algebras $\mathcal{O}(X/\!/\Gamma)$ and 
$\mathcal{O}(X)\rtimes \Gamma$ are usually (i.e. if the action of $\Gamma$ on $X$ is neither 
trivial nor free) not Morita equivalent. In examples relevant to the representation theory of 
reductive $p$-adic groups these two finite-type $\mathcal{O}(X/\Gamma)$-algebras are
equivalent via a weakening of Morita equivalence referred to as ``geometric equivalence",
see the appendix.

\section{Statement of the conjecture}
\label{sec:statement}
As above, $\mathcal{G}$ is a quasi-split connected reductive $p$-adic group or
an inner form of $\GL_n(F)$, and 
$\mathfrak{s}\in\mathfrak{B}(\mathcal G)$.
Let $\{\mathrm{Langlands\:\: parameters}\}^\fs$ be the set of Langlands 
parameters for $\mathfrak{s}\in\mathfrak{B}(\mathcal G)$ and $H^\fs$ the 
stabilizer of this set in the dual group $G$.

The conjecture consists of the following five statements. 
\begin{enumerate}
\item  The infinitesimal character 
\[
\pi^\mathfrak{s} : \Irr (\mathcal G)^\fs \to T^\mathfrak{s} / W^\mathfrak{s}
\]  
is one-to-one if and only if the action of $ W^\mathfrak{s}$ on $T^\mathfrak{s}$ is free.
\item The extended quotient of the second kind $(T^{\mathfrak{s}} /\!/ W^{\mathfrak{s}})_2$ 
surjects by a canonical finite-to-one map onto 
$\{\mathrm{Langlands\:\: parameters}\}^\fs / H^\fs$.
\item The extended quotient of the second kind $(T^{\mathfrak{s}}/\!/W^{\mathfrak{s}})_2$ 
is canonically in bijection with the Bernstein component $\Irr (\mathcal G)^\fs$.
\item The canonical bijection 
\[
(T^{\mathfrak{s}}/\!/W^{\mathfrak{s}})_2 \longleftrightarrow \Irr (\mathcal G)^\fs
\]
comes from a canonical geometric equivalence of the two unital finite-type
$\mathcal{O}(T^{\mathfrak{s}}/W^{\mathfrak{s}})$-algebras 
$\mathcal{O}(T^{\mathfrak{s}})\rtimes W^{\mathfrak{s}}$ and $\mathcal H^\fs$.
See the appendix for details on ``geometric equivalence".
\item The above maps and the local Langlands correspondence fit in a commutative triangle
\begin{equation*}  
\xymatrix{   
& (T^\fs \q W^\fs)_2 \ar[dr]\ar[dl] & \\ 
\Irr(\mathcal{G})^\fs   \ar[rr] & &   \{\mathrm{Langlands\:\: parameters}\}^\fs/H^{\fs}} 
\end{equation*}
\item A $c$-$\Irr$ system can be chosen for the action of $W^{\mathfrak{s}}$ on 
$T^{\mathfrak{s}}$ such that the resulting bijection
\[
\epsilon \colon T^{\mathfrak{s}}/\!/W^{\mathfrak{s}}\longrightarrow 
(T^{\mathfrak{s}} /\!/ W^\fs )_2
\]
when composed with the canonical bijection 
$(T^\fs /\!/ W^\fs )_2 \rightarrow \Irr (\mathcal G)^\fs$
gives a bijection
\[
\mu^\fs : T^\fs /\!/ W^\fs \longrightarrow \Irr (\cG)^\fs
\]
which has the following six properties:
\end{enumerate}

\noindent Notation for Property 1:\\
Within the smooth dual $\mathrm{Irr}(\mathcal G)$, we have the tempered dual \\
$\Irr (\mathcal G)_{\mathrm{temp}}$ = 
\{smooth tempered irreducible representations of $\mathcal G\} /\sim$\\
$T_{\cpt}^\fs$ = maximal compact subgroup of  $T^\fs$.\\
$T_{\cpt}^\fs$ is a compact real torus.
The action of  $W^\fs$  on $T^\fs$ preserves $T_{\cpt}^\fs$, so we can form 
the compact orbifold $T_{\cpt}^\fs /\!/ W^\fs$.\\

\noindent\underline{Property 1 of the bijection $\mu^\fs$} :\\
The bijection $\mu^\fs : T^\fs /\!/ W^\fs \longrightarrow  \Irr (\cG)^\fs $ maps 
$T_{\cpt}^\fs /\!/ W^\fs$ onto $\Irr (\mathcal G)^\fs \cap 
\Irr (\mathcal G)_{\mathrm{temp}}$, and hence restricts to a bijection
\[
\mu^\fs \colon T_{\cpt}^\fs /\!/ W^\fs 
\longleftrightarrow \Irr (\mathcal G)^\fs \cap \Irr (\mathcal G)_{\mathrm{temp}}
\]

\noindent\underline{Property 2 of the bijection $\mu^\fs$} :

\noindent For many $\mathfrak{s}$ the diagram 
\begin{displaymath} 
\xymatrix{
T^\fs /\!/W^\fs \ar[dr]_{\rho^\fs} \ar[rr]^{\mu^\fs} 
&  &     \Irr (\mathcal G)^\fs \ar[dl] ^{\pi^\fs}  \\
 & T^\fs / W^\fs & }
\end{displaymath}
does not commute.
\vspace{3mm}

\noindent\underline{Property 3 of the bijection $\mu^\fs$}:

\noindent In the possibly non-commutative diagram
\begin{displaymath} 
\xymatrix{
T^\fs /\!/W^\fs \ar[dr]_{\rho^\fs} \ar[rr]^{\mu^\fs}       
&  &     \Irr (\mathcal G)^\fs \ar[dl] ^{\pi^\fs}  \\
 & T^\fs / W^\fs & }
\end{displaymath}
the bijection $\mu^\fs : T^\fs /\!/W^\fs \longrightarrow \Irr (\mathcal G)^\fs$ 
is continuous where $T^\fs /\!/W^\fs$ has the Zariski topology and 
$\Irr (\mathcal G)^\fs$ has the Jacobson topology  --- and the composition 
\[
\pi^\fs \circ \mu^\fs : T^\fs /\!/W^\fs 
\longrightarrow T^\fs/ W^\fs
\]
is a morphism of algebraic varieties.
\vspace{3mm}

\noindent\underline{Property 4 of the bijection $\mu^\fs$}:

\noindent  There is an algebraic family
\[
\theta_z : T^\fs/\!/W^\fs \longrightarrow T^\fs/W^\fs
\]
of finite morphisms of algebraic varieties, with $z \in \Cset^{\times}$, such that
\[
\theta_{1} = \rho^\fs,  \quad  \theta_{\sqrt{q}} = 
\pi^\fs \circ \mu^\fs,\quad\text{and}\quad 
\theta_{\sqrt{q}}(T^\fs /\!/W^\fs - T^\fs/W^\fs) = R(\pi^\fs).
\]
Here $q$ is the order of the residue field of the p-adic field $F$ over which $\mathcal G$ 
is defined and $R(\pi^\fs)\subset T^\fs/W^\fs$ is the sub-variety of reducibility. Setting 
\[
Y_{z} = \theta_{z}(T^\fs/\!/W^\fs - T^\fs/W^\fs)
\]
a flat family of sub-varieties of $T^\fs/W^\fs$ is obtained with 
\[
Y_1 =  R(\rho^\fs) , \qquad Y_{\sqrt{q}} = R(\pi^\fs) .
\]
\vspace{1mm}

\noindent\underline{Property 5 of the bijection $\mu^\fs$ (Correcting cocharacters)}:

\noindent For each irreducible component  $\bf{c}$ of the affine variety    
$T^\fs/\!/W^\fs$ there is a cocharacter (i.e. a homomorphism of algebraic groups)
\[
h_{c} : \Cset^{\times} \longrightarrow  T^\fs
\]
such that 
\[ 
\theta_{z} [w,t] = b(h_{c}(z) \cdot t) 
\]
for all $[w, t]\in\bf{c}$. \\ 
Let $b\colon T^\fs \longrightarrow T^\fs/W^\fs$ be the quotient map.
Here, as above, points of $\widetilde{T}_\fs$ are pairs $(w, t)$ with 
$w\in W^\fs,\:t\in T^\fs$ and $wt=t$. $[w, t]$ is the 
point in $T^\fs/\!/W^\fs$ obtained by applying the quotient map 
$\widetilde{T}^\fs \rightarrow T^\fs /\!/W^\fs$ to $(w, t)$.\\
\noindent Remark. The equality 
$$ \theta_{z} [w,t] = b(h_{c}(z) \cdot t) $$
is to be interpreted thus:\\
Let $Z_1$, $Z_2$, $\ldots$, $Z_r$ be the irreducible components of the affine 
variety 
$T^\fs /\!/ W^\fs$ and let $h_1$, $h_2$, $\ldots$, $h_r $ be the cocharacters
as in the statement of Property 5. Let 
\[
\nu^\fs \colon \widetilde{T^\fs} \longrightarrow T^\fs /\!/W^\fs
\]
be the quotient map.\\
Then irreducible components $X_1$, $X_2$, $\ldots$, $X_r$
of the affine variety $\widetilde{T^\fs}$ can be chosen with
\begin{itemize}
\item $\nu^\fs (X_j) = Z_j$ for $j =1, 2, \dots, r$
\item For each $z \in \mathbb{C}^{\times}$ the map 
$m_z \colon X_j \rightarrow T^\fs /W^\fs$, which is the composition    
\[
X_j\longrightarrow \quad T^\fs \quad \longrightarrow T^\fs /W^\fs
\]
\[
(w, t) \longmapsto h_j(z)t \longmapsto b(h_j(z)t) ,
\]
makes the diagram
\begin{displaymath} 
\xymatrix{
X_j\ar[dr]_{m_z} \ar[rr]^{\nu^\fs}  &  & Z_j \ar[dl] ^{\theta_z}  \\
& T^\fs / W^\fs & }
\end{displaymath}
commutative. Note that $h_j(z)t$ is the product of $h_j(z)$ and $t$ 
in the algebraic group $T^{\fs}$.
\end{itemize}
Remark. The conjecture asserts that to calculate 
\[
\pi^\mathfrak{s} \colon {\Irr}^\fs(\cG)\longrightarrow T^\mathfrak{s}/W^\mathfrak{s}
\]
two steps suffice:\\
Step 1: Calculate $\rho^\mathfrak{s} \colon T^\mathfrak{s}/\!/W^\mathfrak{s} \longrightarrow  
T^\mathfrak{s}/W^\mathfrak{s}.$\\
Step 2:  Determine the correcting cocharacters. \\

\vspace{2mm} 
\noindent The cocharacter assigned to $T^\fs/W^\fs \hookrightarrow T^\fs/\!/W^\fs$ 
is always the trivial cocharacter mapping $\C^{\times}$
to the unit element of $T^\fs$. So all the non-trivial correcting is taking place on 
$T^\fs/\!/W^\fs - T^\fs/W^\fs$.

\vspace{3mm}
\noindent Notation for Property 6.\\
If $S$ and $V$ are sets, a \emph{labelling} of $S$ by $V$ is a  map of sets 
$\lambda \colon S\rightarrow V$.
\noindent\underline{Property 6 of the bijection $\mu_\mathfrak{s}$ (L-packets)}:\\
As in Property 5, let $\{Z_1, \ldots, Z_r\}$ be irreducible components of the affine 
variety $T^\fs /\!/W^\fs$. Then a labelling $\lambda\colon \{Z_1, Z_2, \ldots Z_r\} 
\rightarrow V$ exists such that:\\
For every two points $[w,t]$ and $[w',t]$ of $T^\fs/\!/W^\fs$:
\[
\mu^\fs [w,t] \;\mathrm{and}\; \mu^\fs [w',t'] \text{ are in the same L-packet}
\]
if and only
\begin{itemize}
\item[(i)] $\theta_z [w,t] = \theta_z [w',t']$ for all $z \in \C^\times$; 
\item[(ii)] $\lambda [w,t] = \lambda [w',t']$, where we lifted $\lambda$ to a
labelling of $T^\fs /\!/ W^\fs$ in the obvious way.
\end{itemize}

\noindent Remark. An L-packet can have non-empty intersection with more than one 
Bernstein component. The conjecture does not
address this issue. The conjecture only describes the intersections of L-packets 
with any one given Bernstein component.\\
 
\noindent In brief, the conjecture asserts that --- once a Bernstein component 
has been fixed --- intersections of L-packets with that Bernstein component 
consisting of more than one point are ``caused" by repetitions among the 
correcting cocharacters. If, for any one given Bernstein component, the 
correcting cocharacters 
$h_1$, $h_2$, $\ldots$, $h_r$ are all distinct, then
(according to the conjecture) the intersections of L-packets with that Bernstein 
component are singletons.\\

\noindent A Langlands parameter can be taken to be a homomorphism of topological groups 
\[
\mathbf W_F \times \SL_2 (\mathbb{C}) \longrightarrow {}^L\mathcal G ,
\]
where $F$ is the $p$-adic field over which $\mathcal G$ is defined,  $\bW_F$ is the Weil 
group of $F$, and ${}^L\mathcal G$ is the $L$-group of $\mathcal G$.  
Let $G$, as before, denote the complex dual group of $\cG$.
By restricting a 
Langlands parameter to the standard maximal torus of $\SL_2 (\mathbb{C})$ a cocharacter
\[
\mathbb{C}^{\times} \longrightarrow T
\]
is obtained, where $T$ is a maximal torus of $G$.    
In examples, these give the correcting cocharacters. \\

\noindent For any $\mathcal G$ and any $\fs\in\mathfrak{B}(\mathcal G)$ the finite group 
$W^\fs$ is an extended finite Coxeter group i.e. is a semi-direct product for the action 
of a finite abelian group $\Gamma$ on a finite Weyl group $W$:
\[
W^\fs = W\rtimes \Gamma .
\]
Due to this restriction on which finite groups can actually occur as a $W^\fs$, 
in examples there is often a clear preferred choice
of $c$-$\Irr$ system for the action of $W^\fs$ on $T^\fs$.

\smallskip
What happens if $\cG$ is a connected reductive $p$-adic group which is
not quasi-split? Many Bernstein components $\Irr (\cG)^\fs$ have the
geometric structure as in the statement of the conjecture above. However, in some examples 
there are Bernstein components $\Irr (\cG)^\fs$ which are canonically in bijection not
with $(T^{\fs}/\!/W^\fs)_2$ but with $(T^{\fs}/\!/W^\fs)_2$-twisted by a 2-cocycle. 
See Section~\ref{sec:teq} for the definition of
$(T^{\fs}/\!/W^\fs)_2$-twisted by a 2-cocycle. 
The authors of this paper are currently formulating a precise statement of the conjecture 
for connected reductive p-adic groups which are not quasi-split. 
Our precise statement will be given elsewhere.

\part{Examples}

\section{Remarks on the supercuspidal case}
\label{sec:sc}
Recall the group $X_{\unr}(\cM)$ of unramified characters of $\cM$, its
finite subgroup  
$\Stab(\sigma)$ defined in (\ref{eqn:Psis}, and the complex algebraic
torus $T^\fs$, defined in (\ref{eq:defTs}) as the quotient of $X_{\unr}(\cM)$ 
by $\Stab(\sigma)$.
We first consider the special case in which $\fs\in\fB(\cG)$ is a
\emph{supercuspidal inertial pair} in $\cG$, that is, $\fs=[\cG,\pi]_G$,
where $\pi$ is a supercuspidal irreducible representation of $\cG$. 
We have 
\[T^\fs\simeq \Irr(\cG)^\fs.\]
On the other hand, the group $W^\fs$ is the trivial group $\{1\}$.
Hence we have $(T^\fs\q W^\fs)_2=T^\fs/W^\fs=T^\fs$. 
It follows that the left slanted map in the commutative triangle from statement
(5) in Section \ref{sec:statement} is here the identity map.
In other words, \emph{when $\fs$ is supercuspidal, 
the triangle collapses into the horizontal map}, \ie 
the existence of the commutative triangle of maps is equivalent to the 
existence of the local Langlands map.

\smallskip

The other parts of the conjecture are trivially true when $\fs$ is
supercuspidal, except Property~6. 
Each of the maps $\mu^\fs$,  $\pi^\fs$, and $\rho^\fs$ being equal to the 
identity map, we have 
\[\theta_{\sqrt q}=\theta_1\,(=\identity).\]
It implies that only the trivial cocharacter can occur. 
On the other hand, the labelling should consist in a \emph{unique label}.
Hence the conditions~(i) and~(ii) in Property~6 are empty. 
It follows that Property~6 is equivalent to the following property:

\smallskip

\noindent\underline{Property S:}
When $\fs$ is supercuspidal, the set of representations of $\cG$ which belong 
to the intersection of $\Irr(\cG)^\fs$ with an L-packet is always a singleton.

\smallskip

Property S is obviously true if $Z(\cG)$ is compact, because in that case there 
is no non-trivial unramified character of $\cG$, thus $\Irr(\cG)^\fs$ itself is 
reduced to a singleton. It is also valid if $\cG=\GL_n(F)$, because 
each L-packet is a singleton.

Property S is expected to hold in general. For instance, it is true for the 
supercuspidal L-packets constructed in \cite{DR}.

\section{The general linear group and its inner forms}
\label{sec:GLn}
Let $F$ be a local non-archimedean field. 
Let $\cG$ be an inner form of
the general linear group $\cG^*:=\GL_n(F)$, $n\ge 1$, that is, $\cG$ is a 
group of the form $\GL_m(D)$, where $D$ is an $F$-division algebra, of 
dimension $d^2$ over its centre $F$, and where
$n=md$ (see for instance \cite[\S~25]{ArIT}). By $F$-division algebra we mean 
a finite dimensional $F$-algebra with centre $F$, in which every nonzero
element is invertible.

We have $\cG=A^\times$, where $A$ is a simple central $F$-algebra. Let 
$V$ be a simple left $A$-module. The algebra $\End_A(V)$ is an $F$-division 
algebra, the opposite of which we denote by $D$. 
Considering $V$ as a right $D$-vector space, we have a canonical
isomorphism of $F$-algebras between $A$ and $\End_D(V)$. 

Heiermann \cite{Hei} proved that
every Bernstein component of the category of smooth $\cG$-modules is 
equivalent to the
module category of an affine Hecke algebra. Together with
\cite[Theorem 5.4.2]{Sol} this proves a large
part of the ABPS conjecture for $\cG$: properties 1--5 from Section
\ref{sec:statement}. In particular, this provides a bijection between
$T^\fs/\!/W^\fs$ and $\Irr(\mathcal{G})^\fs$ for any point $\fs\in\fB(\cG)$. 
Moreover, $T^\fs/\!/W^\fs$ and $(T^\fs/\!/W^\fs)_2$ are canonically isomorphic. 
In subsection~\ref{subsec:left}, we shall construct a canonical bijection
from $(T^\fs/\!/W^\fs)_2$ to $\Irr(\mathcal{G})^\fs$ by following a 
different approach based on
the fact due to S\'echerre and Stevens \cite{SeSt} that the affine Hecke 
algebra which occurs in the picture here is a tensor product of affine
Hecke algebras of equal parameter type. 

\smallskip

Let $\bW_F$ denote the Weil group of $F$. 
The complex dual group of the group $\cG$ is $G=\GL_m(\Cset)$. An $L$-group for
a given $p$-adic group can take one of several forms. The $L$-group
of an inner form of a (quasi-)split group may be identified with the
$L$-group of the latter \cite[\S~26]{ArIT}. Moreover, the $L$-group of a
split group can be taken to be equal to its complex dual group. Hence we may and
we do define the $L$-group of $\cG$ as 
\[
{}^L\cG=\GL_n(\Cset).
\]
Then a Langlands parameter for $\cG$ is a \emph{relevant} continuous 
homomorphism of topological groups 
\[
\Phi\colon \bW_F\times\SL_2(\Cset)\longrightarrow  \GL_n(\Cset),
\]
for which the image in $G$ of any element is semisimple, and which
commutes with the projections of $\WD$ and $G$ onto $\bW_F$. 
Two Langlands parameters are equivalent if they are conjugate under ${}^L\cG$. 
The set of equivalence classes of Langlands parameters is denoted $\Phi(\cG)$.

We recall that ``relevant'' means that if the image of $\Phi$ is contained 
in a Levi subgroup
$M$ of $G$, then $M$ must be the $L$-group of a Levi subgroup $\cM$ of
$\cG$. Hence here it means that $d$ must divide all the $n_i$ if 
$M\simeq\GL(n_1,\Cset)\times\cdots\times\GL(n_h,\Cset)$. Then we set $m_i := n_i/d$ 
and define $\cM=\GL(m_1,D)\times\cdots\times\GL(m_h,D)$. In the particular case of
$d=1$, \ie for the group $\cG^*$, every Langlands parameter is relevant. Hence we get
\[
\Phi(\cG) \subset \Phi(\cG^*).
\] 
In the next three subsections we shall construct, for every Bernstein component 
of $\cG$, a commutative triangle of natural bijections as Property 7 of 
Section \ref{sec:statement}.

\subsection{Types and Hecke algebras for $\GL_m(D)$}
\label{subsec:left} \ 

Let $\fs$ be any point in $\fB(\cG)$. 
Recall from \cite{BKtyp} that an \emph{$\fs$-type} is a pair $(\cK,\lambda)$,
with $\cK$ an open compact subgroup of $\cG$ and $(\lambda,\cV)$ is an
irreducible smooth representation of $\cK$, such that $\Irr^\fs(\cG)$ is precisely
the set of irreducible smooth representations of $\cG$ which contain $\lambda$.
We shall denote by $(\check\lambda,\cV^\vee)$ the contragredient representation of 
$\lambda$. The endomorphism-valued Hecke algebra $\cH(\cG,\lambda)$ attached to
$(\cK,\lambda)$ is defined to be the space of compactly supported functions
$f\colon \cG\to\End_{\Cset}(\cV^\vee)$, such that   
\[
f(k_1gk_2)=\check\lambda(k_1)f(g)\check\lambda(k_2),\quad\text{where }
k_1,k_2 \in \cK \text{ and } g\in \cG .
\]
The standard convolution gives $\cH(\cG,\lambda)$ the structure of an
associative $\Cset$-algebra with unity. There is a canonical bijection 
\begin{equation} \label{eqn:SS}
\Irr^\fs(\cG)\to \Irr(\cH(\cG,\lambda)).
\end{equation}
Generalizing the work of Bushnell and Kutzko \cite{BKsstyp}, S\'echerre and 
Stevens have constructed in \cite{SeSt} an $\fs$-type $(\cK^\fs,\lambda^\fs)$
for each $\fs\in\fB(\cG)$ and explicitly described the structure of the
algebra $\cH(\cG,\lambda^\fs)$. 

We shall recall the results from \cite{SeSt} that we need.
Let $\fs=[\cM,\sigma]_\cG$. The Levi subgroup $\cM$ is the stabilizer
of some decomposition $V=\bigoplus_{j=1}^h V_j$ into subspaces, which
gives an identification 
\begin{equation} \label{eqn:LeviM}
\cM\simeq\prod\nolimits_{j=1}^h \GL_{m_j}(D),\quad\text{where } m_j = \dim_D V_j .
\end{equation}
We can then write 
\[
\sigma = \bigotimes\nolimits_{j=1}^h\sigma_j,
\]
where, for each $j$, the representation $\sigma_j$ is an irreducible 
unitary supercuspidal representation of $\cG_j:=\GL_{m_j}(D)$. 

We define an equivalence relation on $\{1,2,\ldots,h\}$ by
\begin{equation} \label{eqn:equivrel}
j\sim k\quad  \text{if and only if} \quad m_j = m_k \quad \text{and} \quad
[\cG_j,\sigma_j]_{\cG_j}=[\cG_k,\sigma_k]_{\cG_k} ,
\end{equation}
where we have identified $\cG_j$ with $\cG_k$ whenever $m_j=m_k$. We may,
and do, assume that $\sigma_j=\sigma_k$ whenever $j\sim k$, since this
does not change the inertial class $\fs$. Denote by $S_1$, $S_2$,
$\ldots$, $S_l$ the equivalence classes. For $i=1,2,\ldots, l$, we denote
by $e_i$ the cardinality of $S_i$. We call the integers $e_1$, $e_2$, 
$\ldots$, $e_l$ the \emph{exponents} of $\fs$. Hence we get
\begin{equation} \label{eqn:Msigma}
\cM \simeq \prod\nolimits_{i=1}^l\GL_{m_i}(D)^{e_i} \quad\text{and}\quad
\sigma \simeq \sigma_1^{e_1}\otimes\cdots\otimes \sigma_l^{e_l},
\end{equation}
where $\sigma_1 , \ldots , \sigma_l$ are pairwise distinct (after
unramified twist). We abbreviate 
$\fs_i := [\GL_{m_i}(D)^{e_i},\sigma_i^{\otimes e_i}]_{\GL_{m_i e_i}(D)}$
and we say that $\fs$ has \emph{exponents} $e_1, \ldots, e_l$.
In the setting of \eqref{eqn:Msigma}
\begin{align} \label{eq:Ws}
& W^\fs = N_{\cG}(\cM,\sigma) / \cM = 
\prod_{i=1}^l W^{\fs_i} \cong \prod_{i=1}^l \fS_{e_i} , \\
& \text{Stab}(\sigma) = \{ \chi \in X_\unr (\cM) : \sigma \otimes \chi \} =
\prod\nolimits_{i=1}^l \text{Stab}(\sigma_i)^{e_i} .
\end{align}
Recall that every unramified character of $\cG_i = \GL_{m_i}(D)$ is of the form
$g \mapsto |\text{Nrd}(g)|_F^z$ for some $z \in \C$, where Nrd $: \Mat_{m_i}(D) \to F$ 
is the reduced norm. This sets up natural isomorphisms
\begin{align*}
& X_{\unr}(\cG_i) \cong \C \Big/ \frac{2 \pi \sqrt{-1}}{\log q_F} 
\mathbb Z \cong \C^\times , \\
& X_{\unr}(\cM) \cong \prod\nolimits_{i=1}^l (\C^\times )^{e_i} .
\end{align*}
Let $n(\sigma_i)$ be the torsion number of $\sigma_i$, that is,
the order of Stab$(\sigma_i)$. Let $T^\fs$ and $T_i$ be the Bernstein tori associated 
to $\fs$ and $[\cG_i,\sigma_i]_{\cG_i}$, as in \eqref{eq:defTs}. There are isomorphisms
\begin{align}
& \nonumber T_i \cong X_{\unr}(\cG_i) / \text{Stab}(\sigma_i) \cong 
\C \Big/ \frac{2 \pi \sqrt{-1}}{n(\sigma_i) \log q_F} \mathbb Z \cong \C^\times, \\
& \label{eq:TfsM} T^\fs = \prod_{i=1}^l T^{\fs_i} = 
\prod_{i=1}^l ( T_i )^{e_i} \cong X_{\unr}(\cM) / \text{Stab}(\sigma) .  
\end{align}
As $W^\fs$ stabilizes $\sigma \in \Irr (\cM)$, the bijection \eqref{eq:TfsM} is 
$W^\fs$-equivariant. Although the above isomorphisms are not canonical, a consequence of 
the next theorem is that they can be made so by an appropriate choice of the $\sigma_i$.

\begin{thm} \label{thm:GLDleft}
The extended quotient of the second kind $(T^\fs\q W^\fs)_2$ is canonically 
in bijection with the Bernstein component $\Irr(\cG)^\fs$.
\end{thm}
\begin{proof}
For every $i\in\{1,2,\ldots,l\}$, as proved in \cite{SeSt}, there exists a pair
$(\cK_i,\lambda_i)$, formed by an open compact
subgroup $\cK_i$ of $\cG_i=\GL_{m_i}(D)$ and a smooth irreducible representation
$\lambda_i$ such that the representation $\widetilde\lambda_i$ extends $\lambda_i$ 
to the normalizer $\widetilde \cK_i$ of $\cK_i$ in $\cG_i$ and the
supercuspidal representation $\sigma_i$ is compactly induced
from $\widetilde\lambda_i$: 
\begin{equation} \label{eq:cInd}
\sigma_i=\cInd_{\widetilde \cK_i}^{\cG_i}\widetilde\lambda_i.
\end{equation}
The group $\cK_i$ is a \emph{maximal simple type} in the
terminology of \cite{Sec}. Its construction (given in \cite{Sec})
generalizes the construction made by Bushnell and Kutzko in \cite{BK} in
the case of the general linear group. 

We shall denote by $\fo_D$ the ring of integers of $D$, and by $q_D$ the
order of the residue field of $D$. 

The group $\cK_i=\cK(\fA_i,\beta_i)$ is built from a \emph{simple stratum} 
$[\fA_i,a_i,0,\beta_i]$ of the $F$-algebra $A_i:=\End_D(V_i)$. 
The definition of a simple stratum is given in \cite[Def~1.18]{Sec}. Here 
we just recall that $\beta_i$ is an element 
of $A_i$ such that the $F$-algebra $E_i:=F[\beta_i]$ is a field, 
$\fA_i$ is an $\fo_F$-hereditary order of $A_i$ that is normalized by
$E_i^\times$, and $a_i$ is a positive integer.
Then the centralizer of $E_i$ in $A_i$, denoted $B_i$, is a simple
central $E_i$-algebra. We fix a simple left $B_i$-module $V_{E_i}$ and write 
$D_{E_i}$ for the algebra opposite to $\End_{B_i}(V_{E_i})$. We define
\begin{equation} \label{eqn:qi} 
\begin{aligned}
& m_{E_i}:=\dim_{D_{E_i}}V_{E_i} ,\qquad q_i:=q_{D_{E_i}}^{m_{E_i}} , \\
& \tW_{e_i} := X^* \big( T_i \big)^{e_i} \rtimes \fS_{e_i} =
X^* \big( T^{\fs_i} \big) \rtimes W^{\fs_i} .
\end{aligned}
\end{equation}
Notice that $\tW_{e_i}$ is an extended affine Weyl group of type $\GL_{e_i}$.
Now we quote \cite[Main Theorem and Prop.~5.7]{SeSt}: there is an
isomorphism of unital $\mathbb{C}$-algebras
\begin{equation} \label{eqn:HH}
\mathcal{H}(\cG,\lambda^\fs) \cong \bigotimes\nolimits_{i=1}^l \cH_{q_i}(\tW_{e_i}) .
\end{equation} 
For later use we also mention that this isomorphism preserves supports
\cite[Proposition 7.1 and Theorem 7.2]{SeSt}.
The factors $\mathcal{H}_{q_i}(\tW_{e_i})$ are (extended) affine Hecke algebras 
whose structure is given explicitly for instance in \cite[5.4.6, 5.6.6]{BK}.

By combining Corollary~\ref{cor:J} from Example~3 in the Appendix 
with the multiplicativity property of
the extended quotient of the second kind, we obtain a canonical bijection
\begin{equation} \label{eqn:bijs}
\bigotimes\nolimits_{i=1}^l
\cH_{q_i}(\tW_{e_i})\to(T^\fs\q W^\fs)_2.
\end{equation}
The composition of \eqref{eqn:SS}, \eqref{eqn:HH} and \eqref{eqn:bijs} 
gives a canonical bijection
\[
\Irr^\fs(\cG) \to (T^\fs\q W^\fs)_2. \qedhere
\] 
\end{proof}

Let $\cH(\cG)$ be the Hecke algebra associated to $\cG$, \ie 
\[
\cH(\cG): = \bigcup\nolimits_\cK \cH(\cG//\cK),
\]
where $\cK$ are open compact subgroups of $\cG$, and
$\cH(\cG//\cK)$ is the convolution algebra of all complex-valued,
compactly-supported functions on $\cG$ which are $\cK$-biinvariant.
The Hecke algebra $\cH(\cG)$ is a non-commutative, non-unital,
non-finitely-generated $\Cset$-algebra. It admits a canonical
decomposition into ideals, the Bernstein decomposition:
\[
\cH(\cG) = \bigoplus_{\fs \in \fB(\cG)} \cH(\cG)^\fs.
\]
The following result extends Theorem~3.1 in \cite{BP} from $\GL_n (F)$ to its
inner forms. This can be viewed as a ``topological shadow'' of conjecture 1.1.

\begin{thm} \label{thm:HPHdR}
Let $\fs\in\fB(\cG)$. Then the periodic cyclic homology of $\cH(\cG)^{\fs}$ 
is isomorphic to the periodised de Rham cohomology of $T^\fs /\!/ W^\fs$:
\[
\HP_* (\cH(\cG)^{\fs}) \simeq H^*(T^\fs /\!/ W^\fs;\Cset).
\] 
\end{thm}
\begin{proof}
The proof follows those of \cite[Theorem~3.1]{BP}.
Let $e_{\lambda}^\fs$ be the idempotent attached
to the semisimple $\fs$-type $(\cK^\fs,\lambda^\fs)$ 
as in \cite[Definition 2.9]{BKtyp}:
\[
e_{\lambda}^\fs(g) = 
\begin{cases}(\text{vol}\, \cK^\fs)^{-1} \dim \lambda \,\tr
(\tau(g^{-1}))& \text{ if } g \in \cK^\fs \;,
\cr 0 & \text{ if } g \in \cG \setminus \cK^\fs .
\end{cases}\]
The idempotent $e_{\lambda}^\fs$ is then a \emph{special} idempotent in the 
Hecke algebra $\cH(\cG)$ according to \cite[Definition 3.11]{BKtyp}. 
It follows from \cite[\S 3]{BKtyp} that
\[
\cH^{\fs}(\cG) = \cH (\cG) * e_{\lambda}^\fs * \cH (\cG).
\]
We then have a Morita equivalence 
\[
\cH (\cG) * e_{\lambda}^\fs * \cH (\cG) \sim_{\Morita} 
e_{\lambda}^\fs * \cH (\cG) * e_{\lambda}^\fs.
\]
By \cite[2.12]{BKtyp} we have a canonical isomorphism of unital $\Cset$-algebras:
\begin{equation}\label{eq:HGs}
\mathcal{H}(\cG, \lambda^\fs) \otimes_{\mathbb{C}} \End_{\mathbb{C}} \cV
\cong e_{\lambda}^\fs * \mathcal{H}(\cG) * e_{\lambda}^\fs ,
\end{equation}
so that the algebra $e_{\lambda}^\fs * \mathcal{H}(\cG) * e_{\lambda}^\fs$ is Morita 
equivalent to the algebra $\mathcal{H}(\cG, \lambda^\fs)$. Then, thanks to the 
description of latter given in (\ref{eqn:HH}), we may finish the proof identically as
in \cite[Theorem~3.1]{BP}.
\end{proof}

\begin{ex} {\rm 
For the group $\GL_n (F)$ we can describe the flat family of sub-varieties 
of $T^\fs/W^\fs$ from Section \ref{sec:approx} explicitly.
We assume for simplicity that $\fs=[\GL_{r}(F)^e,\sigma^{\otimes e}]_\cG$, where 
$er=n$. Then the sub-variety $R(\rho^\fs)=\rho^\fs(T^\fs\q W^\fs - T^\fs/W^\fs)$ 
is the hypersurface
$Y_1$ given by the single equation $\prod_{i \neq j}(z_i - z_j) = 0$. 
The sub-variety of reducibility $R(\pi^\fs)$ is the variety
$Y_{\sqrt q}$ given by the single equation $\prod_{i \neq j}(z_i - qz_j) = 0$, 
according to a classical theorem \cite[Theorem 4.2]{BZ1}, \cite{Ze}.  

The polynomial equation $\prod_{i\neq j}(z_i - v^2z_j) = 0$ determines a flat 
family $Y_v$ of hypersurfaces. The hypersurface $Y_1$ is the \emph{flat limit} of 
the family $Y_v$ as $v \to 1$, as in \cite[p. 77]{EH}.  
} \end{ex}

\subsection{The local Langlands correspondence for $\GL_m(D)$} \

We recall the construction of the local Langlands correspondence for $\cG=\GL_m(D)$
from \cite{HS}. It generalizes and relies on the local Langlands 
correspondence for $\cG^* = \GL_n(F)$,
\[ 
\rec_{F,n} \colon \Irr(\GL_n (F)) \to \Phi (\GL_n (F)).
\]
The latter was proven for supercuspidal representations in \cite{LRS,HT,He}, and
extended from there to $\Irr (GL_n (F))$ in \cite{Ze}.

If $\pi$ is an irreducible smooth representation of $\cG$, we shall
denote its character by $\Theta_\pi$.
Let $\cEesq(\cG)$ denote the set of equivalence classes
of irreducible essentially square-integrable representations of $\cG$.
Recall the \emph{Jacquet--Langlands correspondence} \cite{DKV,BadJL}:

\noindent
There exists a bijection
\[
\JL\colon \cEesq(\cG)\to \cEesq(\cG^*)
\]
such that for each $\pi\in \cEesq(\cG)$:
\[
\Theta_{\pi}(g)=(-1)^{n-m}\,\Theta_{\JL(\pi)}(g^*),
\]
for any pair $(g,g^*)\in \cG\times \cG^*$ of regular semisimple elements
such that $g$ and $g^*$ have the same characteristic polynomial. 
(Recall that an element in $\cG$ is called \emph{regular semisimple} 
if its characteristic polynomial admits only simple roots in
an algebraic closure of $F$.) 

Let $\Phi$ be a Langlands parameter for $\cG$. Replacing it by an equivalent
one, we may assume that there exists a standard Levi subgroup 
$M \subset \GL_n (\C)$ such that the image of $\Phi$ is contained in $M$
but not in any smaller Levi subgroup. Again replacing $\Phi$ by an
equivalent parameter, we can achieve that
\[
M = \prod_{i=1}^l(\GL_{n_i}(\C))^{e_i}
\quad\text{and}\quad
\Phi = \prod_{i=1}^l \Phi_i^{\otimes e_i} ,
\]
where $\Phi_i \in \Phi (\GL_{n_i}(F))$ is not equivalent to $\Phi_j$ for 
$i \neq j$.
Since $\Phi$ is relevant for $\cG ,\; m_i := n_i / d$ is an integer and
$M$ corresponds to the standard Levi subgroup
\[
\cM = \prod\nolimits_{i=1}^l(\GL_{m_i}(D))^{e_i} \subset \GL_m (D) .
\]
By construction the image of $\Phi_i$ is not contained in any proper Levi
subgroup of $\GL_{n_i}(F)$, so $\rec_{F,n_i}^{-1}(\Phi_i) \in \cEesq (\GL_{n_i}(F))$.
The Jacquet--Langlands correspondence produces
\begin{align*}
& \sigma_i := \JL^{-1} \big( \rec_{F,n_i}^{-1}(\Phi_i) \big) \in \cEesq (\GL_{m_i}(D)) ,\\
& \sigma := \prod\nolimits_{i=1}^l \sigma_i^{\otimes e_i} \in \cEesq (\cM) .
\end{align*}
The assignment $\Phi \mapsto (\cM,\sigma)$ sets up a bijection
\begin{equation}\label{eq:LLC2}
\Phi (\cG) \longleftrightarrow \{ (\cM,\sigma) : \cM \text{ a Levi subgroup of } \cG ,
\sigma \in \cEesq (\cM) \} / \cG .
\end{equation}
It is known from \cite[Theorem B.2.d]{DKV} and \cite{Bad2} that for inner forms of
$\GL_n (F)$ normalized parabolic induction sends irreducible square-integrable
(modulo centre) representations to irreducible tempered representations. 
Together with the Langlands classification \cite{Kon} this implies that there exists 
a natural bijection between $\Irr (\cG)$ and the right hand side of \eqref{eq:LLC2}. 
It sends $(\cM,\sigma)$ to the Langlands quotient $L \big( I_\cM^\cG (\sigma) \big)$, 
where the parabolic induction goes via a parabolic subgroup with Levi factor $\cM$, 
with respect to which the central character of $\sigma$ is "positive". 
The combination of these results yields:

\begin{thm}\label{thm:LLCGLD}
The natural bijection
\[
\Phi (\cG) \to \Irr (\cG) : 
\Phi \mapsto (\cM,\sigma) \mapsto L \big( I_\cM^\cG (\sigma) \big) 
\]
is the local Langlands correspondence for $\cG = \GL_m (D)$. 
\end{thm}
We denote the inverse map by
\begin{equation}\label{eqn:LLCD}
\rec_{D,m} \colon \Irr (\cG) \to \Phi (\cG) . 
\end{equation}
Let $\fs = [\cM,\sigma]_{\cG}$ be an inertial equivalence class for $\cG$.
We want to understand the space of Langlands parameters
\[
\Phi (\cG)^\fs := \rec_{D,m}(\Irr^\fs (\cG)) 
\]
that corresponds to the Bernstein component $\Irr^\fs (\cG)$.
We may assume that $\cM$ and $\sigma$ are as in \eqref{eqn:Msigma}.
Via Theorem \ref{thm:LLCGLD} $\sigma_i$ corresponds to
\[
\rec_{D,m_i}(\sigma_i) = \rec_{F,n_i}(\JL (\sigma_i)) \in \Phi (\GL_{m_i}(D)) .
\]
We choose a representative $\eta_i \colon \mathbf W_F \times \SL_2 (\C) \to 
\GL_{n_i}(\C)$ and we put 
\[
\eta = \prod\nolimits_{i=1}^l (\eta_i )^{\otimes e_i} ,
\]
a representative for $\rec_{D,m}(I_\cM^\cG (\sigma))$.
Since $\JL (\sigma_i)$ is essentially square integrable, $\eta_i$ is an
irreducible representation of $\bW_F\times\SL_2(\Cset)$. 
It follows that the centralizer of $\eta_i$ in $\Mat_{n_i}(\Cset)$ is
$\Z(\Mat_{n_i}(\Cset)) = \Cset \Id$.
>From this we see that the centralizer of $(\eta_i)^{\otimes e_i}$ in 
$\Mat_{n_i e_i}(\Cset)$ is 
\[
\Z_{\Mat_{n_i e_i}(\Cset)} (\eta_i^{\otimes e_i}) = \Mat_{e_i}(\C) \otimes 
\Z(\Mat_{n_i}(\C)) 
\subset \Mat_{e_i}(\C) \otimes \Mat_{n_i}(\C) \cong \Mat_{n_i e_i}(\C) .  
\]
In this way the centralizer of $\eta_i^{\otimes e_i}$ in $\GL_{n_i e_i}(\C)$ 
becomes isomorphic to $\GL_{e_i}(\C)$ and
\begin{equation}
\Z_{\GL_n (\C)} (\eta) \cong \prod\nolimits_{i=1}^l \GL_{e_i}(\C) .
\end{equation}
This means that, for any Langlands parameter for $\GL_{e_i}(D)$
\[
\Phi_i \colon \mathbf W_F \times \SL_2 (\C) \to \GL_{e_i}(\C) \subset 
\GL_{n_i e_i}(\C) , 
\]
$\Phi_i \eta_i^{\otimes e_i}$ is a Langlands parameter for $\GL_{m_i e_i}(D)$. 
More generally, for any product $\Phi = \prod_{i=1}^l \Phi_i$ of such maps, 
\begin{equation}\label{eq:Phieta}
\Phi \eta \text{ is Langlands parameter for } \cG.
\end{equation} 
Let $\bI_F \subset \bW_F$ be the inertia group and let $\Frob \in \mathbf W_F$ 
be a Frobenius element. Since all Langlands parameters in $\Phi (\cG)^\fs$
have the same restriction to the inertia group $\mathbf I_F$ (up to equivalence), 
it suffices to consider $\Phi_i$ which are trivial on $\mathbf I_F$. We will show 
in the next subsection that $\Phi (\cG)^\fs$ indeed consists of such products 
$\Phi \eta$ (up to equivalence).

Recall from \cite[\S 10.3]{BorAut} that the local Langlands correspondence is compatible 
with twisting by central characters. The group $X_\unr (\GL_{m_i}(D))$ is naturally 
in bijection with $\Z(\GL_{n_i}(\C))$. Hence for any $\chi \in X_\unr (\GL_{m_i}(D))$
\[
\rec_{D,m_i}(\sigma_i \otimes \chi) = \Phi_\chi \rec_{D,m_i}(\sigma) = \Phi_\chi \eta_i, 
\]
where $\Phi_\chi$ is trivial on $\mathbf I_F \times \SL_2 (\C)$ and 
$\Phi_\chi (\Frob) \in \Z(\GL_{n_i}(\C))$ corresponds to $\chi$. 
By Theorem \ref{thm:LLCGLD} $\Phi_\chi \eta_i$
is $\GL_{n_i}(\C)$-conjugate to $\Phi_i$ if and only if $\chi \in \text{Stab}(\sigma_i)$.
More generally, for any $\chi \in X_\unr (\cM)$,
\begin{equation} \label{eq:etaConjugate}
\Phi_\chi \eta \text{ is $M$-conjugate to } \eta \text{ if and only if }
\chi \in \text{Stab}(\sigma) .
\end{equation} 
Thus the unramified twists of $\sigma \in \Irr (\cM)$ are naturally parametrized by
the torus $T^\fs \cong X_\unr (\cM) / \text{Stab}(\sigma)$.

\subsection{The commutative triangle for $\GL_m(D)$} \

Recall the short exact sequence 
\[
1 \to \bI_F \to \bW_F \overset{d} \to \Zset \cong \langle \Frob \rangle \to 1.
\]
A character $\nu\colon \bW_F\to\Cset^\times$ is said to be
\emph{unramified} if $\nu$ is trivial on $\bI_F$. Then $\nu(w)=z_\nu^{d(w)}$
for some $z_\nu\in\Cset^\times$ ($w\in \bW_F$).
Let $X_{\unr}(\bW_F)$ denote the group of all unramified characters of
$\bW_F$. Then the map $\nu\mapsto z_\nu$ is an isomorphism from $X_{\unr}(\bW_F)$
to $\Cset^\times$.
Denote by $\spe(a)$ the $a$-dimensional irreducible complex representation of
$\SL_2(\Cset)$. 
Recall that each $L$-parameter $\Phi$ is of the form
$\Phi=\Phi_1\oplus\cdots\oplus\Phi_h$ with each $\Phi_i$ irreducible. The next result is Theorem 1.2 in our Introduction.
   
\begin{thm} \label{thm:bijphi} 
Let $\fs\in\fB(\cG)$. There is a commutative diagram
\[
\xymatrix{   & (T^\fs/\!/W^\fs)_2 \ar[dr]^{\nnu^\fs} \ar[dl] &
\\  
\Irr(\mathcal{G})^\fs    \ar[rr] & & \Phi(G)^\fs  } 
\]
in which all the arrows are natural bijections.
\end{thm}
\begin{proof}
The bottom and left slanted maps where already established in Theorems 
\ref{thm:GLDleft} and \ref{thm:LLCGLD}. Since these are canonical bijections,
we could simply define the right slanted map as the composition of the other two.
Yet we prefer to give an explicit construction of $\nnu^\fs$, which highlights
its geometric origin. This is inspired by \cite[\S~1]{BP} and 
\cite[Theorem~4.1]{ABP2}, with the difference that we are using here the extended 
quotient of the second kind $(T^\fs/\!/W^\fs)_2$ instead of $(T^\fs/\!/W^\fs)$.

Let $\cM$ be a Levi subgroup of $\cG$ and $\sigma$ an irreducible unitary 
supercuspidal representation of $\cM$ such that $\fs=[\cM,\sigma]_\cG$.  
As in \eqref{eqn:Msigma}, we may assume that 
\[
\cM \simeq \prod\nolimits_{i=1}^l\GL_{m_i}(D)^{e_i} = \prod\nolimits_{i=1}^l\cG_i^{e_i}
\quad\text{and}\quad
\sigma \simeq \sigma_1^{e_1}\otimes\cdots\otimes \sigma_l^{e_l},
\]
where the $[\cG_i,\sigma_i]_\cG$ are pairwise distinct. 
Recall from \eqref{eq:TfsM} that
\[
T^\fs = \prod\nolimits_{i=1}^l (T_i )^{e_i} ,
\]
where each $T_i$ is isomorphic to $\C^\times$. To make the below
constructions well-defined, we must make a canonical choice for $\sigma_i$ in
its inertial equivalence class in $\Irr(\cG_i)$. The Hecke algebra
$\cH (\cG_i,\lambda_i)$ associated to this inertial class, as in \eqref{eqn:SS},
is canonically isomorphic to the ring of regular functions on 
$X_{\unr}(\cG_i) / \text{Stab}(\sigma_i) \cong \C^\times$. We choose $\sigma_i$ as 
\eqref{eq:cInd}, so that it corresponds to the unit element of this torus.
 
By \eqref{eq:Ws} the 
group $W^\fs$ is isomorphic to a product of symmetric groups $\fS_{e_i}$. 
Since the extended quotient of the second kind is multiplicative, we obtain
\begin{equation}\label{eq:multExtquot}
(T^\fs\q W^\fs)_2 \simeq (T_1^{e_1}\q \fS_{e_1})_2
\times \cdots \times (T_l^{e_l}\q \fS_{e_l})_2.
\end{equation}
We set $\fs_i:=[\cG_i^{e_i},\sigma_i^{\otimes e_i}]_{\cG'_i}$, 
for each $i\in\{1,\ldots,l\}$, where $\cG_i'=\GL_{e_im_i}(D)$.
We observe that
\begin{equation} \label{eqn:multPhi}
\Phi(\cG)^\fs\simeq\Phi(\cG'_1)^{\fs_1}\times\cdots\times
\Phi(\cG'_l)^{\fs_l},
\end{equation}
Put $\cM':=\cG_1'\times\cdots\times\cG_l'$. Then Theorem \ref{thm:LLCGLD} and 
\eqref{eqn:multPhi} imply that
\begin{equation}\label{eq:multIrr}
\begin{array}{ccc}
\Irr (\cG'_1)^{\fs_1} \times \cdots \times \Irr (\cG'_l)^{\fs_l} & \to & 
\Irr (\cG)^\fs ,\\
(\pi_1' ,\ldots ,\pi_l') & \mapsto & 
L \big( I_{\cM'}^\cG (\pi_1' \otimes \cdots \otimes \pi_l') \big)
\end{array}
\end{equation}
is a bijection.
Hence we are reduced to define canonical bijections 
\[
(T_i^{e_i}\q \fS_{e_i})_2 \to \Phi(\cG_i')^{\fs_i}.
\]
Thus we may, and do, assume that $m=er$ (so that $n = erd$) and 
\begin{equation}\label{eq:fsSpecial}
\fs=[\GL_{r}(D)^e,\sigma^{\otimes e}]_\cG.
\end{equation}
Then we have $T^\fs\simeq (\Cset^\times)^e$ and $W^\fs\simeq\fS_e$. 

Let $t\in(\Cset^{\times})^{e}$. We write $t$ in the form
\begin{equation} \label{eqn:descrt}
t=(\underbrace{z_1,\ldots,z_1}_{\text{$b_1$ terms}},
\underbrace{z_2,\ldots,z_2}_{\text{$b_2$ terms}},\ldots,
\underbrace{z_h,\ldots,z_h}_{\text{$b_h$ terms}}),
\end{equation}
where $z_1,\ldots,z_h\in\Cset^\times$ are such that $z_i\ne z_j$ for $i\ne j$, 
and where $b_1+b_2+\cdots+b_h=e$. 
Let $W_t^\fs$ denote the stabilizer of $t$ under the action of $W^\fs$. We  have 
\begin{equation} \label{eqn:Fixt}
W^\fs_t\simeq\fS_{b_1}\times\fS_{b_2}\times\cdots\times\fS_{b_h}.
\end{equation}
In particular every $\tau \in W^\fs$ can be written as
\begin{equation}\label{eq:multTau}
\tau = \tau_1 \otimes \cdots \otimes \tau_h \quad 
\text{with } \tau_j\in\Irr(\fS_{b_j}).
\end{equation}
Let $p_j=p(\tau_j)$ denote the partition of $b_j$ which corresponds to
$\tau_j$. We write:
\begin{equation}\label{eq:pj}
p_j = (p_{j,1},\ldots,p_{j,l_j}),\quad\text{where }
p_{j,1}+\cdots+p_{j,l_j} = b_j .
\end{equation}
Recall that we have chosen $\sigma \in \Irr (\GL_r(D))$ such that, via the Hecke 
algebra of a supercuspidal type, it corresponds to the unit element of the torus 
$X_\unr ( \GL_r (D) ) / \text{Stab}(\sigma)$. We fix a representative 
\begin{equation} \label{eqn:eta}
\eta \quad \text{for} \quad \rec_{D,r} (\sigma) \in \Irr(\WD).
\end{equation}
Corresponding to each pair $(\eta,p_j)$ as above, where $j\in \{1,\ldots,h\}$, 
we have a Langlands parameter: 
\begin{equation} \label{eqn:Phie}
\begin{aligned}
& \Phi_{\eta,p_j}^F \colon 
\WD\to \GL_{b_j r d}(\Cset) \subset \Mat_{b_j}(\C) \otimes \Mat_{rd}(\C) , \\
& \Phi_{\eta,p_j}^F = 
\big(\spe(p_{j,1}) \oplus \cdots \oplus \spe(p_{j,l_j}) \big) \otimes \eta .
\end{aligned}
\end{equation}
Recall that
\[
(T^\fs\q W^\fs)_2\simeq \left\{(t,\tau)\,:\,t\in (\Cset^{\times})^{e},
\tau \in\Irr(W^\fs_t)
\right\}/W^\fs.
\]
Define $\nu_j \in X_{\unr}(\bW_F)$ by $\nu_j(\Frob):=z_j$ for each
$j\in\{1,\ldots,h\}$. We fix a setwise section 
\[
\psi_\sigma \colon \C^\times \cong X_\unr ( \GL_r (D) ) / \text{Stab}(\sigma) 
\to \C^\times \cong X_\unr (\GL_r (D)) .
\]
We define a map
\begin{eqnarray} \label{eqn:parametre}
\begin{aligned}
& \nnu^\fs \colon (T^\fs\q W^\fs)_2\to \Phi(\cG)^\fs , \\
& \nnu^\fs (t,\tau) = 
(\psi_\sigma \circ \nu_1) \otimes \Phi_{\eta,p_1}^F \oplus \cdots \oplus 
(\psi_\sigma \circ \nu_h) \otimes \Phi_{\eta,p_h}^F,
\end{aligned}
\end{eqnarray}
where $W_t^\fs\simeq \fS_{b_1}\times\cdots\times\fS_{b_h}$ and 
$\tau=\tau_1\otimes\cdots\otimes\tau_h$, with $\tau_j\in\Irr(\fS_{b_j})$,
and $p_j=p(\tau_j)$ denotes the partition of $b_j$ which corresponds to $\tau_j$.
By \eqref{eq:etaConjugate} $\nnu^\fs (t,\tau)$ does not depend on the choice
of $\psi_\sigma$, up to equivalence of Langlands parameters.

The above process can be reversed. Let $\Phi\in\Phi(\cG)^\fs$. By 
\eqref{eq:etaConjugate} we may assume that its restriction to $\bW_F$ is of the form
\[
(\psi_\sigma )^e \circ ( \nu_1 \oplus \cdots \oplus \nu_1
\oplus \nu_2 \oplus \cdots \oplus \nu_2 \oplus \cdots \oplus \nu_h
\oplus \cdots \oplus \nu_h ) \otimes \eta,
\]
where $\nu_1,\ldots,\nu_h\in X_{\unr}(\bW_F)$ are such that $\nu_i\ne\nu_j$
if $i\ne j$ (which is equivalent, since $\nu_i$ and $\nu_j$ are unramified, 
to $\nu_i(\Frob)\ne\nu_j(\Frob)$ if $i\ne j$). For each $j\in\{1,\ldots,h\}$, 
let $b_j$ denote the number of occurrences of $\nu_j$. We set 
\[
t:=(\nu_1(\Frob),\nu_1(\Frob),\ldots,\nu_h(\Frob))\in(\Cset^\times)^e.
\] 
For each $j\in\{1,\ldots,h\}$, let $p_j$ be a partition of $b_j$, such that 
\begin{equation} \label{eqn:formPhi}
\Phi = \psi_\sigma \circ \nu_1 \otimes \Phi_{\eta,p_1}^F \oplus \cdots \oplus
\psi_\sigma \circ \nu_h \otimes \Phi_{\eta,p_h}^F .
\end{equation}
For $j=1,\ldots,h$, let $\tau_j\in\Irr(\fS_{b_j})$ be the irreducible
representation of $\fS_{b_j}$ which is parametrized by the partition $p_j$. 
We set $\tau:=\tau_1\otimes\cdots\otimes\tau_h$. The map
$\Phi\mapsto (t,\tau)$ is the inverse of the map $\nnu^\fs$. 
Thus $\nnu^\fs$ is a bijection.

Now we have all the arrows of the diagram in Theorem \ref{thm:bijphi}, it remains
to show that it commutes. Although all the arrows are canonical, this is not at 
all obvious, and our proof will use some results that will be established
only in part 3. 

In view of \eqref{eq:multExtquot}, \eqref{eqn:multPhi}
and \eqref{eq:multIrr}, it suffices to consider $\fs$ as in \eqref{eq:fsSpecial}.
Let ${\widetilde W}_e := \mathbb Z^e \rtimes \fS_e$ and $\tilde q := q_{D_E}^{m_E}$ 
be as in \eqref{eqn:qi}. The left slanted bijection in Theorem \ref{thm:bijphi}
runs via the affine Hecke algebra $\cH_{\tilde q} ({\widetilde W}_e)$ associated 
to $\fs$ in \eqref{eqn:HH}, so we must actually look at the extended diagram
\begin{equation}\label{eq:diagramGLe}
\xymatrix{
\Irr (\cH_{\tilde q} ({\widetilde W}_e)) \ar[r] & 
\big( (\C^\times )^e /\!/ \fS_e \big)_2 \ar[d]^{\varphi^\fs} \\
\Irr (\GL_{er}(D))^\fs \ar[u]^{\kappa^\fs} \ar[r]^{\rec_{D,er}} & \Phi (\GL_{er}(D))^\fs
} 
\end{equation}
First we consider the essentially square-integrable representations. For those
the Langlands parameter is an irreducible representation of $\mathbf W_F \times 
\SL_2 (\C)$, of the form
\[
\psi_\sigma \circ \nu_z \otimes \Phi_{e,\eta}^F = 
\psi_\sigma \circ \nu_z \otimes R(e) \otimes \eta ,
\]
where $\nu_z \in X_{\unr}(\mathbf W_F)$ with $\nu_z (\Frob) = z$.
The partition $(e)$ of $e$ corresponds to the sign representation of $\fS_e$, so
\begin{equation}\label{eq:nusign}
(\varphi^\fs )^{-1}(\psi_\sigma \circ \nu_z \otimes R(e) \otimes \eta) = 
(z,\ldots,z,\mathrm{sgn}_{\fS_e} ) .
\end{equation}
Let $\mathrm{St}_e$ be the Steinberg representation of $\cH_{\tilde q} 
({\widetilde W}_e)$. It is the unique irreducible, essentially square-integrable
representation which is tempered and has real infinitesimal central character.
The top map in \eqref{eq:diagramGLe}, as constructed by Lusztig \cite{LuCellsIII} 
and described in Example 3 of the Appendix, sends $(z,\ldots,z) \otimes 
\mathrm{St}_e$ to \eqref{eq:nusign}. This can be deduced with \cite[Theorem 8.3]{KL},
but we will also prove it later as a special case of Theorem \ref{thm:S.4}.

The map $\kappa^\fs$ comes from the support-preserving algebra isomorphism 
\eqref{eqn:HH}. By \cite[(3.19)]{Sol} this map respects temperedness of
representations. The argument in \cite{Sol} relies on Casselman's criteria, see 
\cite[\S 4.4]{Cas2} and \cite[\S 2.7]{Opd}. A small variation on it shows that 
$\kappa^\fs$ also preserves essential square-integrability. Both statements apply to 
$\rec_{D,er}^{-1}(R(e) \otimes \eta)$ because it is essentially square-integrable
and unitary, hence tempered. 

By the Zelevinsky classification for $\GL_n (F)$ \cite{Ze}, 
$\rec_{F,n}^{-1}(R(e) \otimes \eta)$ is a consituent of 
\[
I_{\GL_{rd} (F)^e}^{\GL_n (F)} \big( \nu_F^{(1-e)/2} \JL (\sigma) \otimes \nu_F^{(3-e)/2} 
\JL (\sigma) \otimes \cdots \otimes \nu_F^{(e-1)/2} \JL (\sigma) \big) ,
\]
where $\nu_F (g^*) = |\det (g^*)|_F$. It follows that 
$\rec_{D,er}^{-1}(R(e) \otimes \eta)$ is a constituent of
\[
I_{\GL_{r} (D)^e}^{\GL_{er} (D)} \big( \nu_D^{(1-e)/2} \sigma \otimes 
\nu_D^{(3-e)/2} \sigma \otimes \cdots \otimes \nu_D^{(e-1)/2} \sigma \big) ,
\]
where $\nu_D (g) = |\mathrm{Nrd}(g)|_F$. By our choice of $\sigma$ the infinitesimal
central character of 
$\kappa^\fs \big( I_{\GL_{r} (D)^e}^{\GL_{er} (D)} ( \sigma^{\otimes e} ) \big)$ is
$1 \in (\C^\times)^e / \fS_e$, so $\kappa^\fs \big( \rec_{D,er}^{-1}(R(e) \otimes 
\eta) \big)$ has real infinitesimal central character. Thus 
$\kappa^\fs \big( \rec_{D,er}^{-1}(R(e) \otimes \eta) \big)$ possesses all the properties 
that characterize $\mathrm{St}_e$ in $\Irr (\cH_{\tilde q} ({\widetilde W}_e))$, 
and these two representations are isomorphic.

By definition $\kappa^\fs \circ \rec_{D,er}^{-1}$ transforms a twist by 
$\psi_\sigma \circ \nu_z$ into a twist by $(z,\ldots,z)$, so
\begin{equation*}
\kappa^\fs \big( \rec_{D,er}^{-1} (\psi_\sigma \circ \nu_z \otimes R(e) \otimes \eta) 
\big) = (z,\ldots, z) \otimes \mathrm{St}_e .
\end{equation*}
We conclude that the diagram \eqref{eq:diagramGLe} commutes for all
essentially square-integrable representations.

Now we take $t$ and $\tau$ as in \eqref{eqn:descrt} and \eqref{eq:multTau}.
The construction of the Langlands correspondence for $\GL_{er}(D)$ in 
\eqref{eq:LLC2} implies
\begin{multline}\label{eq:Lttau}
\rec_{D,er}^{-1}(\varphi^\fs (t,\tau)) = \\
L \Big( I_{\GL_{b_1 r}(D) \times \cdots \times \GL_{b_h r}(D)}^{\GL_{er}(D)} \big(
\bigotimes\nolimits_{j=1}^h \rec_{D, b_j r}^{-1}(\psi_\sigma \circ \nu_j 
\otimes \Phi_{\eta,p_j}^F) \big) \Big) = \\
L \Big( I_{\prod_{j=1}^h \prod_{i=1}^{l_j} \GL_{p_{j,i} r}(D)}^{\GL_{er}(D)}
\big( \bigotimes\nolimits_{j=1}^h
\bigotimes\nolimits_{i=1}^{l_j} \rec_{D, p_{j,i} r}^{-1}(\psi_\sigma \circ \nu_j 
\otimes R(p_{j,i}) \otimes \eta) \big) \Big) .
\end{multline}
The $\GL_{p_{j,i} r}(D)$-representation $\rec_{D, p_{j,i} r}^{-1}(\psi_\sigma \circ 
\nu_j \otimes R(p_{j,i}) \otimes \eta)$ is essentially square-integrable,
so by the above we know where it goes in the diagram \eqref{eq:diagramGLe}.
Since \eqref{eqn:HH} preserves supports, it respects parabolic induction and
Langlands quotients (see \cite{Sol} for these notions in the context of affine
Hecke algebras). In particular every subgroup 
$\GL_{p_{j,i} r}(D) \subset \GL_{er}(D)$ corresponds to a subalgebra 
$\cH_{\tilde q} (\widetilde{W}_{p_{j,i}}) \subset \cH_{\tilde q} (\widetilde{W}_e)$. 
Thus the representation of $\cH_{\tilde q} (\widetilde{W}_e)$ associated to 
\eqref{eq:Lttau} via \eqref{eqn:HH} is
\begin{equation}\label{eq:LIndH}
L \Big( \Ind_{\bigotimes_{j=1}^h \bigotimes_{i=1}^{l_j} \cH_{\tilde q}
(\widetilde{W}_{p_{j,i}})}^{\cH_{\tilde q}({\widetilde W}_e)} 
\big( \bigotimes\nolimits_{j=1}^h \bigotimes\nolimits_{i=1}^{l_j} 
(z_j, \ldots,z_j) \otimes \mathrm{St}_{p_{j,i}} \big) \Big) .
\end{equation}
In view of the Zelevinsky classification for affine Hecke algebras of type 
$\GL_n$ \cite{KL,Ze}, the Langlands parameter of \eqref{eq:LIndH} is 
\[
\bigoplus\nolimits_{j=1}^h z_j \otimes \bigoplus\nolimits_{i=1}^{l_j} R(p_{j,i}) . 
\]
Now Theorem \ref{thm:ps}, in combination with the notations 
\eqref{eqn:descrt}--\eqref{eq:pj}, shows that the top map in \eqref{eq:diagramGLe} 
sends \eqref{eq:LIndH} to $(t,\tau)$. Hence the diagrams in \eqref{eq:diagramGLe} 
and in the statement of the theorem commute. 
\end{proof}

With Theorem \ref{thm:bijphi} we proved statements (1)--(5) of the conjecture in 
Section \ref{sec:statement} for inner forms of $GL_n (F)$. Statement (6), including 
properties 1--6, is also true and can proved using only the affine Hecke algebras 
\eqref{eqn:HH}. This will be a special case of results in Part 3.

\section{Symplectic groups} 
\label{sec:symp}
For symplectic and orthogonal groups, Heiermann \cite{Hei} proved that every 
Bernstein component of the category of smooth modules is equivalent to the 
module category of an affine Hecke algebra. Together with \cite[Theorem 5.4.2]{Sol} 
this proves a large part of the ABPS conjecture for such groups: 
properties 1--5 from Section \ref{sec:statement}.

For the remainder of the conjecture, additional techniques are required.
Property~6 should make use of the description of the L-packets from \cite{Moe}. 
The statements (2) and (5) involve the local Langlands correspondence for
symplectic and orthogonal groups, which is proved by Arthur in \cite{Arbook}.  
However, at present his proof is up to the stabilization of the twisted trace 
formula for the group $\GL_n(F)\rtimes\langle \epsilon_n\rangle$, where $\epsilon_n$ 
is defined in \eqref{eqn:u_n}, see \cite[\S~3.2]{Arbook}. 

In this section, in order to illustrate the fact that extended quotients can 
be easily calculated, we shall compute $T^\fs\q W^\fs$ in the case
when $\cG$ is a symplectic group and $\fs=[\cM,\sigma]_\cG$ with $\cM$ the
Levi subgroup of a maximal parabolic of $\cG$.

Let $u_{n}\in\Mat_n(F)$ be the $n\times n$ matrix defined by
\[u_{n}:=\left(\begin{matrix}
&&&&&\cdot\cr 
&&&&\cdot&\cr
&&&\cdot&&\cr
&&-1&&&\cr
&1&&&&\cr
-1&&&&&\end{matrix}\right).
\]
For $g\in\Mat_n(F)$, we denote by $\lexp{t}{g}$ the transpose matrix of
$g$, and by $\epsilon_n$ the automorphism of $\GL_n(F)$ defined by 
\begin{equation} \label{eqn:u_n}
\epsilon_n(g):=u_n\cdot\lexp{t}{g^{-1}}\cdot u_n^{-1}.
\end{equation}

We denote by $\cG=\Sp_{2n}(F)$ the symplectic group defined with respect
to the symplectic form $u_{2n}$:
\[\cG:=\left\{g\in\GL_{2n}(F)\,:\,\lexp{t}gu_{2n}g=u_{2n}\right\}.\] 
Let $k$ and $m$ two integers such that $k\ge 0$, $m\ge 0$ and $k+m=n$. 
In this section we consider a Levi subgroup $\cM$ of the following form:
\[\cM:=\left\{\left(\begin{matrix}g&&\cr
&g'&\cr
&&\epsilon_k(g)\end{matrix}\right)\,:
\,g\in\GL_k(F),\;\;g'\in \Sp_{2m}(F)\right\}\,\simeq\,
\GL_k(F)\times \cG_m.
\]
Let $\Id_n\in\Mat_n(F)$ denote the identity matrix.
We set
\begin{equation} \label{eqb:wM}
w_\cM:=\left(\begin{matrix}
&&\Id_k\cr
&\Id_{2m}&\cr
\Id_k&&\end{matrix}\right).
\end{equation}
The element $w_\cM$ has order $2$ and we have 
\begin{equation} \label{waction}
w_\cM\cdot\left(\begin{matrix}g&&\cr
&g'&\cr
&&\epsilon_k(g)\end{matrix}\right)\cdot w_\cM=
\left(\begin{matrix}\epsilon_k(g)&&\cr
&g'&\cr
&&g\end{matrix}\right).
\end{equation}
We have
\begin{equation} \label{WGM}
\Nor_\cG(\cM)/\cM=\{1,w_\cM\}.
\end{equation}

\smallskip
Let $\pi_\cM$ be an irreducible supercuspidal representation of $\cM$. We have 
$\pi_\cM=\rho\otimes\sigma$,
where $\rho$ is an irreducible supercuspidal representation of $\GL_k(F)$ and 
$\sigma$ is an irreducible supercuspidal representation of $\Sp_{2m}(F)$.
We set $\fs: =[\cM,\pi_\cM]_\cG$.

If $W^\fs=\{1\}$, then the parabolically induced representation 
$\rho\rtimes\sigma$ of $\cG$ is irreducible.

We will assume from now on that $W^\fs\ne\{1\}$. Then it follows from~(\ref{WGM}) that
\[
W^\fs=\Nor_\cG(\cM)/\cM=\{1,w_\cM\}.
\]
Let $\rho^{\epsilon_k}$ denote the representation of $\GL_k(F)$
defined by 
\[
\rho^{\epsilon_k}(g):=\rho(\epsilon_k(g)),\quad g\in\GL_k(F).
\]
Then equation~(\ref{waction}) gives: 
\begin{equation} \label{eqn:wMcontra}
\lexp{w_\cM}{(\rho\otimes\sigma)}=\rho^{\epsilon_k}\otimes\sigma.
\end{equation}
The theorem of Gel'fand and Kazhdan 
\cite[Theorem~2]{GK} says
that the representation $\rho^{\epsilon_k}$ is equivalent to $\rho^\vee$, the
contragredient representation of $\rho$.
Hence we get
\begin{equation} \label{eqn:contragredient}
\lexp{w_\cM}{(\rho\otimes\sigma)}\cong \rho^\vee\otimes\sigma.
\end{equation} 

\begin{lem} \label{Sec2:exqt}
We have
\[T^\fs\q W^\fs= T^\fs/W^\fs\sqcup\pt_1\sqcup \pt_2.\]
\end{lem}
\begin{proof}
Since the group $\Sp_{2m}(F)$ does not have non-trivial unramified characters,
we get 
\[X_{\unr}(\cM)\cong X_{\unr}(\GL_k(F)) \cong \C \Big/ \frac{2 \pi \sqrt{-1}}{\log q} 
\mathbb Z \cong \C^\times .
\]
We have
\[\Stab(\pi_\cM)=\Stab(\rho).\]
It follows that
\[
T^\fs \cong X_{\unr}(\cM)/\Stab(\pi_\cM)=
X_{\unr}(\GL_k(F))/\Stab(\rho)\cong
\C \Big/ \frac{2 \pi \sqrt{-1}}{n(\rho) \log q} \mathbb Z \cong \C^\times,
\]
where $n(\rho)$ is the order of $\Stab(\rho)$.

On the other hand, since $w_\cM$ belongs to $W^\fs$, we have 
\[
\lexp{w_\cM}{\pi_\cM}\cong\nu_\cM\,\pi_\cM,
\]
for some $\nu_\cM\in X_{\unr}(\cM)$. Hence 
\begin{equation} \label{eqn:wMend}
\lexp{w_\cM}{\pi_\cM}\cong\nu\rho\otimes\sigma,
\end{equation}
for some $\nu\in X_{\unr}(F^\times)$, where we have put
$\nu\rho:=(\nu\circ\det)\otimes\rho$.
Then it follows from \eqref{eqn:contragredient} that
\[
\rho^\vee\otimes\sigma\cong\nu\rho\otimes\sigma,
\]
that is,
\[
\nu^{1/2}\rho\otimes\sigma \cong(\nu^{1/2}\rho)^\vee \otimes\sigma.
\]
It implies that
\[
\rho\nu^{1/2}\cong(\rho\nu^{1/2})^\vee.
\]
Hence, by replacing $\rho$ by $\rho\nu^{1/2}$ if necessary, we can
always assume that the representation $\rho$ is
\emph{self-contragredient}, that is, $\rho\cong\rho^\vee$.

We shall assume from now on, that $\rho$ is self-contragredient. 
Let $\nu_k\in X_{\unr}(\GL_k(F))$. 
We shall denote by $[\nu_k]$ the image of $\nu_k$ in the quotient
$X_{\unr}(\GL_k(F))/\Stab(\rho)$.

We have $\nu_k=\nu\circ\det$ for some
$\nu\in X_{\unr}(F^\times)$. Let $g\in\GL_k(F)$. We obtain
\[
\lexp{w_\cM}{\nu}_k(g)=\nu(\det(\epsilon_k(g)))=\nu(\det(\lexp{t}{g}^{-1}))=
\nu(\det (g^{-1}))=\nu_k(g)^{-1}.
\]
Hence $\nu_k\Stab(\rho)$ is fixed by $w_\cM$ if and only if
$\nu_k^2\in \Stab(\rho)$.  On the other hand, we have
$\Cent_{W^\fs}(w_\cM)=W^\fs$.
It follows that
\[
(T^\fs)^{w_\cM}/\Cent_{W^\fs}(w_\cM)=\left\{[1],[\zeta_{n(\rho)}]\right\},
\]
where we have put
\[
\displaystyle
\zeta_{n(\rho)}:=|\det(\cdot)|_F^{\frac{\pi\sqrt{-1}}{n(\rho)\log q}},
\]
that is,
\[
(T^\fs)^{w_\cM}/\Cent_{W^\fs}(w_\cM)\simeq\{-1,1\} \subset \Cset^\times . \qedhere 
\] 
\end{proof}

Then by using \cite[Theorem~5.4.2]{Sol} (which is closely related to
Property 4 of the bijection $\mu^\fs$) we recover from
Lemma~\ref{Sec2:exqt} the well-known fact that $\nu_k\rho\rtimes\sigma$ 
reduces for exactly two unramified characters $\nu_k$.

\section{The Iwahori spherical representations of $\rG_2$}
\label{sec:G2}
Let $\cG$ be the exceptional group $\rG_2$.   Let $\fs_0 = [\cT,1]_\cG$ where 
$\cT\simeq F^\times\times F^\times$ is a maximal $F$-split torus of $\cG$.  
The following result is a special case of Theorem \ref{thm:ps} in Part 3.

\begin{thm}  
The conjecture (as stated in Section \ref{sec:statement}) is true   
for the point $\fs_0 =[\cT,1]_{\rG_2}$.
\end{thm}

This is such an illustrative example that we include some of the calculations in \cite{ABP3}.

We note that $X_{\unr}(\cT) \cong \,T$ with $T$ a maximal
torus in the Langlands dual group $G = \rG_2(\mathbb{C})$.
The Weyl group $W$ of $\rG_2$ is the dihedral group of order $12$.
The extended quotient is \[ T\q W = T/W \sqcup
\mathfrak{C}_1 \sqcup \mathfrak{C}_2 \sqcup \pt_1 \sqcup \pt_2
\sqcup \pt_3 \sqcup \pt_4 \sqcup \pt_5.\]
The flat family is
$Y_\alpha: = (1-\alpha^2 y)(x - \alpha^2 y) = 0$.
Note that $Y_{\sqrt q} = \fR$ the curve of reducibility points
in the quotient variety $T/W$. The restriction of $\pi_\alpha$
to $T\q W - T/W$ determines a finite morphism
\[
\mathfrak{C}_1 \sqcup \mathfrak{C}_2 \sqcup \pt_1 \sqcup \pt_2
\sqcup \pt_3 \sqcup \pt_4 \sqcup \pt_5 \longrightarrow Y_\alpha.\]

{\sc Example}. The fibre of the point $(q^{-1},1) \in \fR$ via the
map $\pi_{\sqrt q}$ is a set with $5$ points, corresponding to the
fact that there are $5$ smooth irreducible representations of
$\rG_2$ with infinitesimal character $(q^{-1},1)$.

The map $\pi_\alpha$ restricted to the one affine line $\mathfrak{C}_1$ 
is induced by the map $(z,1) \mapsto (\alpha z,\alpha^{-2})$, and restricted 
to the other affine line $\mathfrak{C}_2$ is induced by the map 
$(z,z) \mapsto (\alpha z,\alpha^{-1}z)$. With regard to the second map: 
the two points $(\omega/\sqrt q,\omega/\sqrt q),\, (\omega^2/\sqrt
q,\omega^2/\sqrt q)$ are distinct points in $\mathfrak{C}_2$ but
become identified via $\pi_{\sqrt q}$ in the quotient variety
$T/W$. This implies that the image $\pi_{\sqrt
q}(\mathfrak{C}_2)$ of one affine line has a
\emph{self-intersection point} in the quotient variety
$T/W$. Also, the curves $\pi_{\sqrt q}(\mathfrak{C}_1),
\pi_{\sqrt q}(\mathfrak{C}_2)$ intersect in $3$ points.  These
intersection points account for the number of distinct
constituents in the corresponding induced representations.

\part{The principal series of split reductive $p$-adic groups}

\section{Introduction to Part 3}
\label{sec:intro3}
Let $\cG$ be a connected reductive $p$-adic group, split over $F$, and
let $\mathcal T$ be a split maximal torus in $\cG$. The principal series consists 
of all $\mathcal G$-representations that are obtained with parabolic induction
from characters of $\mathcal T$. 

We denote the collection of all Bernstein components of $\cG$ of the form
$[\cT,\chi ]_\cG$ by $\mathfrak B (\cG,\cT)$ and call these the Bernstein
components in the principal series. The union 
\[
\Irr (\cG,\cT) := \bigcup_{\fs \in \mathfrak B (\cG,\cT)} \Irr (\cG )^\fs
\]
is by definition the set of all irreducible subquotients of principal series
representations of $\mathcal G$.

Let $T$ be the Langlands dual group of $\mathcal T$ and choose a uniformizer $\varpi_F 
\in F$. There is a bijection $t \mapsto \nu$ between points in $T$ and unramified 
quasicharacters of $\mathcal{T}$, determined by the relation 
\[
\nu (\lambda(\varpi_F)) = \lambda (t) 
\]
where $\lambda \in X_* (\mathcal{T}) = X^* (T)$.   
The space $\Irr (\cT )^{[\cT ,\chi]_\cT}$ is in bijection with $T$ via 
$t \mapsto \nu \mapsto \sigma \otimes \nu$. Hence Bernstein's torus $T^\fs$ is isomorphic
to $T$. However, because the isomorphism is not canonical and the action of the group
$W^\fs$ depends on it, we prefer to denote it $T^\fs$.

For each $\fs\in\mathfrak{B}(\cG,\cT)$ we will construct 
a commutative triangle of bijections 
\[ 
\xymatrix{   
& (T^\fs/\!/W^\fs)_2 \ar[dr]\ar[dl] & \\ 
\Irr (\cG)^\fs \ar[rr] & & \{\KLR\:\:\mathrm{parameters}\}^\fs / H }
\]
Here \{KLR parameters$\}^\fs$ is the set of Kazhdan--Lusztig--Reeder 
parameters associated to $\fs\in\mathfrak{B}(\cG)$ and $H$ is the stabilizer of this set  
of parameters in the dual group $G$.

In examples, $T^\fs /\!/ W^\fs$ is much simpler to directly calculate than either 
$\Irr (G)^\fs$ or \{KLR parameters$ \}^\fs$.
 
\vspace{2mm}
\noindent Let us discuss the triangle in the case that $H$ is connected.
The bijectivity of the right slanted arrow (see Section \ref{sec:affSpringer}) 
is essentially a reformulation of results of Kato \cite{Kat}. It involves Weyl
groups of possibly disconnected reductive groups. We will extend the Springer 
correspondence to such groups in Section \ref{sec:celldec}.
 
The left slanted arrow is defined (and by construction bijective) in
\cite{LuCellsIII}. The horizontal map is defined and proved 
to be a bijection in \cite{KL,R}, see Section \ref{sec:repAHA}. The results 
in \cite{R} are based on and extend those of \cite{KL}. 
Thus there are three logically independent definitions and bijectivity proofs.\\

The bijectivity of the horizontal arrow shows that the local Langlands correspondence 
is valid for each such Bernstein component $\Irr (\cG)^\fs$ and describes the 
intersections of L-packets with $\Irr (\cG)^\fs$. Once a $c$-$\Irr$ system has been 
chosen for the action of $W^ \fs$ on $T^\fs$, there is the bijection
\[
T^\fs/\!/W^\fs \longrightarrow (T^{\fs}\q W^\fs)_2 ,
\]
so the L-packets can be described in terms of the extended quotient of the first kind.
In Section \ref{sec:unip} we check that the labelling of the irreducible components  
of $T^\fs/\!/W^\fs$ predicted by the conjecture is provided by the unipotent classes of $H$.

\section{Twisted extended quotient of the second kind}  
\label{sec:teq}
Let $J$ be a finite group and let $\alpha \in \H^2(J; \Cset^{\times})$.  
Let $V$ be a finite dimensional vector space over $\Cset$.   Consider all maps
$\tau \colon J \to \GL(V)$  such that there exists a  
$\Cset^{\times}$-valued $2$-cocycle $c$ on $J$ with 
\[
\tau(j_1) \circ \tau(j_2)  = c(j_1, j_2)\,\tau(j_1j_2),\quad \quad [c] = \alpha
\]
and $\tau$ is irreducible, where $[c]$ is the class of $c$ in 
$\H^2(J; \Cset^{\times})$.   

For such a map $\tau$ let $f \colon J \to \Cset^{\times}$ be any map and consider 
the map $(f\tau)(j): =  f(j)\tau(j)$. This is again such a map 
$\tau$ with cocycle   $c\cdot f$ defined by
\[(c\cdot f)(j_1,j_2):=c(j_1,j_2)\, \frac{f(j_1)f(j_2)}{f(j_1j_2)}.\]  
Given two such maps $\tau_1, \tau_2$, an \emph{isomorphism} is an intertwining 
operator $V_1 \to V_2$.   
Note that $\tau_1 \simeq \tau_2 \Rightarrow c_1 = c_2$. Given
$\tau_1$ and $\tau_2$,  we
define $\tau_1$ to be \emph{equivalent} to $\tau_2$
if and only if there exists $f \colon J \to \Cset^{\times}$  with
$\tau_1$ isomorphic to $f \tau_2$. The set of equivalence classes of maps 
$\tau$ is denoted $\Irr^{\alpha}(J)$. 

\smallskip

Now let $\Gamma$ be a finite group with a given action on a set $X$. 
Let $\square$ be a given function which assigns to each $x \in X$ an element
$\square(x) \in \H^2(\Gamma_x;\Cset^{\times})$ where 
$\Gamma_x = \{\gamma \in \Gamma: \gamma x = x\}$.  
The function $\square$ is required to satisfy the condition
\[
\square(\gamma x) = \gamma_*\square(x), \quad \quad \forall (\gamma,x) \in \Gamma \times X
\]
where $\gamma_*\colon \Gamma_x \to \Gamma_{\gamma x}, \; 
\alpha \mapsto \gamma \alpha \gamma^{-1}$. Now define
\[
\widetilde{X}_2^{\square} = \{(x,\tau) : \tau \in \Irr^{\square(x)}(\Gamma_x)\}.
\]
We have a map $\Gamma \times \widetilde{X}_2^{\square} \to \widetilde{X}_2^{\square}$ 
and we form the \emph{twisted extended quotient of the second kind}
\[
(X\q \Gamma)_2^{\square}: = \widetilde{X}_2^{\square}/\Gamma.
\]
We will apply this construction in the following two special cases.  

\smallskip

{\bf 1.} Given two finite groups $\Gamma_1$, $\Gamma$ and a group homomorphism
$\Gamma \to \Aut(\Gamma_1)$, we can form the semidirect product  $\Gamma_1 \rtimes\Gamma$.  
Let $X = \Irr \, \Gamma_1$.   Now $\Gamma$ acts on $\Irr\, \Gamma_1$ and we get $\square$ 
as follows. Given $x \in \Irr\, \Gamma_1$ choose an irreducible representation  
$\phi: \Gamma_1 \to \GL(V)$ whose isomorphism class is $x$. 
For each $\gamma \in \Gamma_x$ consider $\phi$ twisted by $\gamma$ \ie consider 
$\phi^{\gamma}: \gamma_1 \mapsto \phi(\gamma \gamma_1 \gamma^{-1})$.
Since $\gamma \in \Gamma_x$, $\phi^{\gamma}$ is equivalent to $\phi$ \ie there exists 
an intertwining operator $T_{\gamma} : \phi \simeq \phi^{\gamma}$. 
For this operator we have
\[
T_{\gamma} \circ T_{\gamma'} = c(\gamma,\gamma')T_{\gamma \gamma'}, 
\qquad \gamma, \gamma' \in \Gamma_x
\]
and $\square(x)$ is then the class in $\H^2(\Gamma_x;\Cset^{\times})$ of $c$.   

This leads to a new formulation  of a classical theorem of Clifford.
\begin{lem} 
\label{lem:Clifford} 
We have a canonical bijection
\[
\Irr (\Gamma_1 \rtimes \Gamma) \simeq (\Irr \, \Gamma_1 \q \Gamma)_2^{\square}.
\]
\end{lem}
\begin{proof}  
The proof proceeds by comparing our construction with the classical theory of Clifford; 
for an exposition of Clifford theory, see \cite{RamRam}. 
\end{proof}

\begin{lem} 
If $\Gamma_1$ is abelian, then we have
\[
(\Irr \, \Gamma_1 \q \Gamma)_2^{\square} = (\Irr \, \Gamma_1 \q \Gamma)_2.
\]
\end{lem}
\begin{proof} The irreducible representations $\phi$ of $\Gamma_1$ are $1$-dimensional, 
and we have $\phi^{\gamma} = \phi$.   In that case each cohomology class $\natural(x)$ 
is trivial, and the  projective representations of $\Gamma_x$ which occur in the 
construction are all true representations.
\end{proof}

{\bf 2.} Given a $\Cset$-algebra $R$, a finite group $\Gamma$ and a group
homomorphism
$\Gamma \to \Aut(R)$, we can form the crossed product algebra 
\[R\rtimes\Gamma:=\{\sum_{\gamma\in\Gamma}r_\gamma\gamma\,:\,r_\gamma\in
R\},\]
with multiplication given by the distributive law and the relation
\[\gamma r=\gamma(r)\gamma,\quad\text{for $\gamma\in \Gamma$ and $r\in R$.}\]  
Let $X = \Irr \, R$.   Now $\Gamma$ acts on
$\Irr\, R$ and as above we get $\square$ for free.
Here we have
\[\widetilde
X_2^\square=\{(V,\tau)\,:\,V\in\Irr\,R,\;\tau\in\Irr^{\square(V)}(\Gamma_V)\}.\]

\begin{lem} 
\label{lem:Clifford_algebras} 
We have a canonical bijection
\[
\Irr(R \rtimes \Gamma) \simeq (\Irr \, R \q \Gamma)_2^{\square}.
\]
\end{lem}
\begin{proof}  The proof proceeds by comparing our construction with the 
theory of Clifford as stated in \cite[Theorem~A.6]{RamRam}. 
\end{proof}

\begin{notation} \label{not:rtimes} 
{\rm We shall denote by $\tau_1\rtimes\tau$ (resp. $V\rtimes\tau$) the element
of $\Irr(X\rtimes \Gamma)$ which corresponds to $(\tau_1,\tau)$ (resp.
$(V,\tau)$) by the bijection of Lemma~\ref{lem:Clifford}
(resp.~\ref{lem:Clifford_algebras}).}
\end{notation}

\section{Weyl groups of disconnected groups}
\label{sec:Wdc}

Let $M$ be a reductive complex algebraic group. Then $M$ may have a finite number 
of connected components,  $M^0$ is the identity component of $M$, and $\cW^{M^0}$ 
is the Weyl group of $M^0$:
\[
\cW^{M^0}: = \Nor_{M^0}(T)/T
\]
where $T$ is a maximal torus of $M^0$.   We will need the analogue of the Weyl 
group for the possibly disconnected group $M$. 

\begin{lem} \label{lem:disconnected}
Let $M,\, M^0,\, T$ be as defined above. Then we have 
\[
\Nor_M(T)/T \cong \cW^{M^0} \rtimes \pi_0(M).
\]
\end{lem}
\begin{proof}
The group $\cW^{M^0}$ is a normal subgroup of $ \Nor_M(T)/T$. Indeed,
let $n\in\Nor_{M^0}(T)$ and let $n'\in\Nor_M(T)$, then $n'nn^{\prime-1}$
belongs to $M^0$ (since the latter is normal in $M$) and normalizes $T$, that is,
$n'nn^{\prime-1}\in\Nor_{M^0}(T)$. On the other hand,
$n'(nT)n^{\prime-1}=n'nn^{\prime-1}(n'Tn^{\prime-1})=n'nn^{\prime-1}T$.

Let $B$ be a Borel subgroup of $M^0$ containing $T$.   
Let $w\in \Nor_M(T)/T$. Then $wBw^{-1}$ is a Borel subgroup of $M^0$
(since, by definition, the Borel subgroups of an algebraic group are 
the maximal closed connected solvable subgroups). Moreover, $wBw^{-1}$  
contains $T$. 
In a connected reductive algebraic group, the intersection of two Borel 
subgroups always contains a maximal torus and the two Borel subgroups are 
conjugate by a element of the normalizer of that torus. Hence $B$ and
$wBw^{-1}$ are conjugate by an element $w_1$ of $\cW^{M^0}$.
It follows that $w_1^{-1}w$ normalises $B$. Hence
\[w_1^{-1}w\in \Nor_M(T)/T \cap \Nor_{M}(B)=\Nor_{M}(T,B)/T,\] 
that is, \[
\Nor_M(T)/T = \cW^{M^0}\cdot(\Nor_M(T,B)/T).\] 
Finally, we have
\[\cW^{M^0}\cap(\Nor_M(T,B)/T)=\Nor_{M^0}(T,B)/T=\{1\},\] 
since $\Nor_{M^0}(B)=B$ and $B\cap \Nor_{M^0}(T)=T$. This proves (1).

Now consider the following map:
\begin{align}\label{MM}
\Nor_{M}(T,B)/T\to M/M^0\quad\quad mT\mapsto mM^0.
\end{align}
It is injective. Indeed, let $m,m'\in\Nor_{M}(T,B)$ such that
$mM^0=m'M^0$. Then $m^{-1}m'\in M^0\cap\Nor_{M}(T,B)=\Nor_{M^0}(T,B)=T$
(as we have seen above). Hence $mT=m'T$.

On the other hand, let $m$ be an element in $M$. Then $m^{-1}Bm$ is a
Borel subgroup of $M^0$, hence there exists $m_1\in M^0$ such that
$m^{-1}Bm=m_1^{-1}Bm_1$. It follows that $m_1m^{-1}\in\Nor_M(B)$. Also
$m_1m^{-1}Tmm_1^{-1}$ is a torus of $M^0$ which is contained in 
$m_1m^{-1}Bmm_1^{-1}=B$. Hence $T$ and $m_1m^{-1}Tmm_1^{-1}$ are conjugate
in $B$: there is $b\in B$ such that $m_1m^{-1}Tmm_1^{-1}=b^{-1}Tb$. Then 
$n:=bm_1m^{-1}\in\Nor_M(T,B)$. It gives $m=n^{-1}bm_1$. Since $bm_1\in
M^0$, we obtain $mM^0=n^{-1}M^0$. Hence the map  (\ref{MM}) is surjective.
\end{proof}

Let $G$ be a connected complex reductive group and let $T$ be a maximal torus in $G$.
The Weyl group of $G$ is denoted $\cW^G$. 

\begin{lem} \label{lem:centrals} 
Let $A$ be a subgroup of $T$ and write $M = \Z_G (A)$. Then the isotropy subgroup of $A$ in 
$\cW^G$ is
\[
\cW^G_A = \Nor_M (T) / T \cong \cW^{M^0} \rtimes \pi_0(M) .
\]
In case that the group $M$ is connected, $\cW^G_A$ is the Weyl group of $M$.
\end{lem}
\begin{proof} 
Let $R(G,T)$ denote the root system of $G$. According to \cite[\S~4.1]{SpringerSteinberg},
the group $M = \Cent_G(A)$ is the reductive subgroup of $G$ generated 
by $T$ and those root groups $U_\alpha$ for which $\alpha \in R(G,T)$ has 
trivial restriction to $A$ together with those Weyl group representatives 
$n_w \in \Nor_G(T) \; (w \in \cW^G)$ for which $w(t) = t$ for all $t \in A$.
This shows that $\cW^G_A = \Nor_M (T) / T$, which by Lemma \ref{lem:disconnected}
is isomorphic to $\cW^{M^0} \rtimes \pi_0(M)$.

Also by \cite[\S~4.1]{SpringerSteinberg}, the identity component of $M$ is generated 
by $T$ and those root groups $U_\alpha$ for which $\alpha$ has trivial restriction to $A$. 
Hence the Weyl group $\mathcal W^{M^\circ}$ is the normal subgroup of $\cW^G_A$ generated by 
those reflections $s_\alpha$ and 
\[
\cW^G_A / \cW^{M^\circ} \cong M / M^\circ .
\]
In particular, if $M$ is connected then $\cW^G_A$ is the Weyl group of $M$.
\end{proof}

In summary, for $t \in T$ such that $M = \Z_G (t)$ we have 
\begin{align}
(T\q \cW^G)_2  & =  \{(t,\sigma) : t \in T, \sigma \in
\Irr(\cW_t^G)\}/\cW^G \label{EXT1}\\
\Irr\, \cW^G_t  & =  (\Irr \, \cW^{M^0}\q
\pi_0(M))_2^{\square} \label{EXT2}
\end{align}

\section{An extended Springer correspondence} 
\label{sec:celldec}
Let $M^\circ$ be a connected reductive complex group.
We take $x \in M^\circ$ unipotent and we abbreviate 
\begin{align}
A_x: = \pi_0 (\Cent_{M^0}(x)).
\end{align}
Let $x \in M^\circ$ be unipotent, $\mathcal B^x = \mathcal B^x_{M^\circ}$ the variety of 
Borel subgroups of $M^\circ$ containing $x$. All the irreducible components of 
$\mathcal B^x$ have the same dimension $d(x)$ over $\Rset$, see \cite[Corollary 3.3.24]{CG}.
Let $H_{d(x)} (\mathcal B^x,\C)$ be its top homology, let $\rho$ be an irreducible 
representation of $A_x$ and write
\begin{equation} \label{eq:tauxrho}
\tau (x,\rho) = \mathrm{Hom}_{A_x} \big( \rho, H_{d(x)}(\mathcal B^x ,\C) \big) .
\end{equation}
We call $\rho \in \Irr (A_x)$ geometric if $\tau (x,\rho) \neq 0$.
The Springer correspondence yields a one-to-one correspondence
\begin{equation} \label{eqn:Springercor}
(x,\rho) \mapsto \tau(x,\rho)
\end{equation}
between the set of $M^0$-conjugacy classes of pairs $(x,\rho)$ formed by a
unipotent element $x \in M^0$ and an irreducible geometric representation 
$\rho$ of $A_x$, and the equivalence classes of irreducible representations 
of the Weyl group $\cW^{M^0}$.

\begin{rem} \label{rem:Springer}
The Springer correspondence which employ here
sends the trivial unipotent class to the trivial $\cW^{M^\circ}$-representation
and the regular unipotent class to the sign representation.
It coincides with the correspondence constructed by Lusztig by means of intersection
cohomology. The difference with Springer's construction via a reductive group
over a field of positive characteristic consists of tensoring with
the sign representation of $\cW^{M^0}$, see \cite{Hot}.
\end{rem}

Choose a set of simple reflections for $\mathcal W^{M^\circ}$ and let $\Gamma$ be a group 
of automorphisms of the Coxeter diagram of $W$. Then $\Gamma$ acts on $\mathcal W^{M^\circ}$
by group automorphisms, so we can form the semidirect product $\mathcal W^{M^\circ} \rtimes \Gamma$. 
Furthermore $\Gamma$ acts on $\Irr (\mathcal W^{M^\circ})$, by $\gamma \cdot \tau = 
\tau \circ \gamma^{-1}$. The stabilizer of $\tau \in \Irr (\mathcal W^{M^\circ})$ is denoted 
$\Gamma_\tau$. As described in Section \ref{sec:teq}, Clifford theory for 
$\mathcal W^{M^\circ} \rtimes \Gamma$ produces a 2-cocycle 
$\natural(\tau) : \Gamma_\tau \times \Gamma_\tau \to \C^\times$.

We fix a Borel subgroup $B_0$ of $M^\circ$ containing $T$ and let $\Delta (B_0,T)$ be the
set of roots of $(M^\circ,T)$ that are simple with respect to $B_0$. We may and will assume
that this agrees with the previously chosen simple reflections in $\mathcal W^{M^\circ}$.
In every root subgroup $U_\alpha$ with $\alpha \in \Delta (B_0,T)$ we pick a nontrivial 
element $u_\alpha$. The data $(M^\circ,T,(u_\alpha)_{\alpha \in \Delta (B_0,T)})$ are
called a pinning of $M^\circ$. The action of $\gamma \in \Gamma$ on the Coxeter diagram of 
$\mathcal W^{M^\circ}$ lifts uniquely to an action of $\gamma$ on $M^\circ$ which preserves 
the pinning. In this way we construct the semidirect product $M := M^\circ \rtimes \Gamma$.
By Lemma \ref{lem:centrals} we may identify $\cW^M$ with $\cW^{M^\circ} \rtimes \Gamma$.

For any unipotent element $x \in M^\circ$ the centralizer $\Cent_M (x)$ acts on the variety
$\mathcal B^x_{M^\circ}$. We say that an irreducible representation $\rho_1$ of 
$\Cent_M (x)$ is geometric if it appears in $H_{d(x)}(\mathcal B^x_{M^\circ} ,\C)$. 
Notice that this condition forces $\rho_1$ to factor through the component group 
$\pi_0 (\Cent_M (x))$. 

\begin{prop}\label{prop:S.2}
The class of $\natural(\tau)$ in $H^2 (\Gamma_\tau , \C^\times)$ is trivial for all 
$\tau \in \Irr (\mathcal W^{M^\circ})$. Hence there are natural bijections between the
following sets:
\begin{align*}
& \Irr (\mathcal W^{M^\circ} \rtimes \Gamma) = \Irr (\mathcal W^M ) ,\\
& \big( \Irr (\mathcal W^{M^\circ}) /\!/ \Gamma \big)_2 ,\\ 
& \big\{ (x,\rho_1) \mid x \in M^\circ \text{ unipotent} , \rho_1 \in 
\Irr \big( \pi_0 (\Cent_M (x)) \big) \text{ geometric} \big\} / M .
\end{align*}
\end{prop}
\begin{proof}
There are various ways to construct the Springer correspondence for $\mathcal W^{M^\circ}$, 
for the current proof we use the method with Borel--Moore homology. 
Let $\mathcal \Z_{M^\circ}$ be the Steinberg variety of $M^\circ$ and 
$H_{\top} (\mathcal \Z_{M^\circ})$ its homology in the top degree 
\[
2 \dim_\C \mathcal \Z_{M^\circ} = 4 \dim_\C \mathcal B_{M^\circ} = 
4 (\dim_\C M^\circ - \dim_\C B_0) ,
\]
with rational coefficients. We define a natural algebra isomorphism 
\begin{equation}\label{eq:S.5}
\Q [\mathcal W^{M^\circ}] \to H_{\top}(\mathcal \Z_{M^\circ})  
\end{equation}
as the composition of \cite[Theorem 3.4.1]{CG} and a twist by the sign representation 
of $\Q [\mathcal W^{M^\circ}]$. By \cite[Section 3.5]{CG} the action of $\mathcal W^{M^\circ}$ 
on $H_* (\mathcal B^x ,\C)$ (as defined by Lusztig) corresponds to the convolution 
product in Borel--Moore homology. 

Since $M^\circ$ is normal in $M$, the groups $\Gamma ,M$ and $M / \Z(M)$ act on the 
Steinberg variety $\mathcal \Z_{M^\circ}$ via conjugation. The induced action of the 
connected group $M^\circ$ on $H_{\top}(\mathcal \Z_{M^\circ})$ is trivial, and it easily 
seen from \cite[Section 3.4]{CG} that the action of $\Gamma$ on $H(\mathcal \Z_{M^\circ})$ 
makes \eqref{eq:S.5} $\Gamma$-equivariant.

The groups $\Gamma ,M$ and $M / \Z(M)$ also act on the pairs $(x,\rho)$ 
and on the varieties of Borel subgroups, by
\begin{align*}
& \mathrm{Ad}_m (x,\rho) = (m x m^{-1}, \rho \circ \mathrm{Ad}_m^{-1}) , \\
& \mathrm{Ad}_m : \mathcal B^x \to \mathcal B^{m x m^{-1}} ,\; B \mapsto m B m^{-1} .
\end{align*}
Given $m \in M$, this provides a linear bijection $H_* (\mathrm{Ad}_m)$ : 
\[
\mathrm{Hom}_{A_x}(\rho, H_* (\mathcal B^x ,\C)) \to
\mathrm{Hom}_{A_{mxm^{-1}}}(\rho \circ \mathrm{Ad}_m^{-1}, H_* (\mathcal B^{mxm^{-1}} ,\C)) .
\]
The convolution product in Borel--Moore homology is compatible with these 
$M$-actions so, as in \cite[Lemma 3.5.2]{CG}, the following diagram commutes 
for all $h \in H_{\top}(\mathcal \Z_{M^\circ})$:
\begin{equation}\label{eq:S.6}
\begin{array}{ccc}
H_* (\mathcal B^x ,\C) & \xrightarrow{\; h \;} & H_* (\mathcal B^x ,\C) \\
\downarrow \scriptstyle{H_* (\mathrm{Ad}_m)} & & 
\downarrow \scriptstyle{H_* (\mathrm{Ad}_m)} \\
H_* (\mathcal B^{mxm^{-1}} ,\C) & \xrightarrow{m \cdot h} & 
H_* (\mathcal B^{mxm^{-1}} ,\C) .
\end{array}
\end{equation}
In case $m \in M^\circ \gamma$ and $m \cdot h$ corresponds to $w \in \mathcal W^{M^\circ}$, 
the element $h \in H(\mathcal \Z_{M^\circ})$ corresponds to $\gamma^{-1}(w)$, 
so \eqref{eq:S.6} becomes
\begin{equation}\label{eq:S.7}
H_* (\mathrm{Ad}_m) \circ \tau (x,\rho) (\gamma^{-1}(w)) = 
\tau (mxm^{-1}, \rho \circ \mathrm{Ad}_m^{-1})(w) \circ H_* (\mathrm{Ad}_m) .
\end{equation}
Denoting the $M^\circ$-conjugacy class of $(x,\rho)$ by $[x,\rho]_{M^\circ}$, we can write
\begin{align}\label{eq:S.9}
\Gamma_{\tau (x,\rho)} & = \{ \gamma \in \Gamma \mid \tau (x,\rho) 
\circ \gamma^{-1} \cong \tau (x,\rho) \} \\
\nonumber & = \{ \gamma \in \Gamma \mid [\mathrm{Ad}_\gamma 
(x,\rho)]_{M^\circ} = [x,\rho]_{M^\circ} \} =: \Gamma_{[x,\rho]_{M^\circ}} .
\end{align}
This group fits in an exact sequence
\begin{equation}\label{eq:S.4}
1 \to \pi_0 \big( \Z_{M^\circ} (x,\rho) / \Z(M^\circ) \big) \to \pi_0 \big( \Z_M (x,\rho) / 
\Z(M^\circ) \big) \to \Gamma_{[x,\rho]_{M^\circ}} \to 1 .
\end{equation}
Assume for the moment that \eqref{eq:S.4} has a splitting 
\[
s : \Gamma_{[x,\rho]_{M^\circ}} \to \pi_0 \big( \Z_M (x,\rho) / \Z(M^\circ) \big) . 
\]
By homotopy invariance in Borel--Moore homology $H_* (\mathrm{Ad}_z) = 
\mathrm{id}_{H_* (\mathcal B^x ,\C)}$ for any $z \in \Z_{M^\circ}(x,\rho)^\circ 
\Z(M^\circ)$, so $H_* (\mathrm{Ad}_m)$ is well-defined for \\
$m \in \pi_0 \big( \Z_M (x,\rho) / \Z(M^\circ) \big)$. In particular we obtain for every 
$\gamma \in \Gamma_{\tau (x,\rho)} = \Gamma_{[x,\rho]_{M^\circ}}$ a linear bijection
\[
H_* (\mathrm{Ad}_{s(\gamma )}) : \mathrm{Hom}_{A_x}(\rho, H_{d(x)}
(\mathcal B_x ,\C)) \to \mathrm{Hom}_{A_x}(\rho, H_ (\mathcal B_x ,\C)) ,
\]
which by \eqref{eq:S.7} intertwines the $\mathcal W^{M^\circ}$-representations 
$\tau (x,\rho)$ and $\tau (x,\rho) \circ \gamma^{-1}$. By construction
\[
H_* (\mathrm{Ad}_{s(\gamma )}) \circ H_* (\mathrm{Ad}_{s(\gamma' )}) =
H_* (\mathrm{Ad}_{s(\gamma \gamma')}) .
\]
This proves the triviality of the 2-cocycle $\natural(\tau) = \natural (\tau (x,\rho))$, 
up to the existence of a splitting of \eqref{eq:S.4}. Thus it remains to show the following.

\begin{lem}\label{lem:S.3}
Let $\rho \in \Irr (\pi_0 (\Z_{M^\circ}(x)))$ and write $\Z_M (x,\rho) =
\{ m \in \Z_M (x) | \\ \rho \circ \mathrm{Ad}_m^{-1} \cong \rho \}$. 
The following short exact sequence splits:
\[
1 \to \pi_0 \big( \Z_{M^\circ}(x,\rho) / \Z(M^\circ) \big) \to 
\pi_0 \big( \Z_M (x,\rho) / \Z(M^\circ) \big) \to \Gamma_{[x,\rho]_{M^\circ}} \to 1 .
\] 
\end{lem}
\begin{proof}
First we ignore $\rho$. According to the classification of unipotent orbits in complex
reductive groups \cite[Theorem 5.9.6]{Carter} we may assume that $x$ is distinguished unipotent
in a Levi subgroup $L \subset M^\circ$ that contains $T$. Let $\mathcal D(L)$ be the derived 
subgroup of $L$ and define 
\[
L' := \Z_{M^\circ}(\mathcal D (L)) (T \cap \mathcal D (L)) = \Z_{M^\circ}(\mathcal D (L)) T .
\]
Choose Borel subgroups $B_L \subset L$ and $B'_L \subset L'$ such that $x \in B_L$ and
$T \subset B_L \cap B'_L$.
Let $[x]_{M^\circ}$ be the $M^\circ$-conjugacy class of $x$ and 
$\Gamma_{[x]_{M^\circ}}$ its stabilizer in $\Gamma$. Any $\gamma \in 
\Gamma_{[x]_{M^\circ}}$ must also stabilize the $M^\circ$-conjugacy class of $L$,
and $T = \gamma (T) \subset \gamma (L)$, so there exists a $w_1 \in \mathcal W^{M^\circ}$ with 
$w_1 \gamma (L) = L$. Adjusting $w_1$ by an element of $W(L,T) \subset \mathcal W^{M^\circ}$, 
we can achieve that moreover 
$w_1 \gamma (B_L) = B_L$. Then $w_1 \gamma (L') = L'$, so we can find a unique 
$w_2 \in W (L',T) \subset \mathcal W^{M^\circ}$ with $w_2 w_1 \gamma (B'_L) = B'_L$. 
Notice that the centralizer of $\Phi (B_L,T) \cup \Phi (B'_L,T)$ in $\mathcal W^{M^\circ}$ 
is trivial, because it is generated by
reflections and no root in $\Phi (M^\circ,T)$ is orthogonal to this set of roots.
Therefore the above conditions completely determine $w_2 w_1 \in \mathcal W^{M^\circ}$.

The element $w_1 \gamma \in \mathcal W^{M^\circ} \rtimes \Gamma$ acts on $\Delta (B_L,T)$ 
by a diagram automorphism, so upon choosing $u_\alpha \in U_\alpha \setminus \{1\}$ for 
$\alpha \in \Delta (B_L,T)$, it can be represented by a unique element 
\[
\overline{w_1 \gamma} \in \mathrm{Aut} \big( \mathcal D (L),T,
(u_\alpha )_{\alpha \in \Delta (B_L,T)} \big) .
\]
The distinguished unipotent class of $x \in L$ is determined by its Bala--Carter diagram.
The classification of such diagrams \cite[\S 5.9]{Carter} shows that there exists an element
$\bar x$ in the same class as $x$, such that $\mathrm{Ad}_{\overline{w_1 \gamma}}(\bar x) = 
\bar x$. We may just as well assume that we had $\bar x$ instead of $x$ from the start,
and that $\overline{w_1 \gamma} \in \Z_M (x)$. Clearly we can find a representative 
$\overline{w_2}$ for $w_2$ in $\Z_M (x)$, so we obtain
\[
\overline{w_2} \, \overline{w_1 \gamma} \in \Z_M (x) \cap \Nor_M (T) \quad \text{and} \quad
w_2 w_1 \gamma \in \displaystyle{\frac{\Z_M (x) \cap \Nor_M (T)}{\Z(M^\circ) \, T}} . 
\]
Since $w_2 w_1 \in \mathcal W^{M^\circ}$ is unique,
\[
s : \Gamma_{[x]_{M^\circ}} \to \displaystyle{\frac{\Z_M (x) \cap 
\Nor_M (T)}{\Z(M^\circ) \, T}} ,\; \gamma \mapsto w_2 w_1 \gamma 
\]
is a group homomorphism. 

We still have to analyse the effect of $\Gamma_{[x]_{M^\circ}}$ on $\rho \in \Irr (A_x)$. 
Obviously composing with $\mathrm{Ad}_m$ for $m \in \Z_{M^\circ}(x)$ does not change 
the equivalence class of any representation of $A_x = \pi_0 (\Z_{M^\circ}(x))$. 
Hence $\gamma \in \Gamma_{[x]_{M^\circ}}$ stabilizes $\rho$ if and only if any lift 
of $\gamma$ in $\Z_M (x)$ does. This applies in particular to 
$\overline{w_2} \, \overline{w_1 \gamma}$, and therefore
\[
s ( \Gamma_{[x,\rho]_{M^\circ}}) \subset 
\big(\Z_M (x,\rho) \cap \Nor_M (T)\big) \big/ \big( \Z(M^\circ) \, T \big) .
\]
Since the torus $T$ is connected, $s$ determines a group homomorphism from
$\Gamma_{[x,\rho]_{M^\circ}}$ to $\pi_0 \big( \Z_M (x,\rho) / \Z(M^\circ) \big)$, 
which is the required splitting.
\end{proof}

This also finishes the proof of the first part of Proposition \ref{prop:S.2}. Now the 
bijection between of $\Irr (\mathcal W^M)$ and $\big( \Irr (\mathcal W^{M^\circ}) /\!/ 
\Gamma \big)_2$ follows from Clifford theory, see Section \ref{sec:teq}. 
By \eqref{eq:S.9} and Lemma \ref{lem:S.3} the geometric representations of $\pi_0 (\Z_M (x))$ 
are precisely the representations $\rho \rtimes \sigma$ with $\rho \in \Irr (A_x)$ geometric 
and $\sigma \in \Irr (\Gamma_{[x,\rho]_{M^\circ}})$. This provides the bijection between
$\big( \Irr (\mathcal W^{M^\circ}) /\!/ \Gamma \big)_2$ and the $M$-conjugacy classes of
pairs $(x,\rho_1)$.
\end{proof}

There is natural partial order on the unipotent classes in $M$:
\[
\cO < \cO' \quad \text{when} \quad \overline{\cO} \subsetneq \overline{\cO'} .
\]
Let $\cO_x \subset M$ be the class containing $x$.
We transfer this to partial order on our extended Springer data by defining
\begin{equation}\label{eq:S.32}
(x,\rho_1) < (x',\rho'_1) \quad \text{when} \quad 
\overline{\cO_x} \subsetneq \overline{\cO_{x'}} .
\end{equation}
We will use it to formulate a property of the composition series of some 
$\cW^M$-representations that will appear later on.

\begin{lem}\label{lem:S.6}
Let $x \in M$ be unipotent and let $\rho \rtimes \sigma$ 
be a geometric irreducible representation of $\pi_0 (\Z_M (x))$.
There exist multiplicities \\ $m_{x,\rho \rtimes \sigma ,x',\rho' \rtimes \sigma'} 
\in {\mathbb Z}_{\geq 0}$ such that
\begin{multline*}
\mathrm{Ind}_{W \rtimes \Gamma_{[x,\rho]_{M^\circ}}}^{W \rtimes \Gamma} \big(
\Hom_{A_x} \big( (\rho,H_* (\mathcal B^x,\C) \big) \otimes \sigma \big) \cong \\
\tau (x,\rho) \rtimes \sigma \oplus 
\bigoplus_{(x',\rho' \rtimes \sigma') > (x,\rho \rtimes \sigma)} 
m_{x,\rho \rtimes \sigma ,x',\rho' \rtimes \sigma'}\, \tau (x',\rho') \rtimes \sigma' .
\end{multline*}
\end{lem}
\begin{proof}
Consider the vector space $\Hom_{A_x} \big( (\rho,H_* (\mathcal B^x,\C) \big)$
with the $\cW^{M^\circ}$-action coming from \eqref{eq:S.5}. 
The proof of Proposition \ref{prop:S.2} remains valid for these representations. 
By \cite[Theorem 4.4]{BM} (attributed to Borho and MacPherson) there exist 
multiplicities $m_{x,\rho,x',\rho'} \in {\mathbb Z}_{\geq 0}$ such that
\begin{equation}
\Hom_{A_x} \big( (\rho,H_* (\mathcal B^x,\C) \big) \cong
\tau (x,\rho) \oplus \bigoplus_{(x',\rho') > (x,\rho)} m_{x,\rho ,x',\rho'}\, \tau (x',\rho') .
\end{equation}
By \eqref{eq:S.9} and \eqref{eq:S.7} $\Gamma_{[x,\rho]_{M^\circ}}$ also
stabilizes the $\tau (x',\rho')$ with $m_{x,\rho,x',\rho'} > 0$, and by Lemma \ref{lem:S.3}
the associated 2-cocycles are trivial. It follows that 
\begin{multline}\label{eq:S.31}
\mathrm{Ind}_{W \rtimes \Gamma_{[x,\rho]_{M^\circ}}}^{W \rtimes \Gamma} \big(
\Hom_{A_x} \big( (\rho,H_* (\mathcal B^x,\C) \big) \otimes \sigma \big) \cong \\
\tau (x,\rho) \rtimes \sigma \oplus 
\bigoplus_{(x',\rho') > (x,\rho)} m_{x,\rho ,x',\rho'} 
\mathrm{Ind}_{W \rtimes \Gamma_{[x,\rho]_{M^\circ}}}^{W \rtimes \Gamma} 
\big( \tau (x',\rho') \otimes \sigma \big) .
\end{multline}
Decomposing the right hand side into irreducible representations then gives the
statement of the lemma.
\end{proof}

\section{The Langlands parameter $\Phi$} 
\label{sec:Lp}
Let $\mathbf{W}_F$ denote the Weil group of $F$, let $\mathbf{I}_F$ be the inertia
subgroup of $\mathbf{W}_F$. 
Let $\mathbf{W}_F^{\der}$ denote the closure of the commutator subgroup of $\mathbf{W}_F$, 
and write $\mathbf{W}_F^{\ab} = \mathbf{W}_F/\mathbf{W}^{\der}_F$. 
The group of units in $\mathfrak{o}_F$ will be denoted $\fo_F^\times$.

We recall the Artin reciprocity map $\mathbf{a}_F : \mathbf{W}_F \to F^{\times}$ 
which has the following properties (local class field theory):
\begin{enumerate}
\item The map $\mathbf{a}_F$ induces a topological isomorphism 
$\mathbf{W}^{\ab}_F \simeq F^{\times}$.
\item An element $x \in \mathbf{W}_F$ is a geometric Frobenius if and only if 
$\mathbf{a}_F(x)$ is a prime element $\varpi_F$ of $F$.
\item We have $\mathbf{a}_F(\mathbf{I}_F) = \fo_F^\times$.
\end{enumerate}
We now consider the  principal series of $\cG$. We recall that  
$\mathcal{G}$ denotes a connected reductive split $p$-adic group with 
maximal split torus $\mathcal{T}$, and that
$G ,\;T$ denote the Langlands dual groups of $\mathcal{G} ,\; \mathcal{T}$. 
Next, we consider conjugacy classes in $G$ of continuous morphisms
\[
\Phi\colon \mathbf{W}_F\times \SL_2 (\Cset) \to G
\] 
which are rational on $\SL_2 (\Cset)$ and such that $\Phi(\mathbf{W}_F)$ 
consists of semisimple elements in $G$. 

Let $B_2$ be the upper triangular Borel subgroup in $\SL_2 (\Cset)$.
Let $\mathcal B^{\Phi (\mathbf{W}_F \times B_2)}$ denote the variety of Borel 
subgroups of $G$ containing $\Phi(\mathbf{W}_F \times B_2)$.
The variety $\mathcal B^{\Phi (\mathbf{W}_F \times B_2)}$ is non-empty if and 
only if $\Phi$ factors through $\mathbf W_F^{\ab}$, see \cite[\S 4.2]{R}.  
In that case, we view the domain of $\Phi$ to be $F^{\times} \times \SL_2 (\Cset)$: 
\[
\Phi\colon F^{\times} \times \SL_2 (\Cset) \to G.
\]
In this section we will build such a continuous morphism $\Phi$ from $\fs$ and data 
coming from the extended quotient of second kind. In Section \ref{sec:Borel} we show
how such a Langlands parameter $\Phi$ can be enhanced with a parameter $\rho$.

Throughout this article, a Frobenius element $\Frob_F$ has been chosen and fixed.  
This determines a uniformizer $\varpi_F$ via the equation $\mathbf{a}_F(\Frob_F) = \varpi_F$.  
That in turn gives rise to a group isomorphism $\fo_F^\times \times \mathbb Z \to F^\times$,
which sends $1 \in \mathbb Z$ to $\varpi_F$.
Let $\cT_0$ denote the maximal compact subgroup of $\cT$. As the latter is $F$-split,
\begin{equation}\label{eq:cT0}
\cT \cong F^\times \otimes_{\mathbb Z} X_* (\cT) \cong (\fo_F^\times \times \mathbb Z)
\otimes_{\mathbb Z} X_* (\cT) = \cT_0 \times X_* (\cT) .
\end{equation}
Because $\cW$ does not act on $F^\times$, these isomorphisms are $\cW$-equivariant if
we endow the right hand side with the diagonal $\cW$-action.
Thus \eqref{eq:cT0} determines a $\cW$-equivariant isomorphism of character groups 
\begin{equation}\label{split}
\Irr (\cT) \cong \Irr (\cT_0) \times \Irr (X_* (\cT)) = \Irr (\cT_0) \times X_{\unr}(\cT) .
\end{equation}

\begin{lem}\label{lem:cBernstein} 
Let $\chi$ be a character of $\cT$, and let 
$[\cT,\chi]_{\cG}$ be the inertial class of the pair $(\cT,\chi)$ as in \S 3.   Let 
\begin{align}\label{artin}
\fs = [\cT,\chi]_{\cG}.
\end{align}
Then $\fs$ determines, and is determined by, the $\cW$-orbit of a smooth morphism
\[
c^\fs \colon \fo_F^\times \to T.
\]
\end{lem}
\begin{proof}
There is a natural isomorphism
\[
\Irr (\cT) = \Hom (F^\times \otimes_{\Zset} X_* (\cT),\C^\times) \cong
\Hom (F^\times ,\C^\times \otimes_\Zset X^* (\cT)) = \Hom (F^\times ,T) . 
\]
Together with \eqref{split} we obtain isomorphisms
\begin{align*}
& \Irr (\cT_0) \cong \Hom (\fo_F^\times ,T) , \\
& X_{\unr}(\cT) \cong \Hom (\Zset ,T) = T . 
\end{align*}
Let $\hat \chi \in \Hom (F^\times ,T)$ be the image of $\chi$ under these isomorphisms. 
By the above the restriction of $\hat \chi$ to $\fo_F^\times$ is not disturbed by 
unramified twists, so we take that as $c^\fs$. Conversely, by \eqref{split} $c^\fs$ 
determines $\chi$ up to unramified twists. Two elements of $\Irr (\cT)$ are 
$\cG$-conjugate if and only if they
are $\cW$-conjugate so, in view of \eqref{artin}, the $\cW$-orbit
of the $c^\fs$ contains the same amount of information as $\fs$.
\end{proof}

Let $H = \Cent_G (\text{im} \, c^\fs)$ and let $M = \Z_H (t)$ for some $t \in T$.
Recall that a unipotent element $x\in M^0$ is said to be
\emph{distinguished} if the connected center $\Z_{M^0}^0$ of $M^0$ is 
a maximal torus of $\Cent_{M^0}(x)$. Let $x\in M^0$ unipotent. 
If $x$ is not distinguished, then there is a Levi subgroup $L$ of $M^0$ 
containing $x$ and such that $x\in L$ is distinguished. 

Let $X\in \text{Lie }M^0$ such that $\exp(X) = x$.
A cocharacter $h \colon \Cset^\times\to M^0$ is said to be \emph{associated to} $x$ if 
\[
\Ad(h(t))X=t^2 X \quad\text{for each $t\in\Cset^\times$},
\]
and if the image of $h$ lies in the derived group of some Levi subgroup
$L$ for which $x\in L$ is distinguished (see \cite[Rem.~5.5]{J}
or \cite[Rem.2.12]{FR}). 

A cocharacter associated to a unipotent element $x\in M^0$ is not unique.
However, any two cocharacters associated to a given $x\in M^0$ are
conjugate under elements of $\Cent_{M^0}(x)^0$ (see for instance 
\cite[Lem.~5.3]{J}).

\smallskip

We work with the Jacobson--Morozov theorem \cite[p. 183]{CG}.  Let
$\matje{1}{1}{0}{1}$ be the standard unipotent matrix in $\SL_2(\Cset)$ and let 
$x$ be a unipotent element in $M^0$. There exist rational homomorphisms 
\begin{equation} \label{eqn:gamt}
\gamma \colon \SL_2 (\Cset) \to M^0 \quad \text{with} \quad \gamma \matje{1}{1}{0}{1} = x ,
\end{equation}
see \cite[\S 3.7.4]{CG}. Any two such homomorphisms  $\gamma$ are conjugate by
elements of $\Z_{M^\circ}(x)$. 

For $\alpha \in \Cset^{\times}$ we define the following matrix in $\SL_2 (\Cset)$:
\[
Y_{\alpha} = \matje{\alpha}{0}{0}{\alpha^{-1}} .
\]
Then each $\gamma$ as above determines a cocharacter 
$h\colon \Cset^\times\to M^0$ by setting 
\begin{equation}\label{hh}
h(\alpha):=\gamma(Y_{\alpha})\quad\text{for } \alpha\in\Cset^\times .
\end{equation}
Each cocharacter $h$ obtained in this way is associated to $x$, see \cite[Rem.~5.5]{J}
or \cite[Rem.2.12]{FR}. 
Hence each two such cocharacters are conjugate under $\Cent_{M^0}(x)^0$. 

\begin{lem} \label{lem:cocharindep}
Each cocharacter $h$ above can be 
identified with a cocharacter of $H$ associated to  
$x$, where $x$ is viewed as a unipotent element of $H$.

Any two such cocharacters of $H$ are conjugate by elements 
of $\Cent_H(x)^0$.
\end{lem}
\begin{proof}
Recall J.C.~Jantzen's result \cite[Claim~5.12]{J} 
(see also \cite{FR} for a related study in
positive good characteristic): For any connected reductive 
subgroup $H_2$ of an arbitrary connected complex Lie group $H_1$, the 
cocharacters of $H_2$ associated to a unipotent element $x\in H_2$ are 
precisely the cocharacters of $H_1$ associated to $x$ which take values in $H_2$.  

Applying this with $H_1=H$ and $H_2=M^0$, we get that $h$ can be
identified with a cocharacter of $H$, and is associated to 
$x$ viewed as a unipotent element of $H$.

The last assertion follows from \cite[Lem.~5.3]{J}.
\end{proof}

>From now on we view the $h$ above as cocharacters of $H$ associated to $x$.
Any two $\gamma\colon\SL_2(\Cset)\to M^0\subset H$ as above are
conjugate by elements of $\Cent_H(x)$.

Suppose that $\gamma'\colon \SL_2 (\Cset)\to H$ is an optimal $\SL_2$-homomorphism 
for $x$ such that
\[
\gamma'(Y_\alpha)=\gamma(Y_\alpha)\quad\text{for }
\alpha\in\Cset^\times .
\]
Then $\gamma' = \gamma$, see \cite[Prop.~42]{McN}.

\smallskip

Choose a geometric Frobenius $\varpi_F$ and set $\Phi (\varpi_F) = t \in T$.
Define the Langlands parameter $\Phi$ as follows:
\begin{equation}\label{eqn:Phi}
\Phi \colon F^{\times} \times \SL_2 (\Cset) \to G, \qquad  
(u\varpi_F^n,Y) \mapsto c^\fs (u) \cdot t^n\cdot \gamma(Y) 
\end{equation}
for all  $u \in \fo_F^\times, \; n \in \Zset,\; Y \in \SL_2 (\Cset)$.

Note that the definition of $\Phi$ uses the appropriate data: 
the semisimple element $t \in T$, the map $c^\fs$, and the 
homomorphism $\gamma$ (which depends on the Springer parameter $x$).  

Since $x$ determines $\gamma$ up to $M^\circ$-conjugation, $c^\fs,x$ and $t$ 
determine $\Phi$ up to conjugation by their common centralizer in $G$. 
Notice also that one can recover $c^\fs, x$ and $t$ from $\Phi$ and that
\begin{equation}\label{eq:hPhi}
h (\alpha) = \Phi (1, Y_\alpha) . 
\end{equation}

\section{Varieties of Borel subgroups}
\label{sec:Borel}

We clarify some issues with different varieties of Borel subgroups and different
kinds of parameters arising from them. 
Let $G$ be a connected reductive complex group and let 
\[
\Phi \colon \mathbf W_F \times \SL_2 (\C) \to G 
\]
be as in \eqref{eqn:Phi}. We write 
\begin{align*}
& H = \Z_G (\Phi (\mathbf I_F)) = \Z_G (\im c^\fs), \\
& M = \Z_G (\Phi (\mathbf W_F)) = \Z_H (t). 
\end{align*}
Although both $H$ and $M$ are in general disconnected, $\Phi (\mathbf W_F)$ is always 
contained in $H^\circ$ because it lies in the maximal torus $T$ of $G$ and 
$H^\circ$. Hence $\Phi (\mathbf I_F) \subset \Z(H^\circ)$.

By construction $t$ commutes with $\Phi (\SL_2 (\C)) \subset M$. 
For any $q^{1/2} \in \C^\times$ the element 
\begin{equation}\label{eq:S.12}
t_q := t \Phi \big( Y_{q^{1/2}} \big)
\end{equation} 
satisfies the familiar relation $t_q x t_q^{-1} = x^q$. Indeed
\begin{equation}\label{eq:tqx}
\begin{split}
t_q x t_q^{-1} & = t \Phi (Y_{q^{1/2}}) \Phi \matje{1}{1}{0}{1} \Phi (Y_{q^{1/2}}^{-1}) t^{-1} \\
& = t \Phi \big( Y_{q^{1/2}} \matje{1}{1}{0}{1} Y_{q^{1/2}}^{-1} \big) t^{-1} \\
& = t \Phi \matje{1}{q}{0}{1} t^{-1} = x^q .
\end{split}
\end{equation}
Recall that $B_2$ denotes the upper triangular Borel subgroup of $\SL_2 (\C)$. In the 
flag variety of $M^\circ$ we have the subvarieties $\mathcal B^x_{M^\circ}$ and 
$\mathcal B^{\Phi (B_2)}_{M^\circ}$ of Borel subgroups containing $x$ and $\Phi (B_2)$, 
respectively. Similarly the flag variety of $H^\circ$ has subvarieties 
$\mathcal B^{t,x}_{H^\circ} ,\; \mathcal B^{t_q,x}_{H^\circ}$ and 
\[
\mathcal B^{t, \Phi (B_2)}_{H^\circ} = \mathcal B^{t_q, \Phi (B_2)}_{H^\circ} .
\]
Notice that $\Phi (\mathbf I_F)$ lies in every Borel subgroup of $H^\circ$, because it
is contained in $\Z(H^\circ)$. We abbreviate $\Z_H (\Phi) = 
\Z_H (\Phi (\mathbf W_F \times \SL_2 (\C)))$ and similarly for other groups.

\begin{prop}\label{prop:S.1}
\begin{enumerate}
\item The inclusion maps
\[
\begin{array}{ccccccc}
 & & \Z_{M^\circ}(\Phi) & \to & \Z_{M^\circ}(\Phi (B_2)) & 
\to & \Z_{M^\circ}(x) , \\
\Z_H (t_q,x) & \leftarrow & \Z_H (\Phi) & \to & \Z_H (t, \Phi (B_2)) & 
\to & \Z_H (t,x) , 
\end{array}
\]
are homotopy equivalences. In particular they induce isomorphisms between the 
respective component groups.
\item The inclusions $\mathcal B^{\Phi (B_2)}_{M^\circ} \to \mathcal B^x_{M^\circ}$ and 
$\mathcal B^{t_q,x}_{H^\circ} \leftarrow \mathcal B^{t,\Phi (B_2)}_{H^\circ} 
\to \mathcal B^{t,x}_{H^\circ}$ are homotopy equivalences.
\end{enumerate}
\end{prop}
\begin{proof} 
It suffices to consider the statements for $H$ and $t_q$, 
since the others can be proven in the same way.\\
(1) Our proof uses some elementary observations from \cite[\S 4.3]{R}. 
There is a Levi decomposition
\[
\Z_{H^\circ} (x) = \Z_{H^\circ} (\Phi (\SL_2 (\C))) U_x 
\]
with $\Z_{H^\circ} (\Phi (\SL_2 (\C))) = \Z_{H^\circ} (\Phi (B_2))$ reductive and $U_x$ unipotent. 
Since $t_q \in \Nor_{H^\circ} (\Phi (\SL_2 (\C)))$ and $\Z_H (x^q) = \Z_H (x)$, conjugation by $t_q$ 
preserves this decomposition. Therefore
\begin{equation}\label{eq:S.1}
\Z_{H^\circ} (t_q,x) = \Z_{H^\circ} (\Phi) \Z_{U_x}(t_q) = \Z_{H^\circ} (t_q, \Phi (B_2)) \Z_{U_x}(t_q). 
\end{equation}
We note that 
\[
\Z_{U_x}(t_q) \cap \Z_{H^\circ} (t_q, \Phi (B_2)) \subset U_x \cap \Z_{H^\circ} (\Phi (B_2)) = 1
\]
and that $\Z_{U_x}(t_q) \subset U_x$ is contractible, because it is a unipotent complex group. 
It follows that
\begin{equation}\label{eq:S.10}
\Z_{H^\circ} (\Phi) = \Z_{H^\circ} (t_q, \Phi (B_2)) \to \Z_{H^\circ} (t_q,x) 
\end{equation}
is a homotopy equivalence. If we want to replace $H^\circ$ by $H$, we find
\[
\Z_H (\Phi) / \Z_{H^\circ}(\Phi) = 
\{ h H^\circ \in \pi_0 (H) \mid h \Phi h^{-1} \in \mathrm{Ad}(H^\circ) \Phi \} ,
\]
and similarly with $(t_q, \Phi (B_2))$ or $(t_q,x)$ instead of $\Phi$. 

Let us have a closer look
at the $H^\circ$-conjugacy classes of these objects. Given any $\Phi$, we obviously know what 
$t_q$ and $x$ are. Conversely, suppose that $t_q$ and $x$ are given. We apply a refinement of
the Jacobson--Morozov theorem due to Kazhdan and Lusztig. According to \cite[\S 2.3]{KL} there
exist homomorphisms $\Phi : \mathbf W_F \times \SL_2 (\C) \to G$ as above, which return $t_q$ 
and $x$ in the prescribed way. Moreover all such homomorphisms are conjugate under 
$\Z_{H^\circ} (t_q,x)$, see \cite[\S 2.3.h]{KL} or Section 19. So from $(t_q,x)$ we can reconstruct 
the Ad$(H^\circ)$-orbit of $\Phi$, and this 
gives bijections between $H^\circ$-conjugacy classes of $\Phi ,\; (t_q, \Phi (B_2))$ and $(t_q,x)$. 
Since these bijections clearly are $\pi_0 (H)$-equivariant, we deduce 
\begin{equation}\label{eq:S.11}
\Z_H (\Phi) / \Z_{H^\circ}(\Phi) = \Z_H (t_q, \Phi (B_2)) / \Z_{H^\circ}(t_q, \Phi (B_2)) = 
\Z_H (t_q,x) / \Z_{H^\circ}(t_q,x) .
\end{equation}
Equations \eqref{eq:S.10} and \eqref{eq:S.11} imply that
\[
\Z_H (\Phi) = \Z_H (t_q, \Phi (B_2)) \to \Z_H (t_q,x) 
\]
is also a homotopy equivalence. \\
(2) By the aforementioned result \cite[\S 2.3.h]{KL} 
\begin{equation}\label{eq:S.2}
\Z_{H^\circ} (t_q,x) \cdot \mathcal B^{t_q, \Phi (B_2)}_{H^\circ} = \mathcal B^{t_q,x}_{H^\circ} .
\end{equation}
On the other hand, by \eqref{eq:S.1}
\begin{equation}\label{eq:S.3}
\Z_{H^\circ} (t_q,x) \cdot \mathcal B^{t_q, \Phi (B_2)}_{H^\circ} = 
\Z_{U_x}(t_q) \Z_H (t_q, \Phi (B_2)) \cdot \mathcal B^{t_q, \Phi (B_2)}_{H^\circ} = 
\Z_{U_x}(t_q) \cdot \mathcal B^{t_q, \Phi (B_2)}_{H^\circ} .
\end{equation}
For any $B \in \mathcal B^{t_q, \Phi (B_2)}_{H^\circ}$ and $u \in \Z_{U_x}(t_q)$ it is clear that
\[
u \cdot B \in \mathcal B^{t_q, \Phi (B_2)}_{H^\circ} \; \Longleftrightarrow \;
\Phi (B_2) \subset u B u^{-1} \; \Longleftrightarrow \;
u^{-1} \Phi (B_2) u \subset B .
\]
Furthermore, since $\Phi (B_2) \subset B$ is generated by $x$ and 
$\{ \Phi \matje{\alpha}{0}{0}{\alpha^{-1}} \mid \alpha \in \C^\times \}$, 
the right hand side is equivalent to 
\[
u^{-1} \Phi \matje{\alpha}{0}{0}{\alpha^{-1}} u \in B \quad \forall \alpha \in \C^\times .
\]
In Lie algebra terms this can be reformulated as 
\[
\mathrm{Ad}_{u^{-1}} (d \Phi \matje{\alpha}{0}{0}{-\alpha}) \in \mathrm{Lie} \, B
\quad \forall \alpha \in \C . 
\]
Because $u$ is unipotent, this happens if and only if
\[
\mathrm{Ad}_{u^\lambda} (d \Phi \matje{\alpha}{0}{0}{-\alpha}) \in \mathrm{Lie} \, B
\quad \forall \lambda,\alpha \in \C .  
\]
By the reverse chain of arguments the last statement is equivalent with
\[
u^\lambda \cdot B \in \mathcal B^{t_q, \Phi (B_2)}_{H^\circ} \quad \forall \lambda \in \C .
\]
Thus $\{ u \in \Z_{U_x}(t_q) \mid u \cdot B \in \mathcal B^{t_q, \Phi (B_2)}_{H^\circ} \}$ 
is contractible for all $B \in \mathcal B^{t_q, \Phi (B_2)}_{H^\circ}$, and we already knew 
that $\Z_{U_x}(t_q)$ is contractible. Together with \eqref{eq:S.2} and \eqref{eq:S.3} these 
imply that $\mathcal B^{t_q, \Phi (B_2)}_{H^\circ} \to \mathcal B^{t_q,x}_{H^\circ}$ 
is a homotopy equivalence.
\end{proof}

For the affine Springer correspondence we will need more precise information on the
relation between the varieties for $G$, for $H$ and for $M^\circ$.

\begin{prop}\label{prop:UP}
\begin{enumerate}
\item The variety $\mathcal B^{t,x}_{H^\circ}$ is isomorphic to $[\mathcal W^{H^\circ} : 
\mathcal W^{M^\circ}]$ copies of $\mathcal B^x_{M^\circ}$, and 
$\mathcal B^{t, \Phi (B_2)}_{H^\circ}$ is isomorphic to the same number of copies of 
$\mathcal B^{\Phi (B_2)}_{M^\circ}$.
\item The group $\Z_{H^\circ}(t,x) / \Z_{M^\circ}(x)$ permutes these two
sets of copies freely.
\item The variety $\mathcal B_G^{\Phi (\mathbf W_F \times B_2)}$ is isomorphic to
$[\mathcal W^G : \mathcal W^{H^\circ}]$ copies of $\mathcal B^{t, \Phi (B_2)}_{H^\circ}$. 
The group $\Z_G (\Phi) / \Z_{H^\circ}(\Phi)$ permutes these copies freely.
\end{enumerate}
\end{prop}
\begin{proof}
(1) Let $A$ be a subgroup of $T$ such that $M^\circ = \Z_{H^\circ} (A)^\circ$ and let 
$\mathcal B_{H^\circ}^A$ denote the variety of all Borel subgroups of $H^\circ$ 
which contain $A$. With an adaptation of \cite[p.471]{CG} we will prove that, for any 
$B \in \mathcal B_{H^\circ}^A ,\; B \cap M^0$ is a Borel subgroup of $M^0$. 

Since $B \cap M^\circ \subset B$ is solvable, it suffices to show that its Lie algebra
is a Borel subalgebra of Lie $M^\circ$. Write Lie $T = \mathfrak t$ and let 
\[
\text{Lie } H^\circ = \mathfrak n \oplus \mathfrak t \oplus \mathfrak n_- 
\]
be the triangular decomposition, where Lie $B = \mathfrak n \oplus \mathfrak t$.
Since $A \subset B$, it preserves this decomposition and
\begin{align*}
& \text{Lie } M^\circ = (\text{Lie } H )^A = 
\mathfrak n^A \oplus \mathfrak t \oplus \mathfrak n_-^A ,\\
& \text{Lie } B \cap M^\circ = \text{Lie } B^A = \mathfrak n^A \oplus \mathfrak t .
\end{align*}
The latter is indeed a Borel subalgebra of Lie $M^\circ$.
Thus there is a canonical map 
\begin{equation} \label{eqn:(7)} 
\mathcal B_{H^\circ}^A \to \Flag \, M^0, \quad B \mapsto B \cap M^0 .
\end{equation}
The group $M$ acts by conjugation on $\mathcal B_{H^\circ}^A$ and \eqref{eqn:(7)} 
clearly is $M$-equivariant. By \cite[p. 471]{CG} the $M^\circ$-orbits form a partition
\begin{equation}\label{eq:componentsBA}
\mathcal B_{H^\circ}^A = \mathcal B_1 \sqcup \mathcal B_2 \sqcup \cdots \sqcup \mathcal B_m .
\end{equation}
At the same time these orbits are the connected components of $\mathcal B_{H^\circ}^A$ 
and the irreducible components of the projective variety $\mathcal B_{H^\circ}^A$. 
The argument from \cite[p. 471]{CG} also shows that \eqref{eqn:(7)}, restricted to any one 
of these orbits, is a bijection from the $M^0$-orbit onto Flag $M^0$. 

The number of components $m$ can be determined as in the proof of
\cite[Corollary 3.12.a]{Ste1965}. The collection of Borel subgroups of $M^\circ$ that 
contain the maximal torus $T$ is in bijection with the Weyl group $\mathcal W^{M^\circ}$. 
Retracting via \eqref{eqn:(7)}, we find that every component $\mathcal B_i$ has precisely 
$|\mathcal W^{M^\circ}|$ elements that contain $T$. On the other hand, since $A \subset T ,\;
\mathcal B_{H^\circ}^A$ has $|\mathcal W^{H^\circ}|$ elements that contain $T$, so
\[
m = [\mathcal W^{H^\circ} : \mathcal W^{M^\circ}] .
\]
To obtain our desired isomorphisms of varieties, we let $A$ be the group generated by
$t$ and we restrict $\mathcal B_i \to \text{Flag } M^\circ$ to Borel subgroups
that contain $t,x$ (respectively $t,\Phi (B_2)$). \\ 
(2) By Proposition \ref{prop:S.1} 
\[
\Z_{H^\circ}(t,x) / \Z_{M^\circ}(x) \cong \Z_{H^\circ}(t, \Phi (B_2)) / \Z_{M^\circ}(\Phi (B_2)) .
\]
Since the former is a subgroup of $M / M^\circ$ and the copies under
consideration are in $M$-equivariant bijection with the components \eqref{eq:componentsBA},
it suffices to show that $M / M^\circ$ permutes these components freely.
Pick $B,B'$ in the same component $\mathcal B_i$ and assume that $B' = h B h^{-1}$ 
for some $h \in M$. Since $\mathcal B_i$ is $M^\circ$-equivariantly isomorphic 
to the flag variety of $M^\circ$ we can find $m \in M^\circ$ such that $B' = m^{-1} B m$.
Then $m h$ normalizes $B$, so $m h \in B$. As $B$ is connected, this implies
$m h \in M^\circ$ and $h \in M^\circ$. \\
(3) Apply the proofs of parts 1 and 2 with $A = \Phi (\mathbf I_F) ,\; G$ in the role of 
$H^\circ ,\; H^\circ$ in the role of $M^\circ$ and $t \Phi (B_2)$ in the role of $x$.
\end{proof}

\section{Comparison of different parameters}
\label{sec:comppar}

In the following sections we will make use of several different but related 
kinds of parameters. 
\vspace{2mm}

\noindent \textbf{Kazhdan--Lusztig--Reeder parameters}\\  
For a Langlands parameter as in \eqref{eqn:Phi}, the variety of Borel subgroups \\
$\mathcal B_G^{\Phi (\mathbf W_F \times B_2)}$ is nonempty, and the centralizer
$\Z_G (\Phi)$ of the image of $\Phi$ acts on it. Hence the group of components
$\pi_0 (\Z_G (\Phi))$ acts on the homology $H_* \big( \mathcal B_G^{\Phi (\mathbf W_F 
\times B_2)} ,\C \big)$. We call an irreducible representation $\rho$ of 
$\pi_0 (\Z_G (\Phi))$ geometric if it appears in $H_* \big( \mathcal B_G^{\Phi 
(\mathbf W_F \times B_2)} ,\C \big)$. We define a Kazhdan--Lusztig--Reeder parameter 
for $G$ to be a such pair $(\Phi,\rho)$. The group $G$ acts on these parameters by 
\begin{equation}\label{eq:defKLRparameter}
g \cdot (\Phi,\rho) = (g \Phi g^{-1}, \rho \circ \mathrm{Ad}_g^{-1}) 
\end{equation}
and we denote the corresponding equivalence class by $[\Phi,\rho ]_G$.
\vspace{2mm}

\noindent \textbf{Affine Springer parameters}\\ 
As before, suppose that $t \in G$ is semisimple and that $x \in \Z_G (t)$ is unipotent. 
Then $\Z_G (t,x)$ acts on $\mathcal B_G^{t,x}$ and $\pi_0 (\Z_G (t,x))$ acts on the
homology of this variety. In this setting we say that $\rho_1 \in \Irr \big( 
\pi_0 (\Z_G (t,x)) \big)$ is geometric if it appears in $H_{\top}( \mathcal B_G^{t,x} ,\C )$,
where $\top$ refers to highest degree in which the homology is nonzero, the real
dimension of $\mathcal B_G^{t,x}$.
We call such triples $(t,x,\rho_1)$ affine Springer parameters for $G$,
because they appear naturally in the representation theory of the affine Weyl group
associated to $G$. The group $G$ acts on such parameters by conjugation, and we 
denote the conjugacy classes by $[t,x,\rho_1]_G$.
\vspace{2mm}

\noindent \textbf{Kazhdan--Lusztig parameters}\\ 
Next we consider a unipotent element $x \in G$ and a semisimple element $t_q \in G$ 
such that $t_q x t_q^{-1} = x^q$. As above, $\Z_G (t_q,x)$ acts on the variety
$\mathcal B_G^{t_q,x}$ and we call $\rho_q \in \Irr \big( \pi_0 (\Z_G (t_q,x)) \big)$
geometric if it appears in $H_* \big( \mathcal B_G^{t_q,x} ,\C \big)$. Triples 
$(t_q,x,\rho_q)$ of this kind are known as Kazhdan--Lusztig parameters for $G$.
Again they are endowed with an obvious $G$-action and we denote the equivalence classes
by $[t_q,x,\rho_q]_G$.
\vspace{2mm}

In \cite{KL,R} there are some indications that these three kinds of parameters are
essentially equivalent. Proposition \ref{prop:S.1} allows us to make this precise
in the necessary generality.

\begin{lem}\label{lem:compareParameters}
Let $\fs$ be a Bernstein component in the principal series, associate 
$c^\fs \colon \fo_F^\times \to T$ 
to it as in Lemma \ref{lem:cBernstein} and write $H = \Z_G (c^\fs(\fo_F^\times))$.
There are natural bijections between $H^\circ$-equivalence classes of:
\begin{itemize}
\item Kazhdan--Lusztig--Reeder parameters for $G$ with $\Phi \big|_{\fo_F^\times} = c^\fs$ and\\ 
$\Phi (\varpi_F) \in H^\circ$;
\item affine Springer parameters for $H^\circ$;\
\item Kazhdan--Lusztig parameters for $H^\circ$.
\end{itemize}
\end{lem}
\begin{proof}
Since $\SL_2 (\C)$ is connected and commutes with $\fo_F^\times$, its image under $\Phi$ must
be contained in the connected component of $H$. Therefore KLR-parameters with these 
properties are in canonical bijection with KLR-parameters for $H^\circ$ and it suffices to 
consider the case $H^\circ = G$.

As in \eqref{eqn:Phi} and \eqref{eq:S.12}, any KLR-parameter gives rise to the 
ingredients $t,x$ and $t_q$ for the other two kinds of parameters. As we discussed
after \eqref{eqn:Phi}, the pair $(t,x)$ is enough to recover the conjugacy class
of $\Phi$. A refined version of the Jacobson--Morozov theorem says that the same
goes for the pair $(t_q,x)$, see \cite[\S 2.3]{KL} or \cite[Section 4.2]{R}.

To complete $\Phi, (t,x)$ or $(t_q,x)$ to a parameter of the appropriate kind,
we must add an irreducible representation $\rho ,\rho_1$ or $\rho_q$.
For the affine Springer parameters it does not matter whether we 
consider the total homology or only the homology in top degree. Indeed, it follows 
from Propositions \ref{prop:S.1} and \ref{prop:UP} and \cite[bottom of page~296 and 
Remark 6.5]{Shoji} that any irreducible representation $\rho_1$ which appears in 
$H_* \big( \mathcal B_G^{t,x} ,\C \big)$, already appears 
in the top homology of this variety.

This and Proposition \ref{prop:S.1} show that there is a natural correspondence 
between the possible ingredients $\rho,\rho_1$ and $\rho_q$. 
\end{proof}

\section{The affine Springer correspondence}
\label{sec:affSpringer}

An interesting instance of Proposition \ref{prop:S.2} arises when $M$ is the 
centralizer of a semisimple element $t$ in a connected reductive 
complex group $G$. As before we assume that $t$ lies in a maximal torus $T$ of 
$G$ and we write $\mathcal W^G = W(G,T)$. By Lemma \ref{lem:centrals}
\begin{equation}\label{eq:S.8}
\mathcal W^M := \Nor_M (T) / \Z_M (T) \cong W \rtimes \pi_0 (M) 
\end{equation}
is the stabilizer of $t$ in $\mathcal W^G$, so the role of $\Gamma$ is played by the 
component group $\pi_0 (M)$. In contrast to the setup in Proposition \ref{prop:S.2},
it is possible that some elements of $\pi_0 (M) \setminus \{1\}$ fix $W$ pointwise.
This poses no problems however, as such elements never act trivially on $T$. 
For later use we record the following consequence of \eqref{eq:S.9}:
\begin{equation}\label{eq:23.7}
\pi_0 (M)_{\tau (x,\rho)} \cong \big( \Z_M (x) / \Z_{M^\circ}(x) \big)_\rho .
\end{equation}
Recall from Section \ref{sec:extquot2} that 
\begin{align*}
& \widetilde{T}_2: = \{(t,\sigma) \,:\, t \in T, \sigma \in \Irr(\cW_t^G)\}, \\
&(T\q \cW^G)_2: = \widetilde{T}_2/\cW^G.
\end{align*}
We note that the rational characters of the complex torus $T$ span the regular 
functions on the complex variety $T$:
\[
\mathcal{O}(T) = \C [X^*(T)].
\]
>From \eqref{EXT1} and \eqref{EXT2}, we infer the following rough form of the extended 
Springer correspondence for the affine Weyl group $X^*(T) \rtimes \cW^G$. 

\begin{thm} 
There are canonical bijections
\[
(T\q \cW^G )_2 \simeq \Irr \, (X^*(T) \rtimes \cW^G ) \simeq 
\{(t,\tau(x, \varrho) \rtimes \psi) \}/\cW^G
\]
with $t \in T, \tau(x, \varrho) \in
\Irr \,\cW^{M^0}, \psi \in \Irr(\pi_0(M)_{\tau (x,\varrho)})$.
\end{thm}

Now we recall the geometric realization of irreducible representations of 
$X^* (T) \rtimes \mathcal W^G$ by Kato \cite{Kat}. 
For a unipotent element $x \in M^\circ$ let $\mathcal B^{t,x}_G$ be the variety of Borel 
subgroups of $G$ containing $t$ and $x$. Fix a Borel subgroup $B$ of $G$ containing $T$ 
and let $\theta_{G,B} : \mathcal B^{t,x}_G \to T$ be the morphism defined by
\[
\theta_{G,B}(B') = g^{-1} t g \text{ if } B' = g B g^{-1} \text{ and } t \in g T g^{-1}.
\]
The image of $\theta_{G,B}$ is $\mathcal W^G t$, the map is constant on the irreducible
components of $\mathcal B^{t,x}_G$ and it gives rise to an action of $X^* (T)$ on the 
homology of $\mathcal B^{t,x}_G$. Furthermore $\Q [ \mathcal W^G ] \cong H (\mathcal Z_G)$ 
acts on $H_{d(x)}(\mathcal B^{t,x}_G,\C)$ via the convolution product in Borel--Moore 
homology, as described in \eqref{eq:S.5}. Both actions commute with the action of 
$\Z_G (t,x)$ induced by conjugation of Borel subgroups. 

Let $\rho_1 \in$ Irr$(\Z_G (t,x))$. By \cite[Theorem 4.1]{Kat} the 
$X^* (T) \rtimes \mathcal W^G$-module 
\begin{equation}\label{eq:KatoMod}
\tau (t,x,\rho_1) := \mathrm{Hom}_{\Z_G (t,x)} \big( \rho_1, H_{d(x)}(\mathcal B^{t,x}_G,\C) \big)  
\end{equation}
is either irreducible or zero. Moreover every irreducible representation of 
$X^* (T) \rtimes \mathcal W^G$ is obtained is in this way, and the data $(t,x,\rho_1)$
are unique up to $G$-conjugacy. This generalizes the Springer correspondence for
finite Weyl groups, which can be recovered by considering the representations on 
which $X^* (T)$ acts trivially.

Propositions \ref{prop:S.2} and \ref{prop:UP} shine some new light on this:

\begin{thm}\label{thm:S.3}
\begin{enumerate}
\item There are natural bijections between the following sets:
\begin{itemize} 
\item $\mathrm{Irr}(X^* (T) \rtimes \mathcal W^G) = 
\mathrm{Irr}(\mathcal O (T) \rtimes \mathcal W^G)$;
\item $(T // \mathcal W^G )_2  = \big\{ (t, \tilde \tau ) \mid t \in T , 
\tilde \tau \in \mathrm{Irr}(\mathcal W^M) \big\} / \mathcal W^G$;
\item $\big\{ (t, \tau, \sigma) \mid t \in T , \tau \in \mathrm{Irr}(\mathcal W^{M^\circ}), 
\sigma \in \mathrm{Irr}(\pi_0 (M)_\tau) \big\} / \mathcal W^G$;
\item $\big\{ (t,x,\rho,\sigma) \mid t \in T , x \in M^\circ \, \mathrm{unipotent} , 
\rho \in \mathrm{Irr} \big( \pi_0 (\Z_{M^\circ}(x)) \big) \\ \mathrm{geometric}, 
\sigma \in \mathrm{Irr} \big( \pi_0 (M)_{\tau (x,\rho)} \big) \big\} / G$;
\item $\big\{ (t,x,\rho_1) \mid t \in T, x \in M^\circ \, \mathrm{unipotent}, 
\rho_1 \in \mathrm{Irr} \big( \pi_0 (\Z_G (t,x)) \big) \\ \mathrm{geometric} \big\} / G$.
\end{itemize}
Here a representation of $\pi_0 (\Z_{M^\circ}(x))$ (or $\pi_0 (\Z_G (t,x))$) is 
called geometric if it appears in $H_{d(x)}(\mathcal B^x_{M^\circ} ,\C)$ (respectively 
$H_{d(x)}(\mathcal B^{t,x}_G,\C)$). 
\item The $X^*(T) \rtimes \mathcal W^G$-representation corresponding to $(t,x,\rho_1)$ 
via these bijections is Kato's module \eqref{eq:KatoMod}.
\end{enumerate}
\end{thm}

We remark that in the fourth and fifth sets it would be more natural to allow $t$ to be any
semisimple element of $G$. In fact that would give the affine Springer parameters from
Lemma \ref{lem:compareParameters}.
Clearly $G$ acts on the set of such more general parameters
$(t,x,\rho,\sigma)$ or $(t,x,\rho_1)$, which gives equivalence relations $/G$. The two above 
$/G$ refer to the restrictions of these equivalence relations to parameters with $t \in T$.

\begin{proof} 
(1) Recall that the isotropy group of $t$ in $\mathcal W^G$ is 
\[
\mathcal W^G_t = \mathcal W^M = \mathcal W^{M^\circ} \rtimes \pi_0 (M) .
\]
Hence the bijection between the first two sets is an instance of Clifford theory, see 
Section 5. The second and third sets are in natural bijection by Proposition \ref{prop:S.2}.
The Springer correspondence for $\mathcal W^{M^\circ}$ provides the bijection with the
fourth collection. To establish a natural bijection with the fifth collection, we first 
recall from \eqref{eq:S.9} that $\pi_0 (M)_{\tau (x,\rho)} = 
\pi_0 (M)_{[x,\rho]_{M^\circ}}$. By that and Proposition \ref{prop:S.2} every 
irreducible representation of
\[
\pi_0 (\Z_G (t,x)) \cong \pi_0 (\Z_{M^\circ}(x)) \rtimes \pi_0 (M)_{[x]_{M^\circ}}
\]
is of the form $\rho \rtimes \sigma$ for $\rho$ and $\sigma$ as in the fourth set.
By Proposition \ref{prop:UP}
\begin{equation}\label{eq:S.26}
H_* (\mathcal B^{t,x}_G, \C) \cong H_* (\mathcal B^x_{M^\circ} ,\C) \otimes 
\C [\Z_G (t,x) / \Z_{M^\circ}(x)] \otimes \C^{ [\mathcal W^G : \mathcal W^G_t]}
\end{equation}
as $\Z_G (t,x)$-representations. By \cite[\S 3.1]{R} 
\[
\Z_G (t,x) / \Z_{M^\circ}(x) \cong \pi_0 (M)_{[x]_{M^\circ}}
\] 
is abelian. Hence 
$\mathrm{Ind}_{\pi_0 (M)_{[x,\rho]_{M^\circ}}}^{\pi_0 (M)_{[x]_{M^\circ}}} (\sigma)$ 
appears exactly once in the regular representation of this group and
\begin{multline}\label{eq:S.34}
\mathrm{Hom}_{\pi_0 (\Z_G(t,x))} \big( \rho \rtimes \sigma ,H_{d(x)} 
(\mathcal B^{t,x}_G, \C) \big) \cong \\
\mathrm{Hom}_{\pi_0 (\Z_{M^\circ}(x))} \big( \rho ,H_{d(x)} (\mathcal B^x_{M^\circ}, 
\C) \big) \rtimes \sigma \otimes \C^{[\mathcal W^G : \mathcal W^G_t]} .
\end{multline}
In particular we see that $\rho$ is geometric if and only if $\rho \rtimes \sigma$ is 
geometric, which establishes the final bijection.\\
(2) The $X^* (T) \rtimes \mathcal W^G$-representation constructed from 
$(t,x,\rho \rtimes \sigma)$ by means of our bijections is
\begin{equation}\label{eq:S.28}
\mathrm{Ind}_{X^*(T) \rtimes \mathcal W^G_t}^{X^*(T) \rtimes \mathcal W^G} 
\big( \mathrm{Hom}_{\pi_0 (\Z_{M^\circ}(x))} 
\big( \rho ,H_{d(x)} (\mathcal B^x_{M^\circ}, \C) \big) \rtimes \sigma \big) .
\end{equation}
On the other hand, by \cite[Proposition 6.2]{Kat}
\begin{multline} \label{eq:S.35}
H_* (\mathcal B^{t,x}_G,\C) \cong \mathrm{Ind}_{X^*(T) \rtimes 
\mathcal W^{M^\circ}}^{X^*(T) \rtimes \mathcal W^G} (H_* (\mathcal B^x_{M^\circ} ,\C)) \\
\cong \mathrm{Ind}_{X^*(T) \rtimes \mathcal W^G_t}^{X^*(T) \rtimes \mathcal W^G} 
\big( H_* (\mathcal B^x_{M^\circ} ,\C) \otimes \C [\Z_G (t,x) / \Z_{M^\circ}(x)] \big)
\end{multline}
as $\Z_G (t,x) \times X^* (T) \rtimes \mathcal W^G$-representations. Together with
the proof of part 1 this shows that $\tau (t,x,\rho \rtimes \sigma)$
is isomorphic to \eqref{eq:S.28}. 
\end{proof}

We can extract a little more from the above proof. Recall that $\cO_x$ 
denotes the conjugacy class of $x$ in $M$. Let us agree that the 
affine Springer parameters with a fixed $t \in T$ are partially ordered by
\[
(t,x,\rho_1) < (t,x',\rho'_1) \quad \text{when} \quad
\overline{\cO_x} \subsetneq \overline{\cO_{x'}} .
\]
\begin{lem}\label{lem:S.7}
There exist multiplicities $m_{t,x,\rho_1,x',\rho'_1} \in {\mathbb Z}_{\geq 0}$ such that
\begin{multline*}
\mathrm{Hom}_{\pi_0 (\Z_G(t,x))} \big( \rho_1 ,
H_* (\mathcal B^{t,x}_G, \C) \big) \cong \\
\tau (t,x,\rho_1) \oplus \bigoplus_{(t,x',\rho'_1) > (t,x,\rho_1)}
m_{t,x,\rho_1,x',\rho'_1} \,\tau (t,x',\rho'_1) .
\end{multline*}
\end{lem}
\begin{proof}
It follows from \eqref{eq:S.35}, \eqref{eq:S.26} and \eqref{eq:S.34} that
\begin{multline}\label{eq:S.36}
\mathrm{Hom}_{\pi_0 (\Z_G(t,x))} \big( \rho \rtimes \sigma ,
H_* (\mathcal B^{t,x}_G, \C) \big) \cong \\
\mathrm{Ind}_{X^*(T) \rtimes \mathcal W^G_t}^{X^*(T) \rtimes \mathcal W^G} 
\mathrm{Ind}_{\cW^{M^\circ} \rtimes \pi_0 (M)_{[x,\rho]_{M^\circ}}}^{\cW^G_t}
\big( \mathrm{Hom}_{\pi_0 (\Z_{M^\circ}(x))} 
\big( \rho ,H_{d(x)} (\mathcal B^x_{M^\circ}, \C) \big) \otimes \sigma \big) .
\end{multline}
The functor $\mathrm{Ind}_{X^*(T) \rtimes \mathcal W^G_t}^{X^*(T) \rtimes \mathcal W^G}$
provides an equivalence between the categories 
\begin{itemize}
\item $X^*(T) \rtimes \mathcal W^G_t$-representations with $\cO (T)^{\cW^G_t}$-character $t$;
\item $X^*(T) \rtimes \mathcal W^G$-representations with $\cO (T)^{\cW^G}$-character $\cW^G t$.
\end{itemize}
Therefore we may apply Lemma \ref{lem:S.6} to the right hand side of \eqref{eq:S.36},
which produces the required formula.
\end{proof}

Let us have a look at the representations with an affine Springer parameter of
the form $(t,x=1,\rho_1 = \mathrm{triv})$. Equivalently, the fourth parameter in 
Theorem \ref{thm:S.3} is $(t,x=1,\rho = \mathrm{triv},\sigma = \mathrm{triv})$.
The $\mathcal W^{M^\circ}$-representation with Springer parameter $(x=1,\rho = 
\mathrm{triv})$ is the trivial representation, so $(x=1,\rho = \mathrm{triv},
\sigma = \mathrm{triv})$ corresponds to the trivial representation of 
$\mathcal W^G_t$. With \eqref{eq:S.28} we conclude that the $X^* (T) \rtimes 
\mathcal W^G$-representation with affine Springer parameter $(t,1,\mathrm{triv})$ is
\begin{equation} \label{eq:trivWaff}
\tau (t,1,\triv) = \mathrm{Ind}_{X^*(T) \rtimes \mathcal W^G_t}^{X^*(T) 
\rtimes \mathcal W^G} \big( \mathrm{triv}_{\mathcal W^G_t} \big) .
\end{equation}
Notice that this is the only irreducible $X^* (T) \rtimes \mathcal W^G$-representation 
with an $X^*(T)$-weight $t$ and nonzero $\mathcal W^G$-fixed vectors.

\section{Geometric representations of affine Hecke algebras}
\label{sec:repAHA}

Let $G$ be a connected reductive complex group, $B$ a Borel subgroup and $T$ a maximal
torus of $G$ contained in $B$. Let $\mathcal H (G)$ be the affine Hecke 
algebra with 
the same based root datum as $(G,B,T)$ and with a parameter 
$q \in \C^\times$ which is not a root of unity.

Since later on we will have to deal with disconnected reductive groups, we include
some additional automorphisms in the picture. In every root subgroup $U_\alpha$ with 
$\alpha \in \Delta (B,T)$ we pick a nontrivial element $u_\alpha$. Let $\Gamma$ be a 
finite group of automorphisms of $(G,T,(u_\alpha)_{\alpha \in \Delta (B,T)})$. Since $G$ need 
not be semisimple, it is possible that some elements of $\Gamma$ fix the entire root system 
of $(G,T)$. Notice that $\Gamma$ acts on the Weyl group $\mathcal W^G = W(G,T)$ because it
stabilizes $T$. 

We form the crossed product $\mathcal H (G) \rtimes \Gamma$ with respect to the 
canonical $\Gamma$-action on $\mathcal H (G)$. We define a Kazhdan--Lusztig parameter 
for $\mathcal H (G) \rtimes \Gamma$ to be a triple $(t_q,x,\rho)$ such that
\begin{itemize}
\item $t_q \in G$ is semisimple, $x \in G$ is unipotent and $t_q x t_q^{-1} = x^q$;
\item $\rho$ is an irreducible representation of the component group \\
$\pi_0 (\Z_{G \rtimes \Gamma} (t_q,x))$, such that every irreducible subrepresentation
of the restriction of $\rho$ to $\pi_0 (\Z_G (t_q,x))$ appears in $H_* (\mathcal B^{t_q,x},\C)$.
\end{itemize}
The group $G \rtimes \Gamma$ acts on such parameters by conjugation, and we denote the
conjugacy class of a parameter by $[t_q,x,\rho]_{G \rtimes \Gamma}$. 
Now we generalize \cite[Theorem 7.12]{KL} and \cite[Theorem 3.5.4]{R}:

\begin{thm}\label{thm:S.5}
There exists a natural bijection between $\mathrm{Irr}(\mathcal H (G) \rtimes \Gamma)$ and 
$G \rtimes \Gamma$-conjugacy classes of Kazhdan--Lusztig parameters. The module corresponding
to $(t_q,x,\rho)$ is the unique irreducible quotient of the 
$\mathcal H (G)\rtimes \Gamma$-module 
\[
\mathrm{Hom}_{\pi_0 (\Z_G (t_q,x))} 
\big( \rho_q, H_* (\mathcal B^{t_q,x},\C) \big) \rtimes \rho_2 ,
\]
where $\rho_q \in \mathrm{Irr} \big( \pi_0 (\Z_G (t_q,x)) \big)$ and $\rho_2 \in 
\mathrm{Irr}^{\natural(t_q,x,\rho_q)}(\Gamma_{[t_q,x,\rho_q]_G})$ are such that 
$\rho \cong \rho_q \rtimes \rho_2$.
\end{thm}
\begin{proof}
First we recall the geometric constructions of $\mathcal H (G)$-modules by Kazhdan, Lusztig
and Reeder, taking advantage of Lemma \ref{lem:S.3} to simplify the presentation somewhat. 
As in \cite[\S 1.5]{R}, let
\[
1 \to C \to \tilde G \to G \to 1
\]
be a finite central extension such that $\tilde G$ is a connected reductive group with
simply connected derived group. The kernel $C$ acts naturally on $\mathcal H (\tilde G)$ and 
\begin{equation}\label{eq:S.17}
\mathcal H (\tilde G )^C \cong \mathcal H (G) .
\end{equation}
The action of $\Gamma$ on the based root datum of $(G,B,T)$ lifts uniquely to an action on
the corresponding based root datum for $\tilde G$, so the $\Gamma$-actions on $G$ and on
$\mathcal H (G)$ lift naturally to actions on $\tilde G$ and $\mathcal H (\tilde G)$.
Let $\mathcal H_{\mathbf q} (\tilde G)$ be the variation on 
$\mathcal H (\tilde G)$ with scalars $\mathbb Z [\mathbf q ,\mathbf q^{-1}]$ instead of 
$\C$ and $q \in \C^\times$. In \cite[Theorem 3.5]{KL} an isomorphism
\begin{equation}\label{eq:S.18}
\mathcal H_{\mathbf q} (\tilde G) \cong K^{\tilde G \times \C^\times}(\mathcal Z_{\tilde G}) 
\end{equation}
is constructed, where the right hand side denotes the $\tilde G \times \C^\times$-equivariant
K-theory of the Steinberg variety $\mathcal Z_{\tilde G}$ of $\tilde G$. Since 
$G \rtimes \Gamma$ acts via conjugation on $\tilde G$ and on $\mathcal Z_{\tilde G}$, 
it also acts on $K^{\tilde G \times \C^\times}(\mathcal Z_{\tilde G})$. However, the 
connected group $G$ acts trivially, so the action factors via $\Gamma$. Now the definition 
of the generators in \cite[Theorem 3.5]{KL} shows that \eqref{eq:S.18} is 
$\Gamma$-equivariant. In particular it specializes to $\Gamma$-equivariant isomorphisms
\begin{equation}\label{eq:S.19}
\mathcal H (\tilde G) \cong \mathcal H_{\mathbf q} (\tilde G) 
\otimes_{\mathbb Z [\mathbf q,\mathbf q^{-1}]} \C_q \cong K^{\tilde G \times \C^\times}
(\mathcal Z_{\tilde G}) \otimes_{\mathbb Z [\mathbf q,\mathbf q^{-1}]} \C_q .
\end{equation}
Let $(\tilde t_q,\tilde x) \in (\tilde G)^2$ be a lift of $(t_q,x) \in G^2$ with
$\tilde x$ unipotent. The $\tilde G$-conjugacy class of $\tilde t_q$ defines a central 
character of $\mathcal H (\tilde G)$ and according to \cite[Proposition 8.1.5]{CG} 
the associated localization is
\[
\mathcal H (\tilde G) \otimes_{\Z (\mathcal H (\tilde G))} \C_{\tilde t_q} \cong
K^{\tilde G \times \C^\times}(\mathcal Z_{\tilde G}) \otimes_{R (\tilde G \times \C^\times)} 
\C_{\tilde t_q,q} \cong H_* (\mathcal Z^{\tilde t_q,q} ,\C) . 
\]
Any Borel subgroup of $\tilde G$ contains $C$, so $\mathcal B^{\tilde t_q,\tilde x} = 
\mathcal B^{\tilde t_q,\tilde x}_{\tilde G}$ and $\mathcal B^{t_q,x} = \mathcal B^{t_q,x}_G$ 
are isomorphic algebraic varieties. From \cite[p. 414]{CG} we see
that the convolution product in Borel--Moore homology leads to an action of 
$H_* (\mathcal Z^{\tilde t_q,q}_{\tilde G} ,\C)$ on $H_* ( \mathcal B^{\tilde t_q,\tilde x}, 
\C)$. Notice that for $\tilde h \in H_* (\mathcal Z^{\tilde t_q,q}_{\tilde G} ,\C)$ and 
$g \in G \rtimes \Gamma$ we have
\[
g \cdot \tilde h \in H_* (\mathcal Z^{g \tilde t_q g^{-1},q}_{\tilde G} ,\C) \cong 
\mathcal H (\tilde G) \otimes_{\Z (\mathcal H (\tilde G))} \C_{g \tilde t_q g^{-1}} .
\]
An obvious generalization of \cite[Lemma 8.1.8]{CG} says that all these 
constructions are compatible with the above actions of $G \rtimes \Gamma$,
in the sense that the following diagram commutes:
\begin{equation}\label{eq:S.20}
\begin{array}{ccc}
H_* ( \mathcal B^{\tilde t_q,\tilde x}, \C)
& \xrightarrow{\; \tilde h \;} & H_* ( \mathcal B^{\tilde t_q,\tilde x}, \C) \\
\downarrow \scriptstyle{H_* (\mathrm{Ad}_g)} & & 
\downarrow \scriptstyle{H_* (\mathrm{Ad}_g)} \\
H_* ( \mathcal B^{g \tilde t_q g^{-1},g \tilde x g^{-1}}, \C) & 
\xrightarrow{\;g \cdot \tilde h\;} & 
H_* ( \mathcal B^{g \tilde t_q g^{-1},g \tilde x g^{-1}}, \C) .
\end{array}
\end{equation}
In particular the component group $\pi_0 (\Z_{\tilde G}(\tilde t_q,\tilde x))$ acts on
$H_* ( \mathcal B^{\tilde t_q,\tilde x}, \C)$ by $\mathcal H (\tilde G)$-intertwiners.
Let $\tilde \rho$ be an irreducible representation of this component group, appearing
in $H_* ( \mathcal B^{\tilde t_q,\tilde x}, \C)$. In other words, $(\tilde t_q,
\tilde x,\tilde \rho)$ is a Kazhdan--Lusztig parameter for $\mathcal H (\tilde G)$. 
According to \cite[Theorem 7.12]{KL}
\begin{equation}\label{eq:S.16}
\mathrm{Hom}_{\pi_0 (\Z_{\tilde G}(\tilde t_q,\tilde x))} 
(\tilde \rho , H_* (\mathcal B^{\tilde t_q,\tilde x},\C)) 
\end{equation}
is a $\mathcal H (\tilde G)$-module with a unique irreducible quotient, say 
$V_{\tilde t_q,\tilde x,\tilde \rho}$. 

Following \cite[\S 3.3]{R} we define a group $R_{\tilde t_q,\tilde x}$ by
\begin{equation}\label{eq:S.21}
1 \to \pi_0 (\Z_{\tilde G}(\tilde t_q,\tilde x)) \to \pi_0 (\Z_{\tilde G} (t_q,x)) \to 
R_{\tilde t_q,\tilde x} \to 1 .
\end{equation}
Since the derived group of $\tilde G$ is simply connected, $\Z_{\tilde G} (t_q)^\circ = 
\Z_{\tilde G} (\tilde t_q)$. Furthermore $\Z (\tilde G)$ acts trivially on 
$H_* ( \mathcal B^{\tilde t_q,\tilde x}, \C)$, so we may just as well replace 
\eqref{eq:S.21} by the short exact sequence
\begin{equation*}
1 \to \pi_0 (\Z_{\Z_G (t_q)^\circ}(x) / \Z(G)) \to 
\pi_0 (\Z_G (t_q,x) / \Z(G)) \to R_{\tilde t_q,\tilde x} \to 1 .
\end{equation*}
>From Lemma \ref{lem:S.3} (with the trivial representation of $\pi_0 (\Z_G (t_q)^\circ (x))$ 
in the role of $\rho$) 
we know that the latter exact sequence splits. Hence all the 2-cocycles of subgroups of 
$R_{\tilde t_q,\tilde x}$ appearing in \cite[Section 3]{R} are trivial.

Let $\tilde \sigma$ be any irreducible representation of $R_{\tilde t_q,\tilde x,\tilde \rho}$,
the stabilizer of the isomorphism class of $\tilde \rho$ in $R_{\tilde t_q,\tilde x}$.
Clifford theory for \eqref{eq:S.21} produces $\tilde \rho \rtimes \tilde \sigma \in$ 
Irr$(\pi_0 (\Z_{\tilde G} (t_q,x)))$, a representation which factors through
$\pi_0 (\Z_G (t_q,x))$ because $C$ acts trivially. Moreover by \cite[Lemma 3.5.1]{R}
it appears in $H_* (\mathcal B^{t_q,x},\C)$, and conversely every irreducible representation
with the latter property is of the form $\tilde \rho \rtimes \tilde \sigma$.

With the above in mind, \cite[Lemma 3.5.2]{R} says that the $\mathcal H (G)$-module
\begin{equation}\label{eq:S.27}
\begin{split}
M(t_q,x,\tilde \rho \rtimes \tilde \sigma) := &\; \mathrm{Hom}_{\pi_0 (\Z_G (t_q,x))} 
\big( \tilde \rho \rtimes \tilde \sigma, H_* (\mathcal B^{t_q,x},\C) \big) \\
= &\; \mathrm{Hom}_{R_{\tilde t_q,\tilde x,\tilde \rho}} \big( \tilde \sigma, 
\mathrm{Hom}_{\pi_0 (\Z_{\tilde G}(\tilde t_q,\tilde x))} (\tilde \rho , 
H_* (\mathcal B^{\tilde t_q,\tilde x},\C))  \big)
\end{split} 
\end{equation}
has a unique irreducible quotient
\begin{equation}\label{eq:S.29}
\pi (t_q,x,\tilde \rho \rtimes \tilde \sigma) = 
\mathrm{Hom}_{R_{\tilde t_q,\tilde x,\tilde \rho}} 
(\tilde \sigma , V_{\tilde t_q,\tilde x,\tilde \rho}) .
\end{equation}
According  \cite[Lemma 3.5.3]{R} this sets up a bijection between 
Irr$(\mathcal H (G))$ and $G$-conjugacy classes of Kazhdan--Lusztig parameters for $G$.

\begin{rem}
The module \eqref{eq:S.27} is well-defined for any $q \in \C^\times$, although for roots of
unity it may have more than one irreducible quotient. For $q=1$ the algebra $\mathcal H (G)$
reduces to $\C [X^* (T) \rtimes \mathcal W^G]$ and \cite[Section 8.2]{CG} shows that Kato's
module \eqref{eq:KatoMod} is a direct summand of $M (t_1,x,\rho_1)$. 
\end{rem}

Next we study what $\Gamma$ does to all these objects. There is natural action of $\Gamma$
on Kazhdan--Lusztig parameters for $G$, namely 
\[
\gamma \cdot (t_q,x, \rho_q) =  \big( \gamma t_q \gamma^{-1}, 
\gamma x \gamma^{-1}, \rho_q\circ \mathrm{Ad}_\gamma^{-1} \big) .
\]
>From \eqref{eq:S.20} and \eqref{eq:S.27} we deduce that the diagram
\begin{equation}\label{eq:S.22}
\begin{array}{ccc}
\pi (t_q,x,\rho_q ) & \xrightarrow{\; h \;} & \pi (t_q,x,\rho_q) \\
\downarrow \scriptstyle{H_* (\mathrm{Ad}_g)} & & 
\downarrow \scriptstyle{H_* (\mathrm{Ad}_g)} \\
\!\! \pi \big( g t_q g^{-1},g x g^{-1}, \rho_q \circ \mathrm{Ad}_g^{-1} \big) & 
\xrightarrow{\; \gamma (h) \;} & 
\pi \big( g t_q g^{-1},g x g^{-1},\rho_q \circ \mathrm{Ad}_g^{-1} \big)
\end{array}
\end{equation}
commutes for all $g \in G \gamma$ and $h \in \mathcal H (G)$. Hence 
\begin{equation}\label{eq:S.23}
\text{Reeder's parametrization of Irr}(\mathcal H (G)) \text{ is }\Gamma\text{-equivariant.} 
\end{equation}
Let $\pi \in$ Irr$(\mathcal H (G))$ and choose a Kazhdan--Lusztig parameter such that
$\pi$ is equivalent with $\pi (t_q,x,\rho_q)$. Composition with $\gamma^{-1}$ on $\pi$ 
gives rise to a 2-cocycle $\natural (\pi)$ of $\Gamma_\pi$. Clifford theory tells us that 
every irreducible representation of $\mathcal H (G) \rtimes \Gamma$ is of the form 
$\pi \rtimes \rho_2$ for some $\pi \in$ Irr$(\mathcal H (G))$, unique up to 
$\Gamma$-equivalence, and a unique $\rho_2 \in \mathrm{Irr}^{\natural (\pi)}(\Gamma_\pi)$. 
By the above the stabilizer of $\pi$ in $\Gamma$ equals the stabilizer of the $G$-conjugacy 
class $[t_q,x,\rho_q]_G$. Thus we have parametrized Irr$(\mathcal H (G) \rtimes \Gamma)$ 
in a natural way with $G \rtimes \Gamma$-conjugacy classes of quadruples 
$(t_q,x,\rho_q,\rho_2)$, where $(t_q,x,\rho_q)$ is a Kazhdan--Lusztig parameter for $G$ and 
$\rho_2 \in \mathrm{Irr}^{\natural(\pi(t_q,x,\rho_q))} \Gamma_{[t_q,x,\rho_q]_G}$.

The short exact sequence
\begin{equation}\label{eq:S.24}
1 \to \pi_0 (\Z_G (t_q,x)) \to \pi_0 (\Z_{G \rtimes \Gamma} (t_q,x)) \to
\Gamma_{[t_q,x]_G} \to 1
\end{equation}
yields an action of $\Gamma_{[t_q,x]_G}$ on Irr$\big( \pi_0 (\Z_G (t_q,x)) \big)$.
Restricting this to the stabilizer of $\rho_q$, we obtain another 2-cocycle 
$\natural(t_q,x,\rho_q)$ of $\Gamma_{[t_q,x,\rho_q]_G}$, which we want to compare to 
$\natural(\pi(t_q,x,\rho_q))$. Let us decompose
\[
H_* (\mathcal B^{t_q,x},\C) \cong \bigoplus\nolimits_{\rho_q} \rho_q \otimes 
M(t_q,x,\rho_q)
\]
as $\pi_0 ((\Z_G (t_q,x)) \times \mathcal H(G)$-modules. We sum over all $\rho_q \in$
Irr$\big( \pi_0 (\Z_G (t_q,x)) \big)$ for which the contribution is nonzero, and we know
that for such $\rho_q$ the $\mathcal H (G)$-module $M(t_q,x,\rho_q)$ has a unique irreducible 
quotient $\pi (t_q,x,\rho_q)$. Since $\pi_0 (\Z_{G \rtimes \Gamma} (t_q,x))$ acts 
(via conjugation of Borel subgroups) on $H_* (\mathcal B^{t_q,x},\C)$, any splitting of 
\eqref{eq:S.24} as sets provides a 2-cocycle $\natural$ for the action of 
$\Gamma_{[t_q,x,\rho_q]_G}$ on $\rho_q \otimes M(t_q,x,\rho_q)$.
Unfortunately we cannot apply Lemma \ref{lem:S.3} to find a splitting of \eqref{eq:S.24}
as groups, because $\Z_G (t_q)$ need not be connected. Nevertheless $\natural$ 
can be used to describe the actions of $\Gamma_{[t_q,x,\rho_q]_G}$ on both $\rho_q$ and 
$\pi (t_q,x,\rho_q)$, so as cohomology classes
\begin{equation}\label{eq:S.25}
\natural(t_q,x,\rho_q) = \natural = \natural(\pi(t_q,x,\rho_q)) 
\in H^2 (\Gamma_{[t_q,x,\rho_q]_G},\C^\times) .
\end{equation}
It follows that every irreducible representation $\rho$ of $\pi_0 (\Z_{G \rtimes \Gamma} 
(t_q,x))$ is of the form $\rho_q \rtimes \rho_2$ for $\rho_q$ and $\rho_2$ as above.
Moreover $\rho$ determines $\rho_q$ up to $\Gamma_{[t_q,x]_G}$-equivalence and $\rho_2$
is unique if $\rho_q$ has been chosen. Finally, if $\rho_q$ appears in 
$H_{\top}(\mathcal B^{t_q,x},\C)$ then every irreducible $\pi_0 (\Z_G (t_q,x))
$-subrepresentation of $\rho$ does, because $\pi_0 (\Z_{G \rtimes \Gamma} (t_q,x))$ acts 
naturally on $H_* (\mathcal B^{t_q,x},\C)$. Therefore we may replace the above 
quadruples $(t_q,x,\rho_q,\rho_2)$ by Kazhdan--Lusztig parameters $(t_q,x,\rho)$. 
\end{proof}

\section{Spherical representations}
\label{sec:spherical}

Let $G,B,T$ and $\Gamma$ be as in the previous section.
Let $\mathcal H (\cW^G )$ be the Iwahori--Hecke algebra of the Weyl group 
$\cW^G$, with a parameter $q \in \C^\times$ which is not a root of unity. This
is a deformation of the group algebra $\C [\cW^G]$ and a subalgebra of the affine Hecke 
algebra $\cH (G)$. The multiplication is defined in terms of the basis 
$\{ T_w \mid w \in \cW^G\}$, as in \eqref{eq:defHA}.

Recall that $\cH (G)$ also has a commutative subalgebra $\cO (T)$, 
such that the multiplication maps
\begin{equation}\label{eq:multmaps}
\cO (T) \otimes \cH (\cW^G) \longrightarrow \cH (G) \longleftarrow 
\cH (\cW^G) \otimes \cO (T)
\end{equation}
are bijective.

The trivial representation of $\cH (\cW^G) \rtimes \Gamma$ is defined as
\begin{equation}
\triv (T_w \gamma) = q^{\ell (w)} \quad w \in \cW^G, \gamma \in \Gamma .
\end{equation}
It is associated to the idempotent
\[
p_\triv \: p_\Gamma := \sum_{w \in \cW^G} T_w P_{\cW^G}(q)^{-1} \, 
\sum_{\gamma \in \Gamma} \gamma |\Gamma |^{-1} \quad \in \cH (\cW^G) \rtimes \Gamma ,
\]
where $P_{\cW^G}$ is the Poincar\'e polynomial
\[
P_{\cW^G}(q) = \sum\nolimits_{w \in \cW^G} q^{\ell (w)} .
\]
Notice that $P_{\cW^G}(q) \neq 0$ because $q$ is not a root of unity.
The trivial representation appears precisely once in the regular representation of
$\cH (\cW^G) \rtimes \Gamma$, just like for finite groups.

An $\cH (G) \rtimes \Gamma$-module $V$ is called spherical if it is generated by
the subspace $p_\triv p_\Gamma V$ \cite[(2.5)]{HeOp}. This admits a nice
interpretation for the unramified principal series representations. Recall that
$\cH (G) \cong \cH (\cG,\mathcal I)$ for an Iwahori subgroup $\mathcal I \subset \cG$. 
Let $\mathcal K \subset \cG$ be a good maximal compact subgroup containing $\mathcal I$. 
Then $p_\triv$ corresponds to averaging over $\mathcal K$ and $p_\triv \cH (\cG,\mathcal I) 
p_\triv \cong \cH (\cG,\mathcal K)$, see \cite[Section 1]{HeOp}. Hence spherical 
$\cH (\cG,\mathcal I)$-modules correspond to smooth
$\cG$-representations that are generated by their $\mathcal K$-fixed vectors, also known as
$\mathcal K$-spherical $\cG$-representations. By the Satake transform
\begin{equation} \label{eq:Satake}
p_\triv \cH (\cG,\mathcal I) p_\triv \cong \cH (\cG,\mathcal K) \cong \cO (T / \cW^G) ,
\end{equation}
so the irreducible spherical modules of $\cH (G) \cong \cH (\cG,\mathcal I)$ are parametrized
by $T / \cW^G$ via their central characters. We want to determine the Kazhdan--Lusztig 
parameters (as in Theorem \ref{thm:S.5}) of these representations.

\begin{prop}\label{prop:spherical}
For every central character $(\cW^G \rtimes \Gamma) t \in T / (\cW^G \rtimes \Gamma)$ 
there is a unique irreducible spherical $\cH (G) \rtimes \Gamma$-module,
and it has Kazhdan--Lusztig parameter $(t,x=1,\rho = \triv)$.
\end{prop}
\begin{proof}
We will first prove the proposition for $\cH (G)$, and only then consider $\Gamma$.

By the Satake isomorphism \eqref{eq:Satake} there is a unique irreducible spherical
$\cH (G)$-module for every central character $\cW^G t \in T / \cW^G$. The equivalence
classes of Kazhdan--Lusztig parameters of the form $(t,x=1,\rho = \triv)$ are also in
canonical bijection with $T / \cW^G$. Therefore it suffices to show that $\pi (t,1,\triv)$
is spherical for all $t \in T$.

The principal series of $\cH (G)$ consists of the modules 
$\mathrm{Ind}_{\cO (T)}^{\cH (G)} \C_t$ for $t \in T$. This module admits a central
character, namely $\cW^G t$. By \eqref{eq:multmaps} every such module is isomorphic to 
$\cH (\cW^G)$ as a $\cH (\cW^G)$-module. In particular it contains the trivial 
$\cH (\cW^G)$-representation once and has a unique irreducible spherical subquotient. 

As in Section \ref{sec:repAHA}, let $\tilde G$ be a finite central extension of $G$ with
simply connected derived group. Let $\tilde T ,\tilde B$ be the corresponding extensions 
of $T,B$. We identify the roots and the Weyl groups of $\tilde G$ and $G$. 
Let $\tilde t \in \tilde T$ be a lift of $t \in T$.
>From the general theory of Weyl groups it is known that there is a unique
$t^+ \in \cW^G \tilde T$ such that $|\alpha (t^+)| \geq 1$ for all $\alpha \in 
R (\tilde B, \tilde T) = R(B,T)$.
By \eqref{eq:S.20} 
\[
H_* \big( \mathcal B_{\tilde G}^{\tilde t},\C \big) \cong H_*
\big( \mathcal B_{\tilde G}^{t^+} ,\C \big)
\]
as $\cH (\tilde G)$-modules. These $t^+,\tilde B$ fulfill \cite[Lemma 2.8.1]{R}, so by
\cite[Proposition 2.8.2]{R}
\begin{equation}
M_{\tilde t,\tilde x = 1,\tilde \rho = \triv} = H_* (\mathcal B_{\tilde G}^{t^+},\C)
\cong \mathrm{Ind}_{\cO (\tilde T)}^{\cH (\tilde G)} \C_{t^+} .
\end{equation}
According to \cite[(1.5)]{ReeWhittaker}, which applies to $t^+$, the spherical vector $p_\triv$
generates $M_{\tilde t,1,\triv}$. Therefore it cannot lie in any proper 
$\cH (\tilde G)$-submodule of $M_{\tilde t,1,\triv}$ and represents a nonzero element
of $\pi (\tilde t,1,\triv)$. We also note that the central character of $\pi (\tilde t,1,
\triv)$ is that of $M_{\tilde t,1,\triv} ,\; \cW^G \tilde t = \cW^G t^+$. 

Now we analyse this is an $\cH (G)$-module. The group $R_{\tilde t,1} = 
R_{\tilde t,\tilde x = 1, \tilde \rho = \triv}$ from \eqref{eq:S.21} is just the component 
group $\pi_0 (\Z_G (t))$, so by \eqref{eq:S.29}
\[
\pi (\tilde t,1,\triv) \cong \bigoplus\nolimits_\rho \mathrm{Hom}_{\pi_0 (\Z_G (t))}
(\rho, \pi (\tilde t,1,\triv)) = \bigoplus\nolimits_\rho \pi (t,1,\triv) .
\]
The sum runs over $\Irr \big( \pi_0 (\Z_G (t)) \big)$, all these representations $\rho$
contribute nontrivially by \cite[Lemma 3.5.1]{R}. Recall from Lemma \ref{lem:centrals}
that $\pi_0 (\Z_G (t))$ can be realized as a subgroup of $\cW^G$ and from \eqref{eq:Satake}
that $p_\triv \in \pi (\tilde t,1,\triv)$ can be regarded as a function on $\tilde \cG$
which is bi-invariant under a good maximal compact subgroup $\tilde{\mathcal K}$. This brings
us in the setting of \cite[Proposition 4.1]{Cas}, which says that $\pi_0 (\Z_G (t))$ fixes 
$p_\triv \in \pi (\tilde t,1,\triv)$. Hence $\pi (t,1,\triv)$ contains $p_\triv$ and is a
spherical $\cH (G)$-module. Its central character is the restriction of the central character
of $\pi (\tilde t,1,\triv)$, that is, $\cW^G t \in T / \cW^G$.

Now we include $\Gamma$. Suppose that $V$ is a irreducible spherical
$\cH (G) \rtimes \Gamma$-module. By Clifford theory its restriction to $\cH (G)$ is a
direct sum of irreducible $\cH (G)$-modules, each of which contains $p_\triv$. Hence $V$
is built from irreducible spherical $\cH (G)$-modules. By \eqref{eq:S.23}
\[
\gamma \cdot \pi (t,1,\triv) = \pi (\gamma t,1,\triv) ,
\]
so the stabilizer of $\pi (t,1,\triv) \in \Irr (\cH (G))$ in $\Gamma$ equals the stabilizer
of $\cW^G t \in T / \cW^G$ in $\Gamma$. Any isomorphism of $\cH (G)$-modules
\[
\psi_\gamma : \pi (t,1,\triv) \to \pi (\gamma t,1,\triv) 
\]
must restrict to a bijection between the onedimensional subspaces of spherical vectors
in both modules. We normalize $\psi_\gamma$ by $\psi_\gamma (p_\triv) = p_\triv$. Then
$\gamma \mapsto \psi_\gamma$ is multiplicative, so the 2-cocycle of $\Gamma_{\cW^G t}$
is trivial. With Theorem \ref{thm:S.5} this means that the irreducible $\cH (G) \rtimes 
\Gamma$-modules whose restriction to $\cH (G)$ is spherical are parametrized by 
equivalence classes of triples $(t,1,\triv \rtimes \sigma)$ with $\sigma \in
\Irr (\Gamma_{\cW^G t})$. The corresponding module is
\begin{equation*}
\pi (t,1,\triv \rtimes \sigma) = \pi (t,1,\triv) \rtimes \sigma =
\mathrm{Ind}_{\cH (G) \rtimes \Gamma_{\cW^G t}}^{\cH (G) \rtimes \Gamma}
\big( \pi (t,1,\triv) \otimes \sigma \big) .
\end{equation*}
Clearly $\pi (t,1,\triv \rtimes \sigma)$ contains the spherical vector
$p_\triv p_\Gamma$ if and only if $\sigma$ is the trivial representation. It follows
that the irreducible spherical $\cH (G) \rtimes \Gamma$-modules are parametrized by
equivalence classes of triples $\big( t,1,\triv_{\pi_0 (\Z_{G \rtimes \Gamma}(t))} \big)$,
that is, by $T / (\cW^G \rtimes \Gamma)$.
\end{proof}

\section{Main result (the case of a connected endoscopic group)} 
\label{sec:main} 

Let $\chi$ be a smooth character of the maximal torus $\cT \subset \cG$. We recall that
\begin{align*}
\fs & = [\cT,\chi]_{\cG},\\
c^\fs & = \hat \chi \big|_{\fo_F^\times},\\
H & = \Cent_G(\im c^\fs) , \\
W^\fs &  = \Z_{\cW^G} (\im c^\fs) .
\end{align*}
Let \{KLR parameters$\}^\fs$ be the collection of Kazhdan--Lusztig--Reeder parameters 
for $G$ such that $\Phi \big|_{\fo_F^\times} = c^\fs$. Notice that the condition forces
$\Phi (\mathbf W_F \times \SL_2 (\C)) \subset H$. This collection is not closed under 
conjugation by elements of $G$, only $H = \Z_G (\im c^\fs)$ acts naturally on it.

The Bernstein centre associated to $\fs$ is $T^\fs / W^\fs$. Since $T = T^\fs$ is
a maximal torus in $H$, we can identify $T^\fs / W^\fs$ with the space $c(H)_{\ss}$ 
of semisimple conjugacy classes in $H$.

Roche \cite{Roc} proved that $\Irr (\cG)^{\fs}$ is naturally in bijection with 
$\Irr (\mathcal H (H))$, under some restrictions on the residual characteristic $p$. 
Although Roche works with a local non-archimedean field $F$ of characteristic 0, it follows
from \cite{AdRo} that his arguments apply just as well over local fields of positive 
characteristic. We note that for unramified characters $\chi$ this result is already
classical, proven (without any restrictions on $p$) by Borel \cite{Bor}.

In the current section we will prove the most important part of our conjecture in the 
case that $H$ is connected. This happens for most $\fs$, a sufficient condition is:

\begin{lem} \label{lem:Roche}   
Suppose that $G$ has simply connected derived group and that the residual 
characteristic $p$ satisfies the hypothesis in \cite[p.~379]{Roc}. 
Then $H$ is connected.
\end{lem} 
\begin{proof}   
We consider first the case where $\fs=[\cT,1]_\cG$. 
Then we have  $c^\fs = 1, H=G$ and $W^\fs=\cW$.

We assume now that $c^\fs \ne 1$. Then $\im c^\fs$ is a finite abelian subgroup of 
$T$ which has the following structure: the direct product of a finite abelian $p$-group 
$A_p$ with a cyclic group $B_{q-1}$ whose order divides $q - 1$.  This follows from the  
well-known structure theorem for the group $\fo_F^\times$, see \cite[\S 2.2]{I}:
\[ 
\im c^\fs = A_p \cdot B_{q-1}.
\] 
We have
\[
H = \Cent_{H_A}(B_{q-1}) \quad\text{where}\quad H_A := \Cent_G(A_p).
\]
Since $G$ has simply connected derived group, it follows from Steinberg's
connectedness theorem \cite{Steinberg} that the group $H_A$
is connected.
Since $A_p$ is a $p$-group, and $p$ is not a torsion prime for the root system $R(G,T)$,
the derived group of $H_A$ is simply connected (see \cite{Steinbergtorsion}). 

Now  $B_{q-1}$ is cyclic. Applying Steinberg's connectedness theorem to the group $H_A$, 
we get that $H$ itself is connected. 
\end{proof}

\begin{rem}
Notice that $H$ does not necessarily have simply connected derived group in setting
of Lemma \ref{lem:Roche}. For instance, if $G$ is the exceptional group of type $\rG_2$
and $\chi$ is the tensor square of a ramified quadratic character of $F^\times$, 
then $H=\SO(4,\Cset)$.
\end{rem}

In the remainder of this section we will assume that $H$ is connected, Then Lemma
\ref{lem:centrals} shows that $W^\fs$ is the Weyl group of $H$. 

\begin{thm} \label{thm:ps}  
Let $\cG$ be a  split reductive $p$-adic group and let ${\fs}$ 
be a point in the Bernstein spectrum of the principal series of $\cG$. In the
case when $H\ne G$, assume that $H$
is connected and that the residual characteristic $p$ satisfies the conditions of 
\cite[Remark 4.13]{Roc}. Then there is a commutative triangle of natural bijections 
\[ 
\xymatrix{   & (T/\!/W^\fs)_2 \ar[dr]\ar[dl] & \\  
\Irr(\mathcal{G})^\fs    \ar[rr] & & \{\KLR\:\:\mathrm{parameters}\}^\fs / H} 
\] 
In this triangle, the right slanted map stems from the affine Springer correspondence, 
the bottom horizontal map is the bijection established by Reeder \cite{R}, and the 
left slanted map can be constructed via the asymptotic algebra of Lusztig.
\end{thm}

\begin{proof}
The right slanted map is the composition of Theorem \ref{thm:S.3}.1 (applied
to $H$) and Lemma 
\ref{lem:compareParameters} (with the condition $\Phi(\varpi_F)=t$). Since $\Irr(\mathcal{G})^\fs \cong \Irr (\mathcal H (H))$,
we can take as the horizontal map the parametrization of irreducible $\mathcal H (H)$-modules
by Kazhdan, Lusztig and Reeder as described in Section \ref{sec:repAHA}. These are
both natural bijections, so there is a unique left slanted map which makes the 
diagram commute, and it is also natural. We want to identify it in terms of Hecke algebras.

Fix a KLR-parameter $(\Phi,\rho)$ and recall from Theorem \ref{thm:S.3}.2 that the 
corresponding $X^* (T) \rtimes \mathcal W^H$-representation is  
\begin{equation}
\mathrm{Hom}_{\pi_0 (\Z_H (t,x))} \big( \rho, H_{d(x)}(\mathcal B^{t,x}_H,\C) \big) .
\end{equation}
Similarly, by Theorem \ref{thm:S.5} the corresponding $\mathcal H (H)$-module 
is the unique irreducible quotient of the $\mathcal H (H)$-module
\begin{equation}\label{eq:Mtqx}
\mathrm{Hom}_{\pi_0 (\Z_H (t_q,x))} \big( \rho_q, H_*(\mathcal B^{t_q,x}_H,\C) \big) .
\end{equation}
In view of Proposition \ref{prop:S.1} both spaces are unchanged if we replace $t$ by $t_q$ 
and $\rho$ by $\rho_q$, and the vector space \eqref{eq:Mtqx} is also naturally 
isomorphic to
\begin{equation}
\mathrm{Hom}_{\pi_0 (\Z_H (\Phi))} \big( \rho, H_* (\mathcal B_H^{t,\Phi (B_2)} ,\C) \big) .
\end{equation}
Recall Lusztig's asymptotic Hecke algebra $\mathcal J (H)$ from \cite{LuCellsIII}. 
As discussed in Corollary~\ref{cor:J} in Example~3 of the Appendix, 
there are canonical bijections
\begin{equation}\label{eq:bijectionsIrr}
\Irr (\mathcal H (H)) \longleftrightarrow \Irr (\mathcal J (H)) \longleftrightarrow
\Irr (X^* (T) \rtimes \cW^H ) .
\end{equation}
According to \cite[Theorem 4.2]{LuCellsIV} 
$\Irr (\mathcal J (H))$ is naturally parametrized by the set of $H$-conjugacy classes
of Kazhdan--Lusztig parameters for $H$. Lusztig describes the $\mathcal J (H)$-module 
associated to $(t_q,x,\rho_q)$ in terms of equivariant $K$-theory, and with 
\cite[Section 6.2]{CG} we see that its retraction to $\mathcal H (H)$ via 
\[
\cH(H)\overset{\phi_q}\longrightarrow  \cJ(H) 
\overset{\phi_1}\longleftarrow X^* (T) \rtimes \cW^H
\] 
is none other than \eqref{eq:Mtqx}. In \cite[Corollary 3.6]{LuCellsIII} the $a$-function 
is used to single out a particular irreducible quotient $\mathcal H (H)$-module of 
\eqref{eq:Mtqx}. But we saw in \eqref{eq:S.27} that there is only one such quotient,
which by definition is $\pi (t_q,x,\rho_q)$.

Let $\mathcal H_q (H)$ be the affine Hecke algebra with the same based root datum as
$H$ and with parameter $q \in \C^\times$. Thus 
\[
\mathcal H_q (H) = \mathcal H (H) \quad  \text{and} \quad
\mathcal H_1 (H) = \C [X^* (T) \rtimes \mathcal W^H] .
\]
The above describes the retraction of an irreducible $\mathcal J (H)$-module 
corresponding to $(\Phi,\rho)$ to $\mathcal H_q (H)$ for any $q \in \C^\times$ which 
is not a root of unity. But everything depends algebraically on $q$, so the description 
is valid for all $q \in \C^\times$, in particular for $q=1$. Then \cite[Section 8.2]{CG} 
implies that we obtain the $\mathcal H_1 (H)$-module
\begin{equation}\label{eq:S.38}
\mathrm{Hom}_{\pi_0 (\Z_H (t,x))} \big( \rho, H_* (\mathcal B^{t,x}_H,\C) \big) 
\end{equation}
with the action as in \eqref{eq:KatoMod}. The right bijection in \eqref{eq:bijectionsIrr}
sends \eqref{eq:Mtqx} to a certain irreducible quotient of \eqref{eq:S.38} (namely the
unique one with minimal $a$-weight).

For the opposite direction, consider an irreducible $\mathcal H_1 (H)$-module $M$ with 
$a$-weight $a_M$. According to \cite[Corollary 3.6]{LuCellsIII} the $\mathcal J (H)$-module 
\[
\tilde M := \mathcal H_1 (H)^{a_M} \otimes_{\mathcal H_1 (H)} M ,
\]
is irreducible and has $a$-weight $a_M$. See \cite[Lemma 1.9]{LuCellsIII} for the precise
definition of $\tilde M$. 

Now we fix $t \in T$ and we will prove with induction to $\dim \cO_x$ that 
$\widetilde{\tau (t,x,\rho)}$ is none other than \eqref{eq:S.38}. Our main tool is Lemma 
\ref{lem:S.7}, which says that the constituents of \eqref{eq:S.38} are $\tau (t,x,\rho)$
and irreducible representations corresponding to larger affine Springer parameters
(with respect to the partial order defined via the unipotent classes $\cO_x \subset M$).
For $\dim \cO_{x_0} = 0$ we see immediately that only the $\mathcal J^\fs$-module
\[
\mathrm{Hom}_{\pi_0 (\Z_H (t,x_0))} \big( \rho_0, H_* (\mathcal B^{t,x}_H,\C) \big) 
\]
can contain $\tau (t,x_0,\rho_0)$, so that must be $\widetilde{\tau (t,x_0,\rho_0)}$.
For $\dim \cO_{x_n} = n$ Lemma \ref{lem:S.7} says that \eqref{eq:S.38} can only contain 
$\tau (t,x_n,\rho_n)$ if $x \in \overline{\cO_{x_n}}$. But when $\dim \cO_x < n$ 
\[
\widetilde{\tau (t,x_n,\rho_n)} \not\cong 
\mathrm{Hom}_{\pi_0 (\Z_H (t,x))} \big( \rho, H_* (\mathcal B^{t,x}_H,\C) \big) ,
\]
because the right hand side already is $\widetilde{\tau (t,x,\rho)}$, by the induction
hypothesis and the bijectivity of $M \mapsto \tilde M$. So the parameter of 
$\widetilde{\tau (t,x_n,\rho_n)}$ involves an $x$ with $\dim \cO_x = n$.
Then another look at Lemma \ref{lem:S.7} shows that moreover $(x,\rho)$ must be 
$M$-conjugate to $(x_n,\rho_n)$. Hence $\widetilde{\tau (t,x,\rho)}$ is indeed
\eqref{eq:S.38}.

We showed that the bijections \eqref{eq:bijectionsIrr} work out as
\[
\begin{array}{ccccc}
\Irr (\mathcal H (H)) & \leftrightarrow & \Irr (\mathcal J (H)) & \leftrightarrow
& \Irr (X^* (T) \rtimes \cW^H ) \\
\pi (t_q,x,\rho_q) & \leftrightarrow & 
\mathrm{Hom}_{\pi_0 (\Z_H (\Phi))} \big( \rho, H_* (\mathcal B_H^{t,\Phi (B_2)} ,\C) \big) &
\leftrightarrow & \tau (t,x,\rho) ,
\end{array}
\]
where all the objects in the bottom line are determined by the KLR parameter $(\Phi,\rho)$.
\end{proof}

\section{Main result (Hecke algebra version)}
\label{sec:mainH}

In this section $q \in \C^\times$ is allowed to be any element which is not a root of unity. 
We study how the conjecture can be extended to the algebras and modules from Section 
\ref{sec:repAHA}. So let $\Gamma$ be a group of automorphisms of $G$ that
preserves a chosen pinning, which involves $T$ as maximal torus. 
With the disconnected group $G \rtimes \Gamma$ we associate three kinds of parameters:
\begin{itemize}
\item The extended quotient of the second kind $(T /\!/ \mathcal W^G \rtimes \Gamma )_2$.
\item The space $\Irr (\mathcal H (G) \rtimes \Gamma)$ of equivalence classes of irreducible 
representations of the algebra $\mathcal H (G) \rtimes \Gamma$ (with parameter $q$).
\item Equivalence classes of unramified Kazhdan--Lusztig--Reeder parameters. 
Let $\Phi : \mathbf W_F \times \SL_2 (\C) \to G$ be a group homomorphism with 
$\Phi (\mathbf I_F) = 1$ and $\Phi (\mathbf W_F) \subset T$. As in Section \ref{sec:Borel}, 
the component group 
\[
\pi_0 (\Z_{G \rtimes \Gamma}(\Phi)) = \pi_0 (\Z_{G \rtimes \Gamma}(\Phi (\mathbf W_F \times B_2)))
\]
acts on $H_* (\mathcal B^{\Phi (\mathbf W_F \times B_2)}_G ,\C)$. We take $\rho \in$
$\Irr \big( \pi_0 (\Z_{G \rtimes \Gamma}(\Phi)) \big)$ such that every irreducible
$\pi_0 (\Z_G (\Phi))$-subrepresentation of $\rho$ appears in $H_* (\mathcal B^{\Phi 
(\mathbf W_F \times B_2)}_G ,\C)$. The set \{KLR parameters for $G \rtimes \Gamma \}^{\unr}$ 
of pairs $(\Phi,\rho)$ carries an action of $G \rtimes \Gamma$ by conjugation. We consider 
the collection \{KLR parameters for $G \rtimes \Gamma \}^{\unr} / G \rtimes \Gamma$ 
of conjugacy classes $[\Phi,\rho ]_{G \rtimes \Gamma}$.
\end{itemize}

\begin{thm}\label{thm:S.4}
There exists a commutative diagram of natural bijections
\[
\xymatrix{ 
& \hspace{-7mm} (T /\!/ \mathcal W^G \rtimes \Gamma )_2 \hspace{-7mm} \ar[dr]\ar[dl] & \\  
\Irr (\mathcal H (G) \rtimes \Gamma)    \ar[rr] & & 
\{ \textup{KLR parameters for } G \rtimes \Gamma \}^{\unr} / G \rtimes \Gamma } 
\]
It restricts to bijections between the following subsets:
\begin{itemize}
\item the ordinary quotient $T / (\cW^G \rtimes \Gamma) \subset
(T /\!/ \mathcal W^G \rtimes \Gamma )_2$,
\item the collection of spherical representations in $\Irr (\mathcal H (G) \rtimes \Gamma)$,
\item equivalence classes of KLR parameters $(\Phi,\rho)$ for $G \rtimes \Gamma$ with\\
$\Phi (\mathbf I_F \times \SL_2 (\C)) = 1$ and 
$\rho = \triv_{\pi_0 (\Z_{G \rtimes \Gamma}(\Phi))}$.
\end{itemize}
\end{thm}
\begin{proof}
The corresponding statement for $G$, proven in Theorem \ref{thm:ps}, 
is the existence of natural bijections
\begin{equation}\label{eq:S.13}
\xymatrix{ & (T /\!/ \mathcal W^G )_2 \ar[dr]\ar[dl] & \\  
\Irr (\mathcal H (G) )    \ar[rr] & & 
\{ \text{KLR parameters for } G \}^{\unr} / G} 
\end{equation}
Although in Section \ref{sec:main} $q$ was a prime power, we notice that the upper and right 
objects in \eqref{eq:S.13} do not depend on $q$. The algebra $\mathcal H (G)$ does, but 
the bottom and left slanted maps in \eqref{eq:S.13} are defined equally well for our more 
general $q \in \C^\times$, as can be seen from the proofs of Theorems \ref{thm:S.5} and 
\ref{thm:ps}. Thus we may use \eqref{eq:S.13} as our starting point.

\emph{Step 1. The bijections in \eqref{eq:S.13} are $\Gamma$-equivariant.}\\
The action of $\Gamma$ on $(T // \mathcal W^G )_2$ can be written as
\begin{equation}\label{eq:S.14}
\gamma \cdot [t,\tilde \tau]_{\mathcal W^G} = 
[\gamma (t),\tilde \tau \circ \mathrm{Ad}_\gamma^{-1}]_{\mathcal W^G} .
\end{equation}
In terms of the multiplication in $G \rtimes \Gamma$, the action on KLR parameters is
\begin{equation}\label{eq:S.15}
\gamma \cdot [\Phi ,\rho_1]_G = 
[\gamma \Phi \gamma^{-1}, \rho_1 \circ \mathrm{Ad}_\gamma^{-1} ]_G
\end{equation}
We recall the right slanted map in \eqref{eq:S.13} from Theorem \ref{thm:S.3}. 
Write $M = \Z_G (t)$ and $\mathcal W^G_t = W (M^\circ,T) \rtimes \pi_0 (M)$. Then the
$\mathcal W^G_t$-representation $\tilde \tau$ can be written as $\tau (x,\rho_3) \rtimes 
\sigma$ for a unipotent element $x \in M^\circ$, a geometric $\rho_3 \in$ Irr$(\Z_{M^\circ}(x))$ 
and a $\sigma \in$ Irr$(\pi_0 (M)_{\tau (x,\rho_3)})$. The associated KLR parameter is 
$[\Phi,\rho_3 \rtimes \sigma]_G$, where $\Phi \matje{1}{1}{0}{1} = x$ and $\Phi$ maps a 
Frobenius element of $\mathbf W_F$ to $t$.

>From \eqref{eq:S.7} we see that $\tau (x,\rho_3) \circ \mathrm{Ad}_\gamma^{-1}$ is equivalent
with $\tau (\gamma x \gamma^{-1},\rho_3 \circ \mathrm{Ad}_\gamma^{-1})$, so 
\[
\tilde \tau \circ \mathrm{Ad}_\gamma^{-1} \text{ is equivalent with } \tau (\gamma x \gamma^{-1},
\rho_3 \circ \mathrm{Ad}_\gamma^{-1}) \rtimes (\sigma \circ \mathrm{Ad}_\gamma^{-1}) .
\]
Hence \eqref{eq:S.14} is sent to the KLR parameter \eqref{eq:S.15}, which means that the right
slanted map in \eqref{eq:S.13} is indeed $\Gamma$-equivariant.

In view of Proposition \ref{prop:S.1} and \eqref{eq:S.15}, we already showed in \eqref{eq:S.23} 
that the horizontal map in \eqref{eq:S.13} is $\Gamma$-equivariant. By the commutativity of
the triangle, so is the left slanted map.

\emph{Step 2. Suppose that $\pi, [t,\tilde \tau]_{\mathcal W^G}$ and $[\Phi,\rho_1]_G$ are 
three corresponding objects in \eqref{eq:S.13}. Then their stabilizers in $\Gamma$ coincide:}
\[
\Gamma_\pi = \Gamma_{[t,\tilde \tau]_{\mathcal W^G}} = \Gamma_{[\Phi,\rho_1]_G} .
\]
This follows immediately from step 1.

\emph{Step 3. Clifford theory produces 2-cocycles $\natural (\pi) ,\; \natural \big( 
[t,\tilde \tau]_{\mathcal W^G} \big)$ and $\natural \big([\Phi,\rho_1]_G \big)$. 
These are in the same class in $H^2 (\Gamma_\pi ,\C^\times)$.}\\
For $\natural (\pi)$ and $\natural \big([\Phi,\rho_1]_G \big)$ this was already checked in
\eqref{eq:S.25}, where we use Proposition \ref{prop:S.1} to translate between $\Phi$ and
$(t_q,x)$. Comparing \eqref{eq:S.27} and Theorem \ref{thm:S.3}, we see that $\natural (\pi)$
and $\natural \big( [t,\tilde \tau]_{\mathcal W^G} \big)$ come from two very similar 
representations: the difference is that $M(t,x,\rho_1)$ is built the entire homology of a 
variety, while the corresponding $X^* (T) \rtimes \mathcal W^G$-representation uses only the 
homology in top degree. Also the $\Gamma_\pi$-actions on these modules are defined in the
same way, so the two cocycles can be chosen equal.

\emph{Step 4. Upon applying $X \mapsto (X /\!/ \Gamma)^\natural_2$ to the commutative diagram 
\eqref{eq:S.13} we obtain the corresponding diagram for $G \rtimes \Gamma$.}\\
Here $\natural$ denotes the family of 2-cocycles constructed in steps 2 and 3.
For $(T /\!/ \mathcal W^G )_2$ and Irr$(\mathcal H (G))$ we know from Lemmas \ref{lem:Clifford}
and \ref{lem:Clifford_algebras}
that this procedure yields the correct parameters. That it works for Kazhdan--Lusztig--Reeder
parameters was checked in the last part of the proof of Theorem \ref{thm:S.5}. Since we
have the same families of 2-cocycles on all three $\Gamma$-sets, the maps from \eqref{eq:S.13} 
can be lifted in a natural way to the diagram for $G \rtimes \Gamma$.

The ordinary quotient is embedded in $(T /\!/ \mathcal W^G \rtimes \Gamma )_2$ as the
collection of pairs $\big( t,\triv_{(\cW^G \rtimes \Gamma)_t} \big)$. By an obvious
generalization of \eqref{eq:trivWaff} these correspond to the affine Springer parameters
$(t,x=1,\rho = \triv)$. It is clear from the above construction that they are mapped to
KLR parameters $(\Phi,\triv)$ with $\Phi (\mathbf I_F \times \SL_2 (\C)) = 1$ and 
$\Phi (\varpi_F) = t$. By Proposition \ref{prop:spherical} the latter correspond
to the spherical irreducible $\cH (G) \rtimes \Gamma$-modules. 
\end{proof}

\section{Main result (general case)}
\label{sec:general}

We return to the notation from Section \ref{sec:main}. In general the group 
$H = \Z_G (\im c^\fs)$ need not be connected. It is well-known, and already used several 
times in the proof of Proposition \ref{prop:S.2}, that the short exact sequence
\begin{equation} \label{eq:splitH}
1 \to H^\circ / \Z(H^\circ) \to H / \Z(H^\circ) \to \pi_0 (H) \to 1
\end{equation}
is split. More precisely, any choice of a pinning of $H^\circ$ (a Borel subgroup,
a maximal torus, and a nontrivial element in every root subgroup associated to
a simple root) determines such a splitting. We fix a pinning with $T$ as maximal torus, 
and with it we fix actions of $\pi_0 (H)$ on $H^\circ$, on the Dynkin diagram of 
$H^\circ$ and on the Weyl group of $H^\circ$. Lemma \ref{lem:centrals} shows that
\begin{equation}\label{eq:WsH}
W^\fs = \mathcal W^G_{\im c^\fs} \cong \mathcal W^{H^\circ} \rtimes \pi_0 (H) .
\end{equation}
According to \cite[Section 8]{Roc} $\Irr (\mathcal G)^\fs$ is naturally in bijection with
$\Irr (\mathcal H (H))$, where 
\begin{equation}\label{eq:HeckeH}
\mathcal H (H) \cong \mathcal H (H^\circ) \rtimes \pi_0 (H) .
\end{equation}
Now we have collected all the material that is needed to prove our main result 
(Theorem \ref{thm:I3} in our Introduction).

\begin{thm}\label{thm:main}
Let $\cG$ be a  split reductive $p$-adic group and let $\fs = [\mathcal T, \chi]_\cG$ 
be a point in the Bernstein spectrum of the principal series of $\cG$. Assume that 
the residual characteristic $p$ satisfies the conditions of \cite[Remark 4.13]{Roc}
when $H \neq G$. Then there is a commutative triangle of natural bijections 
\[ 
\xymatrix{   & (T^\fs/\!/W^\fs)_2 \ar[dr]\ar[dl] & \\  
\Irr(\mathcal{G})^\fs    \ar[rr] & & \{\KLR\:\:\mathrm{parameters}\}^\fs / H} 
\] 
The slanted maps are generalizations of the slanted maps in Theorem \ref{thm:ps}
and the horizontal map stems from Theorem \ref{thm:S.5}. 

We denote the irreducible $\cG$-representation associated to a KLR parameter 
$(\Phi,\rho)$ by $\pi (\Phi,\rho)$.
\begin{enumerate}
\item The infinitesimal central character of $\pi (\Phi,\rho)$ is the $H$-conjugacy class
\[
\Phi \big( \varpi_F , \matje{q^{1/2}}{0}{0}{q^{-1/2}} \big) 
\in c(H)_{\ss} \cong T^\fs / W^\fs .
\]
\item $\pi (\Phi,\rho)$ is tempered if and only if $\Phi (\mathbf W_F)$ is bounded, which
is the case if and only if $\Phi (\varpi_F)$ lies in a compact subgroup of $H$.
\item $\pi (\Phi,\rho)$ is essentially square-integrable if and only if $\Phi 
(\mathbf W_F \times \SL_2 (\C))$ is not contained in any proper Levi subgroup of $H^\circ$.
\end{enumerate}
\end{thm}
Recall that $\cG$ only has irreducible square-integrable representations if $\Z (\cG)$ 
is compact. A $\cG$-representation is called essentially square-integrable if its 
restriction to the derived group of $\cG$ is square-integrable. This is more general 
than square-integrable modulo centre, because for that notion $\Z(\cG)$ needs to act by 
a unitary character. 
\begin{proof}
The larger part of the commutative triangle was already discussed in \eqref{eq:WsH}, 
\eqref{eq:HeckeH} and Theorem \ref{thm:S.4}. It remains to show that the set 
\{KLR parameters$\}^\fs / H$ (as defined on page \pageref{eq:defKLRparameter}) 
is naturally in bijection with
\{KLR parameters for $H^\circ \rtimes \pi_0 (H) \}^{\unr} / H^\circ \rtimes \pi_0 (H)$.

By \eqref{eq:splitH} we are taking conjugacy classes with respect to the
group $H / \Z(H^\circ)$ in both cases.
It is clear from the definitions that that in both sets the ingredients $\Phi$
are determined by the semisimple element $\Phi (\varpi_F) \in H$. This provides the
desired bijection between the $\Phi$'s in the two collections, so let us focus on the
ingredients $\rho$. 

For $(\Phi,\rho) \in \{\text{KLR parameters}\}^\fs$ the irreducible representation
$\rho$ of the component group $\pi_0 (\Z_H (\Phi)) = \pi_0 (\Z_G (\Phi))$ must appear in 
$H_* \big( \mathcal B_G^{\Phi (\mathbf W_F \times B_2)} ,\C \big)$. By Proposition
\ref{prop:UP}.3 this space is isomorphic, as a $\pi_0 (\Z_G (\Phi))$-representation, to
a number of copies of
\[
\mathrm{Ind}_{\pi_0 (\Z_{H^\circ} (\Phi))}^{\pi_0 (\Z_G (\Phi))} H_* \big( 
\mathcal B_{H^\circ}^{\Phi (\mathbf W_F \times B_2)} ,\C \big) .
\]
Hence the condition on $\rho$ is equivalent to requiring that every irreducible
$\pi_0 (\Z_{H^\circ} (\Phi))$-subrepresentation of $\rho$ appears in
$H_* \big( \mathcal B_{H^\circ}^{\Phi (\mathbf W_F \times B_2)} ,\C \big)$. That 
is exactly the condition on $\rho$ in an unramified KLR parameter for 
$H^\circ \rtimes \pi_0 (H)$. This establishes the properties of the commutative diagram.\\
(1) From the construction in Section \ref{sec:repAHA} we see that the 
$\mathcal H (H^\circ)$-module with Kazhdan--Lusztig parameter $(t_q,x,\rho_q)$ 
has central character
\[
t_q = \Phi \big( \varpi_F , \matje{q^{1/2}}{0}{0}{q^{-1/2}} \big) 
\in c(H^\circ)_{\ss} \cong T / \mathcal W^{H^\circ} .
\]
It follows that the $\mathcal H (H)$-module with parameter $(\Phi,\rho)$ has central 
character $t_q \in c(H)_{\ss} \cong T^\fs / W^\fs$. The corresponding $\cG$-representation 
is obtained via a suitable Morita equivalence, which by definition transforms the central 
character into the infinitesimal character.\\
(2) It was checked in \cite{BHK} that a member of $\Irr (\cG)^\fs$ is tempered if and 
only if the corresponding $\mathcal H (H)$-module is tempered.

For $z \in \C^\times$ we put $V(z) = \log |z|$. According to \cite[Theorem 8.2]{KL}
the $\mathcal H (H )$-module with parameter $(\Phi,\rho)$ is $V$-tempered if and only
if all the eigenvalues of $t = \Phi (\varpi_F)$ on Lie $H$ (via the adjoint representation)
have absolute value 1. That \cite{KL} works with simply connected complex groups is 
inessential to the argument, it also applies to our $H$. But $V$-tempered (for this $V$) 
means only that the restriction of the $\mathcal H (H)$-module to the subalgebra 
$\mathcal H (H^\circ_{\der})$ is tempered, where $H^\circ_{der}$ denotes the derived group of 
$H^\circ$. The $\mathcal H (H)$-module is tempered if and only if moreover the subalgebra
$\mathcal H (\Z(H^0))$ acts on it by a unitary character. This is the case if and only if all 
the eiqenvalues of $t$ (in some realization of $H^0$ as complex matrices) have absolute 
value 1. That in turn means that $t$ lies in the maximal compact subgroup of $T^\fs$.

Since $\Phi (\mathbf W_F)$ is generated by the finite group $\Phi (\mathbf I_F)$ and 
$t = \Phi (\varpi_F)$, the above condition on $t$ is equivalent to boundedness of
$\Phi (\mathbf W_F)$.\\
(3) This is similar to part 2, it follows from \cite[Theorem 8.3]{KL} and \cite{BHK}. 
\end{proof}

Notice that the bijectons in Theorem \ref{thm:main} satisfy the statements 
(1)--(4) from Section \ref{sec:statement} by construction. We will check the 
properties 1--6 in the upcoming sections.

Recall that $\Irr( \cG, \mathcal T)$ is the space of all irreducible 
$\cG$-representations in the principal series. Considering Theorem \ref{thm:main}
for all Bernstein components in the principal series simultaneously, we will establish
the local Langlands correspondence for this class of representations.

\begin{cor}\label{cor:LLC}
Let $\cG$ be a split reductive $p$-adic group, with restrictions on the residual
characteristic as in \cite[Remark 4.13]{Roc}. Then the local Langlands correspondence
holds for the irreducible representations in the principal series of $\cG$, and
it is the bottom row in a commutative triangles of natural bijections
\[ 
\xymatrix{   & (\Irr \, \mathcal T /\!/ \mathcal W^G )_2 \ar[dr]\ar[dl] & \\  
\Irr ( \cG, \mathcal T)    \ar[rr] & & \{\text{KLR parameters for } G \} / G} 
\]
The restriction of this diagram to a single Bernstein component recovers Theorem 
\ref{thm:main}. In particular the bottom arrow generalizes the Kazhdan--Lusztig
parametrization of the irreducible $\cG$-representations in the unramified
principal series.
\end{cor}
\begin{proof}
Let us work out what happens if in Theorem \ref{thm:main} we take the union over 
all Bernstein components $\fs \in \mathfrak B (\cG,\cT)$.

On the left we obtain (by definition) the space $\Irr (\cG ,\mathcal T)$. 
Notice that in Theorem \ref{thm:main}, instead of \{KLR parameters$\}^\fs / H$ we 
could just as well take $G$-conjugacy classes of KLR parameters $(\Phi,\rho)$ such 
that $\Phi \big|_{\mathbf I_F}$ is $G$-conjugate to $c^\fs$. The union of those clearly is 
the space of all $G$-conjugacy classes of KLR parameters for $G$. For the space at
the top of the diagram, choose a smooth character $\chi_\fs$ of $\mathcal T$ such that 
$(\mathcal T,\chi_\fs) \in \fs$.
Recall from Section \ref{sec:intro3} that the $T$ in $(T /\!/ W^\fs )_2$ is actually 
\[
T^\fs := \{ \chi_\fs \otimes t \mid t \in T \},
\]
where $t$ is considered as an unramified character of $\mathcal T$. On the other hand,
$\Irr \, \mathcal T$ can be obtained by picking representatives $\chi_\fs$ for
$( \Irr \, \mathcal T )/ T$ and taking the union of the corresponding $T^\fs$.
Two such spaces $T^\fs$ give rise to the same Bernstein component for $\cG$ 
if and only if they are conjugate by an element of $N_{\cG} (\mathcal T)$, 
or equivalently by an element of $\mathcal W^G$. Therefore
\[
(\Irr \, \mathcal T /\!/ \mathcal W^G )_2 = 
\big( \bigcup_{\fs \in \mathfrak B (\cG,\cT)} \mathcal W^G \cdot T^\fs /\!/ \mathcal W^G \big)_2 
= \bigcup_{\fs \in \mathfrak B (\cG,\cT)} \big( T^\fs /\!/ W^\fs \big)_2 .
\]
Hence the union of the spaces in the commutative triangles from Theorem \ref{thm:main} 
is as desired. The right slanted arrows in these triangles combine to a bijection
\[
(\Irr \, \mathcal T /\!/ \cW^G )_2 \to \{\text{KLR parameters for } G \} / G ,
\] 
because the $\cW^G$-action is compatible with the $G$-action. Suppose that $(\cT ,\chi'_\fs)$
is another base point for $\fs$. Up to an unramified twist, we may assume that $\chi'_\fs =
w \chi_\fs$ for some $w \in \cW^G$. Then the Hecke algebras $\cH (H)$, and $\cH (H')$ are
isomorphic by a map that reflects conjugation by $w$ and it was checked in \cite[Section 6]{R}
that this is compatible with the bijections between $\Irr (\cG)^\fs , \Irr (\cH (H))$ and
$\Irr (\cH (H'))$. It follows that the bottom maps in the triangles from Theorem \ref{thm:main} 
paste to a bijection 
\[
\Irr (\cG ,\cT) \to \{\text{KLR parameters for } G \} / G .
\]
Finally, the map
\[
(\Irr \, \mathcal T /\!/ \cW^G )_2 \to \Irr (\cG ,\cT)
\]
can be defined as the composition of the other two bijections in the above triangle. Then
it is the combination the left slanted maps from Theorem \ref{thm:main} because the triangles
over there are commutative. 
\end{proof}

In \cite[Section 10]{BorAut} Borel stated several "desiderata" for the local Langlands
correspondence. The properties (1), (2) and (3) of Theorem \ref{thm:main} prove some of 
these, whereas the others involve representations outside the principal series and 
therefore fall outside the scope of our results.

\section{The labelling by unipotent classes} 
\label{sec:unip}

Let $\fs \in \mathfrak B (\cG,\cT)$ and construct $c^\fs$ as in Section \ref{sec:Lp}. 
By Theorem \ref{thm:main} we can parametrize $\Irr (\cG)^\fs$ with
$H$-conjugacy classes of KLR parameters $(\Phi,\rho)$ such that $\Phi \big|_{\fo_F^\times} = c^\fs$.
We note that \{KLR parameters$\}^\fs$ is naturally labelled by the unipotent classes in $H$: 
\begin{equation}
\{\text{KLR parameters} \}^{\fs,[x]} := \big\{ (\Phi,\rho) \mid 
\Phi \big( 1, \matje{1}{1}{0}{1} \big) \text{ is conjugate to } x \big\} .
\end{equation}
In this way we can associate to any of the parameters in Theorem \ref{thm:main}
a unique unipotent class in $H$:
\begin{equation} \label{eq:labelling}
\Irr (\cG )^\fs = \bigcup\nolimits_{[x]} \Irr (\cG )^{\fs,[x]} ,\quad
(T^\fs /\!/ W^\fs )_2 = \bigcup\nolimits_{[x]} (T^\fs /\!/ W^\fs )_2^{[x]} .
\end{equation}
Via the affine Springer correspondence from Section \ref{sec:affSpringer} the set of
equivalence classes in \{KLR parameters$\}^\fs$ is naturally in bijection with
$(T^\fs /\!/ W^\fs )_2$. Recall from Section \ref{sec:extquot} that 
\[
\widetilde{T^\fs} = \{ (w,t) \in W^\fs \times T^\fs \mid w t = t \}
\]
and $T^\fs /\!/ W^\fs = \widetilde{T^\fs} / W^\fs$. In view of Section \ref{sec:extquots} 
$(T^\fs /\!/ W^\fs )_2$ is also in bijection with $T^\fs /\!/ W^\fs$, albeit not naturally. 

Only in special cases a canonical bijection $T^\fs /\!/ W^\fs \to (T^\fs /\!/ W^\fs)_2$ 
is available. For example when $G = \GL_n (\C)$, the finite group $\cW^H_t$  
is a product of symmetric groups: in this case there is a canonical  $c$-$\Irr$ system,    
according to the classical theory of Young tableaux. 

In general it can already be hard to define any suitable map from \{KLR parameters$\}^\fs$
to $T^\fs /\!/ W^\fs$, because it is difficult to compare the parameters $\rho$ 
for different $\Phi$'s. It goes better the other way round and with $\Irr (\cG )^\fs$
as target. In this way will transfer the labellings \eqref{eq:labelling} to
$T^\fs /\!/ W^\fs$.

>From \cite[Section 8]{Roc} we know that $\Irr (\cG )^\fs$ is naturally in 
bijection with the equivalence classes of irreducible representations of the
extended affine Hecke algebra $\mathcal H (H)$. To relate it to $T^\fs /\!/ W^\fs$
the parametrization of Kazhdan, Lusztig and Reeder is unsuitable, it is more
convenient to use the methods developed in \cite{Opd,Sol}. 

To describe it, we fix some notation. Choose a Borel subgroup $B \subset H^\circ$ 
containing $T$. Let $P$ be a set of roots of $(H^\circ,T)$ which are simple with
respect to $B$ and let $R_P$ be the root system that they span. They determine
a parabolic subgroup $W_P \subset W^\fs$ and a subtorus
\[
T^P := \{ t \in T \mid \alpha (t) = 1 \; \forall \alpha \in R_P \}^\circ .
\]
In \cite[Theorem 3.3.2]{Sol} $\Irr (\mathcal H (H))$ is mapped, in a natural
finite-to-one way, to equivalence classes of triples $(P,\delta,t)$. Here
$P$ is as above, $t \in T^P$ and $\delta$ is a discrete series
representation of a parabolic subalgebra $\mathcal H_P$ of $\mathcal H (H)$.
This $t$ is the same as in the affine Springer parameters.

The pair $(P,\delta)$ gives rise to a \emph{residual coset} $L$ in the sense of
\cite[Appendix A]{Opd}. Explicitly, it is the translation of $T^P$ by an
element $cc(\delta) \in T$ that represents the central character of $\delta$ 
(a $W_P$-orbit in a subtorus $T_P \subset T$). 
The element $cc(\delta) t \in L$ corresponds to $t_q$. 
The collection of residual cosets is stable under the action of $W^\fs$.

\begin{prop}\label{prop:residual}
\begin{enumerate}
\item There is a natural bijection between
\begin{itemize}
\item $H$-conjugacy classes of Langlands parameters $\Phi$ with 
$\Phi \big|_{\mathbf I_F} = c^\fs$;
\item $W^\fs$-conjugacy classes of pairs $(t_q,L)$ with $L$ a residual coset for
$\mathcal H (H)$ and $t_q \in L$.
\end{itemize}
\item Let $Y^P$ be the union of the residual cosets of $T^P$. The stabilizer of $Y^P$
in $W^\fs$ is the stabilizer of $R_P$.
\item Suppose that $w \in W^\fs$ fixes $cc(\delta)$. Then $w$ stabilizes $R_P$.
\end{enumerate}
\end{prop}
\begin{proof}
(1) Opdam constructed the maps in both directions for $\mathcal H (H^\circ)$. 
To go from $\mathcal H (H^\circ)$ to $\mathcal H (H)$ is easy,
one just has to divide out the action of $\pi_0 (H)$ on both sides. 

Let us describe the maps for $H^\circ$ more explicitly. 
To a residual coset $L$ Opdam \cite[Proposition B.3]{Opd} 
associates a unipotent element $x \in B$ such that $l x l^{-1} = x^q$ for all $l \in L$. 
Then $\Phi$ is a Langlands parameter with data $t_q,x$.

For the opposite direction we may assume that 
\[
\Phi \big( \mathbf W_F ,\matje{z}{0}{0}{z^{-1}} \big) \subset T \quad \forall z \in \C^\times
\]
and that $x = \Phi \big( 1,\matje{1}{1}{0}{1} \big) \in B$. Then 
\[
T^P := \Z_T \big(\Phi (\mathbf I_F \times \SL_2 (\C)) \big)^\circ = 
\Z_T \big( \Phi (\SL_2 (\C)) \big)^\circ
\]
is a maximal torus of $\Z_{H^\circ}\big(\Phi (\mathbf I_F \times \SL_2 (\C)) \big)$. We take
\[
t_q = \Phi \Big(\varpi_F , \matje{q^{1/2}}{0}{0}{q^{-1/2}} \Big) 
\quad \text{and} \quad L = T^P t_q .
\]
This is essentially \cite[Proposition B.4]{Opd}, but our way to write it down avoids
Opdam's assumption that $H^\circ$ is simply connected.\\
(2) Clear, because any element that stabilizes $Y^P$ also stabilizes $T^P$.\\
(3) Since $cc(\delta)$ represents the central character of a discrete series representation
of $\cH_P$, at least one element (say $r$) in its $W_P$-orbit lies in the obtuse 
negative cone in the subtorus $T_P \subset T$ (see Lemma 2.21 and Section 4.1 of \cite{Opd}).
That is, $\log |r|$ can be written as $\sum_{\alpha \in P} c_\alpha \alpha^\vee$ with
$c_\alpha < 0$. Some $W_P$-conjugate $w'$ of $w \in W^\fs$ fixes $r$ and hence $\log |r|$.
But an element of $W^\fs$ can only fix $\log |r|$ if it stabilizes the
collection of coroots $\{ \alpha^\vee \mid \alpha \in P\}$. It follows that $w'$ and $w$ 
stabilize $R_P$.
\end{proof}

In particular the above natural bijection associates to any $W^\fs$-conjugacy class of
residual cosets $L$ a unique unipotent class $[x]$ in $H$. 
Conversely a unipotent class $[x]$ can correspond to more than one $W^\fs$-conjugacy 
class of residual cosets, at most the number of connected components of $\Z_T (x)$
if $x \in B$.

Let $\mathfrak U^\fs \subset B$ be a set of representatives for the unipotent classes in $H$. 
For every $x \in \mathfrak U^\fs$ we choose an algebraic homomorphism
\[
\gamma_x \colon \SL_2 (\C) \to H \quad \text{with} \quad \gamma_x \matje{1}{1}{0}{1} = x
\quad \text{and} \quad \gamma_x \matje{z}{0}{0}{z^{-1}} \in T .
\]
By Lemma \ref{lem:cocharindep} all choices for $\gamma_x$ are conjugate. For each
$x \in \mathfrak U^\fs$ we define
\[
\{ \text{KLR parameters} \}^{\fs,x} = \{ (\Phi,\rho) \mid \Phi 
\big|_{\mathbf I_F \times \SL_2 (\C)} = c^\fs \times \gamma_x , \Phi (\varpi_F) \in T \} .
\]
We endow this set with the topology that ignores $\rho$ and is the Zariski topology
with respect to $\Phi (\varpi_F)$. In this way
\begin{equation} \label{eq:UsParam}
\bigsqcup_{x \in \mathfrak U^\fs} \{ \text{KLR parameters} \}^{\fs,x}
\end{equation}
becomes a nonseparated algebraic variety with maximal separated quotient 
$\sqcup_{x \in \mathfrak U^\fs} \Z_T (\im \gamma_x)$. Notice that \eqref{eq:UsParam} 
contains representatives for all equivalence classes in \{KLR parameters$\}^\fs$.

\begin{prop}\label{prop:mus}
There exists a continuous bijection
\[
\mu^\fs : T^\fs /\!/ W^\fs \to \Irr (\cG )^\fs 
\]
such that:
\begin{enumerate}
\item The diagram
\[
\xymatrix{
T^\fs /\!/ W^\fs \ar[r]^{\mu^\fs} \ar[d]^{\rho^\fs} & \Irr (\cG )^\fs \ar[r] &
\{\text{KLR parameters} \}^\fs / H \ar[d] \\
T^\fs / W^\fs & & c (H)_{\ss} \ar[ll] } 
\]
commutes. Here the unnamed horizontal maps are those from Theorem \ref{thm:main} 
and the right vertical arrow sends $(\Phi,\rho)$ to the $H$-conjugacy class of 
$\Phi (\varpi_F)$.
\item For every unipotent element $x \in H$ the preimage
\[
(T^\fs /\!/ W^\fs )^{[x]} := (\mu^\fs )^{-1} \big( \Irr (\cG )^{\fs,[x]} \big)
\]
is a union of connected components of $T^\fs /\!/ W^\fs$.
\item Let $\epsilon$ be the map that makes the diagram
\[
\xymatrix{ 
T^\fs /\!/ W^\fs \ar[rr]^\epsilon \ar[dr]^{\mu^\fs} & & 
(T^\fs /\!/ W^\fs )_2 \ar[dl] \\
& \Irr (\cG )^\fs & }
\]
commute. Then $\epsilon$ comes from a $c$-$\Irr$ system.
\item $T^\fs /\!/ W^\fs \xrightarrow{\; \mu^\fs \;} \Irr (\cG )^\fs \to 
\{\text{KLR parameters} \}^\fs / H$ lifts to a map 
\[
\widetilde{\mu^\fs} : \widetilde{T^\fs} \to 
\bigsqcup_{x \in \mathfrak U^\fs} \{ \text{KLR parameters} \}^{\fs,x}
\]
such that the restriction of $\widetilde{\mu^\fs}$ to any connected component of 
$\widetilde{T^\fs}$ is algebraic and an isomorphism onto its image.
\end{enumerate}
\end{prop}
\begin{proof} 
Proposition \ref{prop:residual}.1 yields a natural finite-to-one map from $\Irr 
(\mathcal H (H))$ to $W^\fs$-conjugacy classes $(t_q,L)$, namely 
\begin{equation}\label{eq:pitqL}
\pi (\Phi,\rho) \mapsto \Phi \mapsto (t_q,L) .
\end{equation} 
In \cite[Theorem 3.3.2]{Sol} this map was obtained in 
a different way, which shows better how the representations depend on the parameters 
$t,t_q,L$. That was used in \cite[Section 5.4]{Sol} to find a continuous bijection
\begin{equation}\label{eq:mus}
\mu^\fs : T^\fs /\!/ W^\fs \to \Irr (\mathcal H (H)) \cong \Irr (\cG )^\fs .
\end{equation}
The strategy is essentially a step-by-step creation of a $c$-$\Irr$ system for
$T^\fs /\!/ W^\fs$ and $\mathcal H (H)$, only the condition on the unit element
and the trivial representation is not considered in \cite{Sol}. Fortunately
there is some freedom left, which we can exploit to ensure that $\mu^\fs (1,T^\fs)$ 
is the family of spherical $\mathcal H (H)$-representations, see Section 
\ref{sec:spherical}. 
This is possible because every principal series representation of $\mathcal H (H)$ has 
a unique irreducible spherical subquotient, so choosing that for $\mu^\fs (1,t)$ does
not interfere with the rest of the construction. Via Theorem \ref{thm:main} we can 
transfer this $c$-$\Irr$ system to a $c$-$\Irr$ system for the two extended 
quotients of $T^\fs$ by $W^\fs$, so property (3) holds.

By construction the triple $(P,\delta,t)$ associated to the representation 
$\mu^\fs (w,t)$ has the same $t \in T$, modulo $W^\fs$. 
That is, property (1) is fulfilled.

Furthermore $\mu^\fs$ sends every connected component of $T^\fs /\!/ W^\fs$ to a family
of representations with common parameters $(P,\delta)$. Hence these representation are
associated to a common residual coset $L$ and to a common unipotent class $[x]$,
which verifies property (2).

Let $\mathbf c$ be a connected component of $(T^\fs /\!/ W^\fs )^{[x]}$, with $x \in
\mathfrak U^\fs$. The proof of Proposition \ref{prop:residual}.1 shows that 
$\mathbf c$ can be realized in $\Z_T (\im \gamma_x)$. In other words, we can find a 
suitable $w = w (\mathbf c) \in W^\fs$ with $T^w \subset \Z_T (\im \gamma_x)$. Then there 
is a connected component $T^w_{\mathbf c}$ of $T^w$ such that
\[
\begin{split}
& \mathbf c := \big( w,T^w_{\mathbf c} / \Z (w,\mathbf c) \big) ,\\ 
& \text{where } \Z (w,\mathbf c) = \{ g \in \Z_{W^\fs}(w) \mid g \cdot T^w_{\mathbf c} = 
T^w_{\mathbf c} \} . 
\end{split}
\]
In this notation $\widetilde{\mathbf c} := (w,T^w_{\mathbf c})$ is a connected component of
$\widetilde{T^\fs}$. We want to define $\widetilde{\mu^\fs} : \widetilde{\mathbf c} \to
\{\text{KLR parameters} \}^{\fs,x}$. For every $[w,t] \in \mathbf c ,\;
\mu^\fs [w,t]$ determines an equivalence class in \{KLR parameters$\}^{\fs,x}$. Any
$(\Phi,\rho)$ in this equivalence class satisfies $\Phi (\varpi_F) \in \Z_T (\im \gamma_x)
\cap W^\fs t$. Hence there are only finitely many possibilities for $\Phi (\varpi_F)$,
say $t_1,\ldots,t_k$. For every such $t_i$ there is a unique 
$(\Phi_i,\rho_i) \in \{\text{KLR parameters} \}^{\fs,x}$ with $\Phi_i (\varpi_F) = t_i$ and 
$\pi (\Phi_i,\rho_i) \cong \mu^\fs [w,t]$. Every element of $\widetilde{\mathbf c}$ lying 
over $[w,t] \in \mathbf c$ is of the form $(w,t_i)$. (Not every $t_i$ is eligible though, 
for some we would have to modify $w$.) We put
\[
\widetilde{\mu^\fs} (w,t_i) := (\Phi_i ,\rho_i) .
\]
Since the $\rho$'s are irrelevant for the topology on \{KLR parameters$\}^{\fs,x} ,\;
\widetilde{\mu^\fs} (\widetilde{\mathbf c})$ is homeomorphic to $T^w_{\mathbf c}$ and 
$\widetilde{\mu^\fs} : \widetilde{\mathbf c} \to \widetilde{\mu^\fs} (\widetilde{\mathbf c})$
is an isomorphism of affine varieties. This settles the final property (4).
\end{proof}

\section{Correcting cocharacters and L-packets}
\label{sec:cochar}

In this section we construct the correcting cocharacters on the extended quotient
$T^\fs /\!/ W^\fs$. As conjectured in Section \ref{sec:statement}, these show how
to determine when two elements of $T^\fs /\!/ W^\fs$ give rise to $\cG$-representations
in the same L-packets.

Every KLR parameter $(\Phi,\rho)$ naturally determines a cocharacter $h_\Phi$ 
and elements $\theta (\Phi,\rho,z) \in T^\fs$ by
\begin{equation}\label{eq:defhPhi}
\begin{aligned}
& h_\Phi (z) = \Phi \big( 1,\matje{z}{0}{0}{z^{-1}} \big) ,\\
& \theta (\Phi,\rho,z) = \Phi \big( \varpi_F, \matje{z}{0}{0}{z^{-1}} \big) =
\Phi (\varpi_F) h_\Phi (z) .
\end{aligned}
\end{equation}
Although these formulas obviously do not depend on $\rho$, it turns out to be convenient to 
include it in the notation anyway.
However, in this way we would end up with infinitely many correcting cocharacters, most
of them with range outside $T$. To reduce to finitely many cocharacters with values in $T$, 
we will restrict to KLR parameters associated to $x \in \mathfrak U^\fs$ \eqref{eq:UsParam}.

Recall that part (2) of Proposition \ref{prop:mus} determines a labelling of the connected 
components of $T^\fs /\!/ W^\fs$ by unipotent classes in $H$. This enables us to define the 
correcting cocharacters: for a connected component $\mathbf c$ of $T^\fs /\!/ W^\fs$ with 
label (represented by) $x \in \mathfrak U^\fs$ we take the cocharacter 
\begin{equation}\label{eq:defhx}
h_{\mathbf c} = h_x : \C^\times \to T ,\quad h_x (z) = \gamma_x \matje{z}{0}{0}{z^{-1}} .
\end{equation}
Let $\widetilde{\mathbf c}$ be a connected component of $\widetilde{T^\fs}$ that projects
onto $\mathbf c$. We define
\begin{equation} \label{eq:defThetaz}
\begin{aligned}
& \widetilde{\theta_z} : \widetilde{\mathbf c} \to T^\fs ,& & (w,t) \mapsto 
\theta \big( \widetilde{\mu^\fs}(w,t),z \big) , \\
& \theta_z : \mathbf c \to T^\fs / W^\fs ,& & [w,t] \mapsto W^\fs \widetilde{\theta_z}(w,t) . 
\end{aligned}
\end{equation}
For $\widetilde{\mathbf c}$ as in the proof of Proposition \ref{prop:mus}, which we can
always achieve by adjusting by element of $W^\fs$, our construction results in
\[
\widetilde{\theta_z} (w,t) = t \, h_{\mathbf c}(z) .
\]

\begin{lem}\label{lem:Lpackets}
Let $[w,t],[w',t'] \in T^\fs /\!/W^\fs$. Then $\mu^\fs [w,t]$ and $\mu^\fs [w',t']$ are
in the same L-packet if and only if
\begin{itemize}
\item $[w,t]$ and $[w',t']$ are labelled by the same unipotent class in $H$;
\item $\theta_z [w,t] = \theta_z [w',t']$ for all $z \in \C^\times$.
\end{itemize}
\end{lem}
\begin{proof}
Suppose that the two $\cG$-representations $\mu^\fs [w,t] = \pi (\Phi,\rho)$ and \\
$\mu^\fs [w',t'] = \pi (\Phi',\rho')$ belong to the 
same L-packet. By definition this means that $\Phi$ and $\Phi'$ are $G$-conjugate. 
Hence they are labelled by the same unipotent class, say $[x]$ with $x \in \mathfrak U^\fs$. 
By choosing suitable representatives we may assume that $\Phi = \Phi'$ and that 
$\{(\Phi,\rho),(\Phi,\rho')\} \subset \{\text{KLR parameters} \}^{\fs,x}$. Then
$\theta (\Phi,\rho,z) = \theta (\Phi,\rho',z)$ for all $z \in \C^\times$. Although
in general $\theta (\Phi,\rho,z) \neq \widetilde{\theta_z} (w,t)$, they differ only 
by an element of $W^\fs$. Hence $\theta_z [w,t] = \theta_z [w',t']$ for all $z \in \C^\times$.

Conversely, suppose that $[w,t],[w',t']$ fulfill the two conditions of the lemma. Let 
$x \in \mathfrak U^\fs$ be the representative for the unipotent class which labels them.
By Proposition \ref{prop:residual}.1 we may assume that $T^w \cup T^{w'} \subset 
\Z_T (\im \gamma_x)$. Then 
\[
\widetilde{\theta_z} [w,t] = t \, h_x (z) \quad \text{and} \quad
\widetilde{\theta_z} [w',t'] = t' \, h_x (z)
\]
are $W^\fs$ conjugate for all $z \in \C^\times$. As these points depend continuously on $z$
and $W^\fs$ is finite, this implies that there exists a $v \in W^\fs$ such that
\[
v (t \, h_x (z)) = t' \, h_x (z) \quad \text{for all } z \in \C^\times .
\]
For $z = 1$ we obtain $v(t) = t'$, so $v$ fixes $h_x (z)$ for all $z$. Via the Proposition
\ref{prop:residual}.1, $h_x (q^{1/2})$ becomes an element $cc(\delta)$ for a residual coset
$L_x$. By parts (2) and (3) of Proposition \ref{prop:residual} $v$ stabilizes the collection
of residual cosets determined by $x$, namely the connected components of $\Z_T (\im \gamma_x)
h_x (q^{1/2})$. 

Let $(t_q,L), (t'_q,L')$ be associated to $\mu^\fs [w,t], \mu^\fs [w',t']$
by \eqref{eq:pitqL}. Then $t_q = t h_x (q^{1/2})$ and $t'_q = t' h_x (q^{1/2})$, so the above
applies. Hence $v$ sends $L$ to another residual coset determined by $x$. As $v(L)$ 
contains $t'_q$, it must be $L'$. Thus $(t_q,L)$ and $(t'_q,L')$ are $W^\fs$-conjugate, which
by Proposition \ref{prop:residual}.1 implies that they correspond to conjugate Langlands
parameters $\Phi$ and $\Phi'$. So $\mu^\fs [w,t]$ and $\mu^\fs [w',t']$ are
in the same L-packet.
\end{proof}

\begin{cor}\label{cor:properties}
Properties 1--6 from Section \ref{sec:statement} hold in for $\mu^\fs$ as in 
Proposition \ref{prop:mus}, with the morphism $\theta_z$ from \eqref{eq:defThetaz}
and the labelling by unipotent classes in $H$. 

Together with Theorem \ref{thm:main} this proves the conjectures from Section 
\ref{sec:statement} for all Bernstein components in the principal series of a split reductive 
$p$-adic group (with mild restrictions on the residual characteristic).
\end{cor}
\begin{proof}
Property 1 follows from Theorem \ref{thm:main}.2 and Proposition \ref{prop:mus}.1.
The definitions of \eqref{eq:defhx} and \eqref{eq:defThetaz} establish
property 5. The construction of $\theta_z$, in combination with Theorem \ref{thm:main}.1 
and Proposition \ref{prop:mus}.1, shows that properties 2,3 and 4 are fulfilled. 
Property 6 is none other than Lemma \ref{lem:Lpackets}.
\end{proof}

\appendix
\section{Geometric equivalence}

Let $X$ be a complex affine variety and let $k=\mathcal{O}(X)$ be its coordinate algebra.
Equivalently, $k$ is a unital algebra over the complex numbers which 
is commutative, finitely generated, and nilpotent-free. A $k$-algebra is an algebra $A$ 
over the complex numbers which is a $k$-module
(with an evident compatibility between the algebra structure of $A$ and the $k$-module 
structure of $A$). $A$ is of finite type if as a $k$-module
$A$ is finitely generated. This appendix will introduce  ---  for finite type $k$-algebras --- 
a weakening of Morita equivalence called \emph{geometric equivalence}.\\

The new equivalence relation preserves the primitive ideal space
(i.e. the set of isomorphism classes of simple $A$-modules) and the periodic cyclic homology. 
However, the new equivalence relation
permits a tearing apart of strata in the primitive ideal space which is not allowed by 
Morita equivalence. The ABP conjecture
(i.e. the conjecture stated in Part(1) of this paper) asserts that the finite type algebra 
which Bernstein constructs for any given Bernstein component of a reductive p-adic group 
is geometrically equivalent to the coordinate algebra of the associated extended quotient  --- 
and that the geometric equivalence can be
chosen so that the resulting bijection between the Bernstein component and the extended 
quotient has properties as in the statement of ABP.

\subsection{$k$-algebras}

Let $X$ be a complex affine variety and $k=\mathcal{O}(X)$ its coordinate algebra.

A \emph{$k$-algebra} is a $\mathbb{C}$-algebra $A$ such that $A$ is a unital (left) $k$-module with:
\[
\lambda(\omega a)=\omega(\lambda a)=(\lambda\omega)a\quad \forall 
(\lambda,\omega,a)\in\CC \times k \times A
\]
and
\[
\omega(a_1a_2)=(\omega a_1)a_2=a_1(\omega a_2)\quad \forall (\omega,a_1,a_2)\in k\times A\times A.
\]
Denote the center of $A$ by $\Z(A)$
\[
\Z(A):=\{c\in A\mid ca=ac\,\,\forall a\in A\}
\]
$k$-algebras are not required to be unital.\\
\underline{Remark}.Let $A$ be a unital $k$-algebra. Denote the unit of $A$ by $1_A$.
$\omega\mapsto \omega1_A$ is then a unital morphism of $\CC$-algebras $k\to \Z(A)$.
Thus a unital $k$-algebra is a unital $\CC$-algebra $A$ with a given unital morphism of $\CC$-algebras
\[
k\longrightarrow \Z(A).
\]
\indent Let $A,B$ be two $k$-algebras. A morphism of $k$-algebras is a morphism of $\CC$-algebras
\[
f\colon A\to B
\]
which is also a morphism of (left) $k$-modules,
\[
f(\omega a)=\omega f(a)\quad \forall (\omega,a)\in k\times A.
\]
\indent Let $A$ be a $k$-algebra. A representation of $A$ [or a (left) $A$-module] 
is a $\CC$-vector space $V$ with given morphisms of $\CC$-algebras
\[
A\longrightarrow \mathrm{Hom}_\CC(V,V) \quad\quad\quad k\longrightarrow \mathrm{Hom}_\CC(V,V)
\]
such that \setlength{\parskip}{0pt}
\begin{enumerate}
\item $k\rightarrow\mathrm{Hom}_\CC(V,V)$ is unital
\end{enumerate} 
and
\begin{enumerate}
\setcounter{enumi}{1}
\item $(\omega a)v=\omega(av)=a(\omega v)$ $\forall (\omega,a,v)\in k\times A\times V$.
\end{enumerate}
Two representations
\[
A\longrightarrow \mathrm{Hom}_\CC(V_1,V_1) \quad\quad\quad k\longrightarrow \mathrm{Hom}_\CC(V_1,V_1)
\]
and
\[
A\longrightarrow \mathrm{Hom}_\CC(V_2,V_2) \quad\quad\quad k\longrightarrow \mathrm{Hom}_\CC(V_2,V_2)
\]
are equivalent (or isomorphic) if $\exists$ an isomorphism of $\CC$ vector spaces 
$T\colon V_1\rightarrow V_2$ which intertwines the two $A$-actions
and the two $k$-actions.\\
\begin{comment}
A representation $\varphi\colon A\to \mathrm{Hom}_\CC(V,V)$ is \emph{non-degenerate} 
iff $AV=V$. i.e. for any $v\in V$, $\exists \,\,v_1,v_2,\ldots,v_r\in V$ and 
$a_1,a_2,\ldots,a_r\in A$ with 
\[
v=a_1v_1+a_2v_2+\cdots+a_rv_r.
\]
\end{comment}
\begin{comment}
\indent\underline{Notation}. For economy of notation, a representation will be denoted \\
$\varphi\colon A\to \mathrm{Hom}_\CC(V,V)$. The unital morphism of $\mathbb{C}$-algebras
\[
k\longrightarrow\mathrm{Hom}_\CC(V, V)
\] 
is understood to be included in the given structure.\\
\end{comment}
\indent A representation is irreducible if $A\to \mathrm{Hom}_\CC(V,V)$ is not the 
zero map and $\not\exists$ a sub-$\CC$-vector space $W$ of $V$ with:
\[
\{0\}\ne W\ne V
\]
and 
\[
\omega w\in W\quad\forall\;(\omega, w)\in k\times W
\]
and
\[
aw\in W\quad\forall\;(a, w)\in A \times W 
\]
Irr($A$) denotes the set of equivalence classes of irreducible representations of $A$.

\subsection{Spectrum preserving morphisms of $k$-algebras}

\underline{Definition}.  An ideal $I$ in a $k$-algebra $A$ is a \emph{$k$-ideal} if 
$\omega a\in I$ $\forall\,(\omega,a)\in k\times I$.\\
An ideal $I\subset A$ is \emph{primitive} if $\exists$ an irreducible representation
\[
\varphi\colon A\to \mathrm{Hom}_\CC(V,V)\quad\quad\quad k\longrightarrow \mathrm{Hom}_\CC(V,V)
\]
with 
\[
I=\mathrm{Kernel}(\varphi)
\]
That is,
\[
0\to I \hookrightarrow A\xrightarrow{\varphi}\mathrm{Hom}_\CC(V,V)
\]
is exact.\\
\underline{Remark}.  Any primitive ideal is a $k$-ideal. Prim($A$) denotes the set of 
primitive ideals in $A$. The map Irr($A$)$\rightarrow$ Prim($A$) which
sends an irreducible representation to its primitive ideal is a bijection if $A$ is a 
finite type $k$-algebra. Since $k$ is Noetherian, 
any $k$-ideal in a finite type $k$-algebra $A$ is itself a finite type $k$-algebra. \\

\underline{Definition}. Let $A$, $B$ be two finite type $k$-algebras, and let
$f\colon A\to B$ be a morphism of $k$-algebras. $f$ is \emph{spectrum preserving} if
\begin{enumerate}
\item Given any primitive ideal $I\subset B$, $\exists$ a unique primitive ideal 
$L\subset A$ with $L\supset f^{-1}(I)$
\end{enumerate} 
and
\begin{enumerate}
\setcounter{enumi}{1}
\item The resulting map 
\[
\mathrm{Prim}(B)\to \mathrm{Prim}(A)
\]
is a bijection.
\end{enumerate}

\underline{Definition}. Let $A$, $B$ be two finite type $k$-algebras, and let
$f\colon A\to B$ be a morphism of $k$-algebras. $f$ is \emph{spectrum preserving 
with respect to filtrations} if $\exists$ $k$-ideals  
\[
0=I_0\subset I_1\subset\cdots\subset I_{r-1}\subset I_r=A\qquad \text{in $A$} 
\]
and $k$ ideals
\[
0=L_0\subset L_1\subset\cdots\subset L_{r-1}\subset L_r=B\qquad \text{in $B$}
\]
with $f(I_j)\subset L_j$, ($j=1,2,\ldots,r$) and $I_j/I_{j-1}\to L_j/L_{j-1}$, 
($j=1,2,\ldots,r$) is spectrum preserving.

\subsection{Algebraic variation of $k$-structure}

\underline{Notation}. If $A$ is a $\CC$-algebra, $A[t, t^{-1}]$ is the $\CC$-algebra of 
Laurent polynomials in the indeterminate $t$ with coefficients
in $A$. \\

\underline{Definition}. Let $A$ be a unital $\CC$-algebra, and let 
\[
\Psi\colon k\to \Z\left(A[t,t^{-1}]\right)
\]
be a unital morphism of $\CC$-algebras. Note that $\Z\left(A[t,t^{-1}]\right) = \Z(A)[t, t^{-1}]$.
For $\zeta\in \CC^\times=\CC-\{0\}$, ev($\zeta$) denotes the ``evaluation at $\zeta$'' map:
\begin{align*}
\mathrm{ev}(\zeta)\colon A[t,t^{-1}]&\to A\\
\sum a_jt^j                         &\mapsto \sum a_j\zeta^j 
\end{align*}
Consider the composition
\[
k\overset{\Psi}{\longrightarrow} \Z\left(A[t,t^{-1}]\right)
\overset{\mathrm{ev}(\zeta)}{\longrightarrow} \Z(A).
\]
Denote the unital $k$-algebra so obtained by $A_\zeta$. The underlying $\CC$-algebra of 
$A_\zeta$ is $A$.
Assume that for all $\zeta\in\CC^\times$,  $A_\zeta$ is a finite type $k$-algebra. 
Then for $\zeta, \zeta^\prime\in\CC^\times,
A_{\zeta^\prime}$ is obtained from $A_{\zeta}$ by an \emph{algebraic variation of $k$-structure}.

\subsection{Definition and examples}

With $k$ fixed, geometric equivalence (for finite type $k$-algebras) is the equivalence 
relation generated by the two elementary moves:
\begin{itemize}
\item Spectrum preserving morphism with respect to filtrations
\item Algebraic variation of $k$-structure
\end{itemize}
Thus two finite type $k$-algebras $A, B$ are \emph{geometrically equivalent} 
if $\exists$ a finite sequence 
$A=A_0, A_1,\ldots,A_r=B$ with each $A_j$ a finite type $k$-algebra such that for 
$j=0, 1, \ldots, r-1$ one of the following three possibilities is valid:
\begin{enumerate}
\item $A_{j+1}$ is obtained from $A_j$ by an algebraic variation of $k$-structure.
\item There is a spectrum preserving morphism with respect to filtrations\\ 
$A_j\rightarrow A_{j+1}$. 
\item There is a spectrum preserving morphism with respect to filtrations \\
$A_{j+1}\rightarrow A_j$.
\end{enumerate}
To give a geometric equivalence relating $A$ and $B$, the finite sequence of elementary 
steps (including the filtrations)
must be given. Once this has been done, a bijection of the primitive ideal spaces and 
an isomorphism of periodic cyclic homology are determined.
\[
\mathrm{Prim}(A)\longleftrightarrow \mathrm{Prim}(B)\quad\quad\quad HP_*(A)\cong HP_*(B)
\]
\underline{Example 1}. 
Two unital finite type $k$-algebras $A, B$ are 
\emph{Morita equivalent}  if there is an equivalence of categories
\begin{center}
$\Big($unital left $A$-modules$\Big) \sim \Big($unital left $B$-modules$\Big)$
\end{center}
Any such equivalence of categories is implemented by a \emph{Morita context} 
i.e. by a pair of unital bimodules
\[
{}_AV_B\quad\quad {}_BW_A
\]
together with given isomorphisms of bimodules
\begin{align*}
\alpha\colon V\otimes_B W&\to A\\
\beta\colon W\otimes _A V &\to B
\end{align*}
A Morita context is required to satisfy certain conditions. These conditions imply 
that the linking algebra formed from the Morita context is a unital finite type
$k$-algebra. The linking algebra, denoted  $M_{2\times2}({}_AV_B,{}_BW_A)$, is :
\[
\begin{tikzpicture}
\node[rotate=90] (equal) at (0,0) {$=$};
\draw (equal) node[above=.2cm] {$M_{2\times2}({}_AV_B,{}_BW_A)$};
\draw (-.4,-.6) +(-.4,-.4) rectangle ++(.4,.4);
\draw (.4,-.6) +(-.4,-.4) rectangle ++(.4,.4);
\draw (-.4,-1.4) +(-.4,-.4) rectangle ++(.4,.4);
\draw (.4,-1.4) +(-.4,-.4) rectangle ++(.4,.4);
\draw (-.4,-.6) node{$A$};
\draw (.4,-.6) node{$V$};
\draw (-.4,-1.4) node{$W$};
\draw (.4,-1.4) node{$B$};
\end{tikzpicture}
\]
The inclusions
\begin{center}
\begin{tikzpicture}[scale=.8]
\draw (1.25,.6) node {$A\hookrightarrow M_{2\times2}({}_AV_B,{}_VW_A)\hookleftarrow B$};
\draw (-1.4,-1) node {$a\mapsto$};
\draw (-.4,-.6) +(-.4,-.4) rectangle ++(.4,.4);
\draw (.4,-.6) +(-.4,-.4) rectangle ++(.4,.4);
\draw (-.4,-1.4) +(-.4,-.4) rectangle ++(.4,.4);
\draw (.4,-1.4) +(-.4,-.4) rectangle ++(.4,.4);
\draw (-.4,-.6) node{$a$};
\draw (.4,-.6) node{$0$};
\draw (-.4,-1.4) node{$0$};
\draw (.4,-1.4) node{$0$};
\begin{scope}[xshift=2.5cm]
\draw (1.4,-1) node {\reflectbox{$\mapsto$}\,\,$b$};
\draw (-.4,-.6) +(-.4,-.4) rectangle ++(.4,.4);
\draw (.4,-.6) +(-.4,-.4) rectangle ++(.4,.4);
\draw (-.4,-1.4) +(-.4,-.4) rectangle ++(.4,.4);
\draw (.4,-1.4) +(-.4,-.4) rectangle ++(.4,.4);
\draw (-.4,-.6) node{$0$};
\draw (.4,-.6) node{$0$};
\draw (-.4,-1.4) node{$0$};
\draw (.4,-1.4) node{$b$};
\end{scope} 
\end{tikzpicture}
\end{center}
are spectrum preserving morphisms of finite type $k$-algebras. Hence $A$ and $B$ are 
geometrically equivalent.\\
If Prim($A$) and Prim($B$) are given the Jacobson topology, then the bijection 
Prim($A$)$\leftrightarrow$Prim($B$) determined by a Morita equivalence is a homeomorphism.\\

\noindent \underline{Example 2}. 
Let $X$ be a complex affine variety, and let $Y$ be a sub-variety. 
$k = \cO(X)$ is the coordinate algebra of $X$, and
$\cO(Y)$ is the coordinate algebra of $Y$. $\cI_Y$ is the ideal in $\cO(X)$:\\
\[
\cI_Y=\left\{\omega\in\cO(X)\mid \omega(p)=0\,\,\forall p\in Y\right\}\\
\] Set :

\vspace{.7cm}
\begin{tikzpicture}[scale=1.2]
\draw (-1.3,-1) node {$A=$};
\draw (3,-1) node {$B=\cO(X)\oplus \cO(Y)$};
\draw (-.4,-.6) +(-.4,-.4) rectangle ++(.4,.4);
\draw (.4,-.6) +(-.4,-.4) rectangle ++(.4,.4);
\draw (-.4,-1.4) +(-.4,-.4) rectangle ++(.4,.4);
\draw (.4,-1.4) +(-.4,-.4) rectangle ++(.4,.4);
\draw (-.4,-.6) node{$\cO(\!X\!)$};
\draw (.4,-.6) node{$\cI_Y$};
\draw (-.4,-1.4) node{$\cI_Y$};
\draw (.4,-1.4) node{$\cO(\!X\!)$};
\end{tikzpicture}\vspace{4mm} 

\noindent Thus $A$ is the sub-algebra of the $2\times2$ matrices with entries in $\cO(X)$ 
consisting of all those $2\times2$ matrices whose off-diagonal entries
are in $\cI_Y$. $B=\cO(X)\oplus \cO(Y)$ is the direct sum --- as an algebra --- 
of $\cO(X)$ and $\cO(Y)$. Both $A$ and $B$ are unital finite type
$k$-algebras. $\Irr(A) = \Prim(A)$ is $X$ with each point of $Y$ doubled. 
$\Irr(B)= \Prim(B)$ is the disjoint union of $X$ and $Y$. Equipped with the Jacobson
topology, $\Prim(A)$ and $\Prim(B)$ are not homeomorphic so $A$ and $B$ are not Morita 
equivalent. However, $A$ and $B$ are geometrically equivalent.\\

\noindent \underline{Example 3}. 
Let $G$ be a connected reductive complex Lie group with maximal 
torus $T$. $\mathcal{W}$ denotes the Weyl group
\[
\mathcal{W} = N_G(T)/T
\]
and $X^*(T)$ is the character group of $T$. 
The semi-direct 
product $X^*(T)\rtimes\mathcal{W}$ is an affine Weyl group. 
In particular, it is a Coxeter group. We fix a system of
generators $S\subset X^*(T)\rtimes\mathcal{W}$ such that 
$(X^*(T)\rtimes\mathcal{W},S)$ is a
Coxeter system. Let $\ell$ denote the resulting length function on
$X^*(T)\rtimes\mathcal{W}$.

For each non-zero complex number $q$, there is the 
affine Hecke algebra $\mathcal{H}_q (G)$. 
This is an affine Hecke algebra with equal  
parameters and $\mathcal{H}_1 (G)$ is the group algebra of the affine Weyl group:
\[
\mathcal{H}_1 (G) = \mathbb{C}[X^*(T)\rtimes\mathcal{W}i,]
\]
\ie
$\cH_q$ is the algebra generated by $T_x$, $x\in X^*(T)\rtimes\mathcal{W}$, 
with relations
\begin{eqnarray} \label{eq:defHA}
T_xT_y=T_{xy}, \quad \text{if $\ell(xy)=\ell(x)+\ell(y)$, and}\cr
(T_s-\beta)(T_s+1)=0, \quad \text{if $s\in S$.}\end{eqnarray}

Using the action of $\mathcal{W}$ on $T$, form the quotient variety $T/\mathcal{W}$ 
and let $k$ be its coordinate algebra,
\[
k = \mathcal{O}(T/\mathcal{W})
\]
For all $q\in\mathbb{C}^{\times}$, $\mathcal{H}_q (G)$ is a unital finite type 
$k$-algebra. Let $\cJ$ be Lusztig's asymptotic algebra. As a
$\Cset$-vector space, $\cJ$ has a basis
$\{t_w\,:\,w\in X^*(T)\rtimes \cW\}$, and there is a canonical structure of
associative $\Cset$-algebra on $\cJ$ (see for instance
\cite[\S~8]{LuAst}).
 
Except for $q$ in a finite set of roots of unity (none of which is $1$) Lusztig 
constructs a morphism of $k$-algebras
\[\phi_q\colon
\mathcal{H}_q(G)\longrightarrow \cJ
\] 
which is spectrum preserving with respect to filtrations (see 
\cite[Theorem~9]{BN}, itself based on \cite{LuAst}). 

Let $C_w$, $w\in X^*(T)\rtimes \cW$ denote the Kazhdan-Lusztig basis of $\cH_q(G)$.
For $w$, $w'$, $w''$ in $X^*(T)\rtimes \cW$, define $h_{w,w',w''}\in A$ by
\[C_w\cdot C_{w'}=\sum_{w''\in X^*(T)\rtimes \cW}h_{w,w',w''}\,C_{w''}.\]
There is a unique function $a\colon X^*(T)\rtimes \cW\to\Nset$ such that for any
$w''\in X^*(T)\rtimes \cW$, $v^{a(w'')}h_{w,w',w''}$ is a polynomial in
$v$ for all $w$, $w'$ in $X^*(T)\rtimes \cW$ and it has non-zero constant
term for some $w$, $w'$.

Let $M$ be a simple $\cH_q(G)$-module (resp. $\cJ$-module).
Lusztig attaches to $M$ (see \cite[p.~82]{LuAst}) an integer
$a=a_M$ by the following two requirements:
\[\begin{matrix}
C_wM=0\;\text{(resp. $t_w M=0$)}&\text{for all $w\in  X^*(T)\rtimes \cW$ 
such that
$a(w)>a$};\cr
C_wM\ne 0\;\text{(resp. $t_w M\ne 0$)}&\text{ for some $w\in
X^*(T)\rtimes \cW$ such
that $a(w)=a$.}\end{matrix}\]
The map $\phi_q$ is the unique bijection \[M\mapsto
M'\] between the primitive ideal spaces of $\cH_q(G)$ and
$\cJ$, with the following properties:
$a_M=a_{M'}$ and the restriction of $M'$ to $\cH_q(G)$ via
$\phi_q$ is an $\cH_q(G)$-module with exactly one
composition factor isomorphic to $M$ and all other composition factors of
the form $M''$ with $a_{M''} > a_M$ (see \cite[Theorem~8.1]{LuAst}).

The algebra $\mathcal{H}_q(G)$ 
is viewed as a $k$-algebra via the canonical isomorphism
\[
\mathcal{O}(T/\mathcal{W})\cong \Z(\mathcal{H}_q(G))
\]
Lusztig's map $\phi_q$ maps $\Z(\mathcal{H}_q(G))$ to $\Z(\cJ)$ 
and thus determines a unique $k$-structure for $\cJ$ such that the map
$\phi_q$ is a morphism of $k$-algebras. $\cJ$ with this $k$-structure 
will be denoted $\cJ_q$. $\mathcal{H}_q (G)$ is then geometrically
equivalent to $\mathcal{H}_1 (G)$ by the three elementary steps
\[
\mathcal{H}_q (G) \to \cJ_q \to \cJ_1 \to \mathcal{H}_1 (G) .
\]
The second elementary step (i.e. passing from $\cJ_q$ to $\cJ_1$) is an algebraic 
variation of $k$-structure. Hence (provided $q$ is not in the exceptional set
of roots of unity) $\mathcal{H}_q (G)$ is geometrically equivalent to $\mathcal{H}_1 (G)
= \mathbb{C}[X^*(T)\rtimes\mathcal{W}]$. 

\begin{cor} \label{cor:J}
There is a canonical bijection
\[(T/\!/\mathcal{W})_2\longleftrightarrow \Irr(\mathcal{H}_q (G))
\]
\end{cor}
This map gives the left slanted arrow in Theorems \ref{thm:bijphi} and \ref{thm:ps}.\\

\noindent \underline{Example 4.}
Let $\cH_{\boldq}(X \rtimes W)$ be the affine Hecke algebra
of $X \rtimes W$ with unequal parameters $\boldq = \{q_1,\ldots, q_k\}$.
We assume that $q_i\in\Rset_{>0}$. Let $\cS_\boldq (X \rtimes W)$
be the Schwartz completion of $\cS_\boldq (X \rtimes W)$, as in \cite[\S 5.4]{Sol}.
In this setting \cite[Lemma~5.3.2]{Sol} gives a morphism of Fr\'echet algebras
\[
\cS_1 (X \rtimes W) \to \cS_\boldq (X \rtimes W) ,
\]
which is spectrum preserving with respect to filtrations.
However, the existence of a geometric equivalence between the $\cO(T / W)$-algebras
$\cO(T) \rtimes W$ and  $\cH_\boldq(T \rtimes W^\fs)$ is still an open question
in case $\boldq$ contains unequal parameters $q_i$.

\end{document}